\documentclass[11pt]{amsart} 
\usepackage{fullpage, amsmath, amssymb, verbatim}
\usepackage[all]{xypic}

\title{Coinvariants of Lie algebras of vector fields on algebraic
varieties
} 
\usepackage{url} \usepackage{graphicx}

\numberwithin{equation}{section}

\theoremstyle{definition}
\newtheorem{theorem}[equation]{Theorem}
\newtheorem{lemma}[equation]{Lemma}
\newtheorem{proposition}[equation]{Proposition}
\newtheorem{corollary}[equation]{Corollary}

\newtheorem{definition}[equation]{Definition}
\newtheorem{notation}[equation]{Notation}
\newtheorem{example}[equation]{Example}
\newtheorem{question}[equation]{Question}

\newtheorem{remark}[equation]{Remark}

\newcommand{\an}{\operatorname{an}}
\newcommand{\PDiff}{\operatorname{PDiff}}
\newcommand{\IC}{\operatorname{IC}}
\newcommand{\IH}{\operatorname{IH}}
\newcommand{\red}{\text{red}}
\newcommand{\even}{\text{even}}

\newcommand{\Sh}{\operatorname{Sh}}
\newcommand{\codim}{\operatorname{codim}}
\newcommand{\Eu}{\operatorname{Eu}}
\newcommand{\Par}{\operatorname{Par}}
\newcommand{\Parpt}{\operatorname{Parpt}}
\newcommand{\Vect}{\operatorname{Vect}}

\newcommand{\vol}{\mathsf{vol}}

\newcommand{\bk}{\mathbf{k}}
\newcommand{\Id}{\operatorname{Id}}

\newcommand{\Stab}{\operatorname{Stab}}
\newcommand{\symm}{\text{symm}}

\newcommand{\pt}{\text{pt}}
\newcommand{\HP}{\mathsf{HP}}

\newcommand{\bZ}{\mathbf{Z}}

\newcommand{\bP}{\mathbf{P}}
\newcommand{\iso}{{\;\stackrel{_\sim}{\to}\;}}
\newcommand{\bC}{\mathbf{C}}
\newcommand{\bG}{\mathbf{G}}

\newcommand{\caH}{\mathcal{H}}
\newcommand{\cLH}{\mathcal{LH}}

\newcommand{\cO}{\mathcal{O}}

\newcommand{\bH}{\mathbf{H}}
\newcommand{\caD}{\mathcal{D}}
\newcommand{\cT}{\mathcal{T}}

\newcommand{\mfv}{\mathfrak{v}}

\newcommand{\Hom}{\operatorname{Hom}}
\newcommand{\sign}{\operatorname{sign}}
\newcommand{\ctr}{\operatorname{ctr}}
\newcommand{\tpl}{\operatorname{top}}
\newcommand{\End}{\operatorname{End}}

\newcommand{\Ind}{\operatorname{Ind}}

\newcommand{\Ext}{\operatorname{Ext}}
\newcommand{\pr}{\operatorname{pr}}

\newcommand{\mfg}{\mathfrak{g}}

\newcommand{\SL}{\mathsf{SL}}
\newcommand{\GL}{\mathsf{GL}}
\newcommand{\Sp}{\mathsf{Sp}}

\newcommand{\gr}{\operatorname{\mathsf{gr}}}
\newcommand{\Spec}{\operatorname{\mathsf{Spec}}}

\newcommand{\onto}{\twoheadrightarrow}
\newcommand{\into}{\hookrightarrow}
\newcommand{\Sym}{\operatorname{\mathsf{Sym}}}

\newcommand{\bA}{\mathbf{A}}

\begin{document}
\date{2012}

\author{Pavel Etingof}
\address{MIT Dept of Math, Room 2-176, 77 Massachusetts Ave, Cambridge, MA 02139, USA}
\email{etingof@math.mit.edu}

\author{Travis Schedler}
\address{U. Texas at Austin, Math Dept, RLM 8.100, 2515 Speedway Stop C1200, Austin, TX 78712, USA}
\email{trasched@gmail.com}



\begin{abstract}
  We prove that the space of coinvariants of functions on an affine
  variety by a Lie algebra of vector fields whose flow generates
  finitely many leaves is finite-dimensional.
  Cases of the theorem include Poisson (or more generally Jacobi)
  varieties with finitely many symplectic leaves under Hamiltonian
  flow, complete intersections in Calabi-Yau varieties with isolated
  singularities under the flow of incompressible vector fields,
  quotients of Calabi-Yau varieties by finite volume-preserving groups
  under the incompressible vector fields, and arbitrary varieties with
  isolated singularities under the flow of all vector fields.  We
  compute this quotient explicitly in many of these cases. The proofs
  involve constructing a natural $\caD$-module representing the
  invariants under the flow of the vector fields, which we prove is
  holonomic if it has finitely many leaves (and whose holonomicity we
  study in more detail).  We give many counterexamples to naive
  generalizations of our results. These examples have been a source of
  motivation for us.
\end{abstract}
\maketitle
\tableofcontents
\section{Introduction}
\subsection{Vector fields on affine schemes}
Let $\bk$ be an algebraically closed field of characteristic zero, and
let $X = \Spec \cO_X$ be an affine scheme of finite type over $\bk$
(we will generalize this to nonaffine schemes in \S \ref{ss:nonaffine}
below).  Our examples will be varieties, so the reader interested only
in these (rather than the general theory, which profits from
restriction to nonreduced subschemes) can freely make this assumption.

We also remark that our results can be generalized to the analytic
setting using the theory of analytic $\caD$-modules, except that in
these cases, the coinvariants need no longer be finite-dimensional,
since analytic varieties can have infinite-dimensional cohomology in
general (e.g., a surface with infinitely many punctures).  But we will
not discuss this here.

When we say $x \in X$, we mean a closed point, which is the same as a
point of the reduced subvariety $X_{\red}$.  Note that (since $\bk$
has characteristic zero) it is well-known that all vector fields on
$X$ (which by definition means derivations of $\cO_X$) are parallel to
$X_\red$ (dating to at least \cite[Theorem 1]{Sei-dirfgt}).  Hence,
there is a restriction map $\Vect(X) \to \Vect(X_{\red})$, although
this is not an isomorphism unless $X=X_\red$.  In particular, for all
global vector fields $\xi \in \Vect(X)$ and all $x \in X$, $\xi|_x \in
T_x X_\red$.

Let $\mfv \subseteq \Vect(X)$ be a Lie subalgebra of the Lie algebra
of vector fields (which is allowed to be all vector fields).  We are
interested in the coinvariant space,
\[
(\cO_X)_{\mfv} := \cO_X / \mfv(\cO_X).
\]
(We remark that one could more generally consider an arbitrary set of
vector fields, but the coinvariants coincide with those of the Lie
algebra generated by that set. One could also consider more generally
differential operators of order $\leq 1$: see Remark
\ref{r:diffop-leq1}.)

Our main results show that, under nice geometric conditions, this
coinvariant space is finite-dimensional, and in fact that the
corresponding $\caD$-module generated by $\mfv$ is holonomic. This
specializes to the finite-dimensionality theorems \cite[Theorem
4]{BEG} and \cite[Theorem 3.1]{ESdm} in the case of Poisson
varieties. It also generalizes a standard result about coinvariants
under the action of a reductive algebraic group (see Remarks
\ref{r:red} and \ref{r:nonfinite} below).

Our first main result can be stated as follows.
\begin{theorem}\label{t:main1}
  Suppose that, for all $i \geq 0$, the locus of $x \in X$ where the
  evaluation $\mfv|_x$ has dimension $\leq i$ has dimension
  at most $i$. Then the coinvariant space $(\cO_X)_{\mfv}$ is
  finite-dimensional.
\end{theorem}
In particular, the hypothesis implies that, on an open dense
subvariety of $X_\red$, $\mfv$ generates the tangent bundle; as we will explain
below, the hypothesis is equivalent to the statement that $X_\red$ is
stratified by locally closed subvarieties with this property.

\subsection{Goals and outline of the paper}
First, in \S \ref{s:gen}, we reformulate Theorem \ref{t:main1}
geometrically and prove it, along with more general
finite-dimensionality and holonomicity theorems.  The main tool
involves the definition of a right $\caD$-module, $M(X,\mfv)$,
generalizing \cite{ESdm}, such that $\Hom(M(X,\mfv), N) \cong N^\mfv$
for all $\caD$-modules $N$, i.e., the $\caD$-module which represents
invariants under the flow of $\mfv$.  Then the theorem above is proved
by studying when this $\caD$-module is holonomic.

The next goal, in \S \ref{s:Csla}, is to study examples related to
Cartan's classification of simple infinite-dimensional transitive Lie
algebras of vector fields on a formal polydisc which are complete with
respect to the jet filtration.  Namely, according to Cartan's
classification \cite{Car-gtcis,GQS-ic}, there are four such Lie
algebras, as follows. For $\xi \in \Vect(X)$, let $L_\xi$ denote the
Lie derivative by $\xi$.  Let $\hat \bA^n$ be the formal neighborhood
of the origin in $\bA^n$, which is a formal polydisc of dimension n.
Then, Cartan's classification consists of:
\begin{enumerate}
\item[(a)] The Lie algebra $\Vect(\hat \bA^n)$ of all vector fields on
  $\hat \bA^n$;
\item[(b)] The Lie algebra $H(\hat \bA^{2n},\omega)$ of all
  Hamiltonian vector fields on $\hat \bA^{2n}$, i.e., preserving the
  standard symplectic form $\omega = \sum_i dx_i \wedge dy_i$;
  explicitly, $\xi$ such that $L_\xi \omega = 0$;
\item[(c)] The Lie algebra $H(\hat \bA^{2n+1},\alpha)$ of all contact
  vector fields on an odd-dimensional formal polydisc, with respect to
  the standard contact structure $\alpha = dt + \sum_i x_i dy_i$,
  i.e., those vector fields satisfying $L_\xi \alpha \in \cO_X \cdot
  \alpha$;
\item[(d)] The Lie algebra $H(\hat \bA^n,\vol)$ of all
  volume-preserving vector fields on $\hat \bA^n$ equipped with the
  standard volume form $\vol = dx_1 \wedge \cdots \wedge dx_n$, i.e.,
  vector fields $\xi$ such that $L_\xi \vol = 0$.
\end{enumerate}
In \S \ref{s:Csla}, we define generalizations of each of these
examples to the global (but still affine), singular, degenerate
situation.  For example, (a) becomes vector fields on arbitrary
schemes of finite type. For (b)--(d), we define generalizations of the
structure on the variety, which in case (b) yields Poisson
varieties. Then, there are essentially two different choices of the
Lie algebra of vector fields. In case (b), these are Hamiltonian
vector fields or Poisson vector fields. We recall that Hamiltonian
vector fields are of the form $\{f,-\}$ for $f \in \cO_X$, and Poisson
vector fields are all vector fields which preserve the Poisson
bracket, i.e., such that $\xi\{f,g\} = \{\xi(f),g\}+\{f,\xi(g)\}$;
this includes all Hamiltonian vector fields.

In each of the cases (a)--(d), we study the leaves under the flow of
$\mfv$ and the condition for the associated $\caD$-module to be
holonomic (and hence for $(\cO_X)_\mfv$ to be finite-dimensional).

In \S \ref{s:ex-dloc} we discuss the globalization of these examples
to the nonaffine setting, which turns out to be straightforward for
Hamiltonian vector fields and all vector fields, but quite nontrivial
for Poisson vector fields (and hence their generalizations). We do not
need this material for the remainder of the paper.

In the remainder of the paper we study in detail three specific
examples for which the $\caD$-module has an interesting and nontrivial
structure which reflects the geometry. In these examples, we explicitly
compute the $\caD$-module and the coinvariants $(\cO_X)_\mfv$.

In \S \ref{s:cy-cplte-int-is}, we consider the case of divergence-free
vector fields on complete intersections in Calabi-Yau
varieties. Holonomicity turns out to be equivalent to having isolated
singularities, and we restrict to this case. Then, the structure of
the $\caD$-module and the coinvariant functions $(\cO_X)_{\mfv}$ is
governed by the Milnor number and link of the isolated singularities.
This is because there is a close relationship between these
coinvariants and the de Rham cohomology of the formal (or analytic)
neighborhood of the singularity, which was studied in, e.g.,
\cite{Gre-GMZ, HLY-plhdRc, Yau-vniis} (which were a source of
motivation for us).

In \S \ref{s:fqcyvar}, we consider quotients of Calabi-Yau varieties
by finite groups of volume-preserving automorphisms. In this case, it
turns out that the $\caD$-module associated to volume-preserving
vector fields is governed by the most singular points, where the
stabilizer is larger than that of any point in some neighborhood.
More generally, rather than working on the quotient $X/G$ where $X$ is
Calabi-Yau and $G$ is a group of volume-preserving automorphisms, we
study the Lie algebra of $G$-invariant volume-preserving vector fields
on $X$ itself.

Finally, in \S \ref{s:sym}, we consider symmetric powers $(S^n X,
\mfv)$ of smooth varieties $(X,\mfv)$ on which $\mfv$ generates the
tangent space everywhere (which we call \emph{transitive}). This
includes the symplectic, locally conformally symplectic, contact, and
Calabi-Yau cases.  In these situations, we explicitly compute the
$\caD$-module and the coinvariant functions.  Dually, the main result
says that the invariant functionals on $\cO_{S^n X}$ form a polynomial
algebra whose generators are the functionals on diagonal embeddings
$X^i \to S^n X$ obtained by pulling back to $X^i$ and taking a
products of invariant functionals on each factor of $X$. For the
$\caD$-module, this expresses $M(S^n X, \mfv)$ as a direct sum of
external tensor products of copies of $M(X,\mfv)$ along each diagonal
embedding.
   

\subsection{Acknowledgements}
The first author's work was
partially supported by the NSF grant DMS-1000113. The second author is
a five-year fellow of the American Institute of Mathematics, and was
partially supported by the ARRA-funded NSF grant DMS-0900233.

\section{General theory}\label{s:gen}
Let $\Omega^\bullet_X := \wedge_{\cO_X}^\bullet \Omega^1_X$ be the
algebraic de Rham complex, where $\Omega^1_X$ is the sheaf of K\"ahler
differentials on $X$.  We will frequently use the de Rham complex
modulo torsion, $\tilde \Omega^\bullet_X := \Omega^\bullet_X /
\text{torsion}$.

By \emph{polyvector fields} of degree $m$ on $X$, we mean
skew-symmetric multiderivations $\wedge^m_{\bk} \cO_X \to \cO_X$. Let
$T^m_X$ be the sheaf of such multiderivations. Equivalently, $T^m_X =
\Hom_{\cO_X}(\Omega^m_X, \cO_X)$, where $\xi \in T^m_X$ is identified
with the homomorphism sending $df_1 \wedge \cdots \wedge df_m$ to
$\xi(f_1 \wedge \cdots \wedge f_m)$. This also coincides with
$\Hom_{\cO_X}(\tilde \Omega^m_X, \cO_X)$.

When $X$ is smooth, then $\tilde \Omega^\bullet_X = \Omega^\bullet_X$,
and its hypercohomology (which, for $X$ affine, is the same as the
cohomology of its complex of global sections) is called the algebraic
de Rham cohomology of $X$.  Over $\bk=\bC$, this cohomology coincides
with the topological cohomology of $X$ under the complex topology, by
a well-known theorem of Grothendieck.  For arbitrary $X$, we will
denote the cohomology of the space of global sections, $\Gamma(\tilde
\Omega^\bullet_X)$, by $H_{DR}^\bullet(X)$, and the hypercohomology of
the complex of sheaves $\tilde \Omega^\bullet_X$ by
$\bH_{DR}^\bullet(X)$ (very often we will use these when $X$ is smooth
and affine, where they both coincide with topological cohomology).


We caution that, when $X$ is smooth, $\Omega_X$ (without a
superscript) will denote the canonical right $\caD_X$-module of volume
forms, which as a $\cO_X$-module coincides with $\Omega_X^{\dim X}$
under the above definition, when $X$ has pure dimension.

By a \emph{local system} on a variety, we mean an $\cO$-coherent right
$\caD$-module on the variety. Moreover, from now on, when we say
$\caD$-module, we always will mean a right $\caD$-module.

\subsection{Reformulation of Theorem \ref{t:main1} in terms of leaves}
Recall that $(X, \mfv)$ is a pair of an affine scheme $X$ of finite
type and a Lie algebra $\mfv \subseteq \Vect(X)$ of vector fields on
$X$.  We will give a more geometric formulation of Theorem
\ref{t:main1} in terms of \emph{leaves} of $X$ under $\mfv$, followed
by a strengthened version in these terms.
\begin{definition}
  An \emph{invariant subscheme} is a locally closed subscheme $Z
  \subseteq X$ preserved by $\mfv$; set-theoretically, this says that,
  at every point $z \in Z$, the evaluation $\mfv|_z$ lies in the
  tangent space $T_z Z_\red$.  A \emph{leaf} is a connected invariant
  (reduced) subvariety $Z$ such that, at every point $z$, in fact
  $\mfv|_z = T_z Z$.  A \emph{degenerate invariant subscheme} is an
  invariant subscheme $Z$ such that, at every point $z \in Z$,
  $\mfv|_z \subsetneq T_z Z_\red$.
\end{definition}
When an invariant subscheme is reduced, we call it an invariant
subvariety. An invariant subscheme $Z \subseteq X$ is degenerate if
and only if the invariant subvariety $Z_\red$ is degenerate.  Note
that the closure of any degenerate invariant subscheme is also such.
Also, leaves are necessarily smooth. Although the same is clearly not
true of degenerate invariant subschemes, we can restrict our attention
to those with smooth reduction by first stratifying $X_\red$ by its
(set-theoretic) singular loci, in view of the classical result:
\begin{theorem}\cite[Corollary to Theorem
  12]{Sei-dirfgt}\label{t:sing-pres}
  The set-theoretic singular locus of $X_\red$ is preserved by all
  vector fields on $X$.
\end{theorem}
We give a proof of a more general assertion in the proof of
Proposition \ref{p:decomp} below.
\begin{remark}\label{r:sch-sing}
  Note that, for the set-theoretic singular locus to be preserved by
  all vector fields, we need to use that the characteristic of $\bk$
  is zero; otherwise the singular locus is not preserved by all vector
  fields: e.g., in characteristic $p > 0$, one has the derivation
  $\partial_x$ of $\bk[x,y]/(y^2-x^p)$, which does not vanish at the
  singular point at the origin.

  On the other hand, in arbitrary characteristic, the scheme-theoretic
  singular locus of a variety of pure dimension $k \geq 0$ is
  preserved, where we define this by the Jacobian ideal: for a variety
  cut out by equations $f_i$ in affine space, this is the ideal
  generated by determinants of $(k \times k)$-minors of the Jacobian
  matrix $(\frac{\partial f_i}{\partial x_j})$ (this is preserved by
  \cite{Har-dcr}, where it is shown that it coincides with the
  smallest nonzero Fitting ideal of the module of K\"ahler
  differentials).  In the above example it would be defined by the
  ideal $(y)$ when $p > 2$.  This is evidently preserved by all vector
  fields, which are all multiples of $\partial_x$.  Note, however,
  that we will not make use of the scheme-theoretic singular locus in
  this paper (except in \S \ref{s:cy-cplte-int-is}, where we will
  explicitly define it), nor will we consider the case of positive
  characteristic.
\end{remark}
\begin{definition} Say that $(X, \mfv)$ has \emph{finitely many
    leaves} if $X_\red$ is a (disjoint) union of finitely many leaves.
\end{definition}
For example, when $X$ is a Poisson variety and $\mfv$ is the Lie algebra of
Hamiltonian vector fields, then this condition says that $X$ has
finitely many symplectic leaves.

We caution that, when $(X, \mfv)$ does not have finitely many leaves,
it does \emph{not} follow that there are infinitely many algebraic
leaves, or any at all:
\begin{example}\label{ex:noleaves}
  Consider the two-dimensional torus $X = (\bA^1\setminus \{0\})^2$, and
  let $\mfv = \langle \xi \rangle$ for some global vector field $\xi$
  which is not algebraically integrable, e.g., $x\partial_x - c
  y\partial_y$ where $c$ is irrational. The analytic leaves of this
  are the level sets of $x^cy$, which are not algebraic.  There are in
  fact \emph{no} algebraic leaves at all.

  However, it is always true that, in the formal neighborhood $\hat
  X_x$ of every point $x \in X$, there exists a formal leaf of $X$
  through $x$: this is the orbit of the formal group obtained by
  integrating $\mfv$.  In the above example, this says that the level
  sets of $x^cy$ do make sense in the formal neighborhood of every
  point $(x,y) \in X$.
\end{example}

The condition of having finitely many leaves is well-behaved:
\begin{proposition}\label{p:decomp}
  Let $X_i := \{x \in X \mid \dim \mfv|_x = i\} \subseteq
  X_\red$. Then $X_i$ is an invariant
  locally closed subvariety.  If $X$ has finitely many leaves, then
  the connected components of the $X_i$ are all leaves.
\end{proposition}
We prove this below. We first note the consequence:
\begin{corollary}\label{c:finleaves}
  There can be at most one decomposition of $X_\red$ into finitely many
  leaves.  The following are equivalent:
\begin{itemize}
\item[(i)] $X$ has finitely many leaves;
\item[(ii)] $X$ contains no  degenerate invariant subvariety;
\item[(iii)] For all $i$, the dimension of $X_i$ is at most $i$.
\end{itemize}
\end{corollary}
\begin{proof}
  For the first statement, suppose that $X = \sqcup_i Z_i = \sqcup_i
  Z'_i$ are two decompositions into leaves.  Then each nonempty
  pairwise intersection $Z_i \cap Z_j'$ is evidently a leaf. Now, for
  each $i$, $Z_i = \sqcup_j (Z_i \cap Z_j')$ is a decomposition of
  $Z_i$ as a disjoint union of locally closed subvarieties of the same
  dimension as $Z_i$.  Since $Z_i$ is connected, this implies that
  this decomposition is trivial, i.e., $Z_j' = Z_i$ for some $j$.

  For the equivalence, first we show that (i) implies (ii). Indeed, if
  $X$ were a union of finitely many leaves and also $X$ contained a
  degenerate invariant subvariety $Z$, we could assume $Z$ is
  irreducible.  Then there would be some $X_i$ such that $X_i \cap Z$
  is open and dense in $Z$.  But then the rank of $\mfv$ along $X_i
  \cap Z$ would be less than the dimension of $Z$, and hence less than
  the dimension of $X_i$, a contradiction.  To show (ii) implies
  (iii), note that, if $\dim X_i > i$, then any open subset of $X_i$
  of pure maximal dimension is degenerate. To show (iii) implies (i),
  note that the decomposition of Proposition \ref{p:decomp} must be
  into leaves if $\dim X_i \leq i$ for all $i$ (in fact, in this case,
  each $X_i$ is a (possibly empty) finite union of leaves of dimension $i$).
\end{proof}
\begin{proof}[Proof of Proposition \ref{p:decomp}]
  First, to see that the $X_i$ are locally closed, it suffices to show
  that $Y_j := \bigsqcup_{i \leq j} X_i$ is closed for all $j$.  This
  statement would be clear if $\mfv$ were finite-dimensional; for
  general $\mfv$ we can write $\mfv$ as a union of its
  finite-dimensional subspaces, and $Y_j(\mfv)$ is the intersection of
  $Y_j(\mfv')$ over all finite-dimensional subspaces $\mfv' \subseteq
  \mfv$.

  Next, we claim that, for all $i \leq k$, the subvariety
  $X_{i,k} \subseteq X$ of points $x \in X_i$ at which $\dim T_x X = k$
  is preserved by all vector fields from $\mfv$.
  
  Let $S := \Spec \bk[\![t]\!]$ and $X_S := \Spec \cO_X[\![t]\!]$.
  For every $\xi \in \mfv$, consider the automorphism $e^{t \xi}$ of
  $\cO_{X_S}$.  For any point $x \in X_{i,k}$, consider the
  corresponding $S$-point $x_S \in X_S$, i.e., $\cO_S$-linear
  homomorphism $\cO_{X_S} \to \cO_S$.  Let $\mathfrak{m} =
  \mathfrak{m}_{x_S}$ be its kernel, i.e., $\mathfrak{m}_x[\![t]\!]$.
  Then, let $\tilde x_S = e^{t \xi} x_S$, another $S$-point of $X_S$,
  and let $\tilde {\mathfrak{m}} = \mathfrak{m}_{x_S}$ be the kernel
  of its associated homomorphism $\cO_{X_S} \to \cO_S$.

Let the cotangent space to $X_S$ at $x_S$ be defined as $T^*_{x_S} X_S = 
\mathfrak{m} / \mathfrak{m}^2$, and similarly $T^*_{x_S'} X_S =
\tilde {\mathfrak{m}}/\tilde {\mathfrak{m}}^2$.  Since $T^*_{x_S} X_S$ is a free
$\cO_S$-module of rank $k$, the same holds for $T^*_{\tilde x_S} X_S$. 

Moreover, we can view $\mfv[\![t]\!]$ as a space of vector fields on
$X_S$ over $S$, i.e., as a subspace of $\cO_S$-derivations $\cO_{X_S}
\to \cO_S$. Since $e^{t\xi}$ is an automorphism preserving
$\mfv[\![t]\!]$, it follows as for $x_S \in X_S$ that the image of
$\mfv[\![t]\!] \to \ \Hom_{\cO_S}(T^*_{\tilde x_S} X_S, \cO_S)$ is a
free $\cO_S$-module of rank $i$.  We conclude that $\tilde x_S \in
(X_{i,k})_S = \Spec \cO_{X_{i,k}}[\![t]\!] \subseteq X_S$.

We conclude from the preceding paragraphs that $\mfv$ is parallel to
$X_{i,k}$, as desired.


This can also be used to prove Theorem \ref{t:sing-pres}: setting $i =
k = \dim X + 1$, we conclude that the intersection of the
(set-theoretic) singular locus with the union of irreducible
components of $X$ of top dimension is preserved by all vector fields;
one can then induct on dimension. Alternatively, one can apply the
above argument, replacing $X_{i,k}$ by the set-theoretic singular
locus of $X$.

  For the final statement of the proposition, note that, if $X$ has a
  degenerate invariant subvariety $Z \subseteq X$, then it cannot be a
  union of finitely many leaves, since one of them would have to be
  open in $Z$, which is impossible.
\end{proof}
\begin{remark}
  In the case that $\bk=\bC$, we could prove the proposition by
  embedding $X$ into $\bC^k$ and
  locally analytically integrating the flow of vector fields of $\mfv$
  (which individually noncanonically lift to $\bC^k$), which must
  preserve the singular locus and the rank of $\mfv$.
\end{remark}
In view of Corollary \ref{c:finleaves}, Theorem \ref{t:main1} above
can be restated as:
\begin{theorem}\label{t:main1-alt} If $(X,\mfv)$ has finitely many
  leaves,
  then $(\cO_X)_{\mfv}$ is finite-dimensional.
\end{theorem}
In the aforementioned Poisson variety case, the theorem is a special case of
\cite[Theorem 1.1]{ESdm}.  Note that the converse to the theorem does not
hold: see Remark \ref{r:hol-fd}.

\begin{remark}\label{r:red}
  Suppose that $X$ is irreducible and that $\mfv$ acts locally
  finitely and semisimply on $\cO_X$, e.g., if $\mfv$ is the Lie
  algebra of a reductive algebraic group acting on $X$.  If, moreover,
  $\mfv$ acts with finitely many leaves, then Theorem
  \ref{t:main1-alt} is elementary. In fact, it is enough to assume
  that $\mfv$ has a dense leaf.  Then, $\dim (\cO_X)_{\mfv}= 1$.  This
  is because, by local finiteness and semisimplicity, the canonical
  map $(\cO_X)^{\mfv} \to (\cO_X)_{\mfv}$ is an isomorphism, and the
  former has dimension one.
\end{remark}
\begin{remark}\label{r:nonfinite}
  One can obtain examples where $\dim (\cO_X)_{\mfv} > 1$ when $\mfv$
  is semisimple and transitive,
  but does not act locally finitely.  For example, let $X \subseteq
  \bA^2$ be any nonempty open affine subvariety such that $0 \notin
  X$. Let $\mathfrak{sl}_2$ act on $X$ by the restriction of its
  action on $\bA^2$. This is the Lie algebra of linear Hamiltonian
  vector fields with respect to the usual symplectic structure on
  $\bA^2$.
  Since $X$ is affine symplectic, if $H(X)$ denotes the Hamiltonian
  vector fields, $(\cO_X)_{H(X)} \cong H^{\dim X}(X) = H_{DR}^{2}(X)$,
  by the usual isomorphism $[f] \mapsto f \cdot
  \vol_X$.\footnote{Dually, in the complex case $\bk=\bC$, the second
    homology of $X$ as a topological space produces the functionals on
    $\cO_X$ invariant under $H(X)$ (and hence also those invariant
    under $\mathfrak{sl}_2$) by $C \in H_2(X) \mapsto \Phi_C$,
    $\Phi_C(f) = \int_C f \vol_X$.}  On the other hand,
  $\mathfrak{sl}_2 \subseteq H(X)$, so $\dim (\cO_X)_{\mathfrak{sl}_2}
  \geq \dim H^2(X)$ (in fact this is an equality since
  $\mathfrak{sl}_2 \cdot \cO_X = H(X) \cdot \cO_X$ inside $\caD_X$,
  as $\mathfrak{sl}_2$ is transitive and volume-preserving, cf.~Proposition
  \ref{p:inc-os} below).  There are many examples of such varieties $X$
  which have $\dim H^2(X) > 1$. For example, if $X$ is the complement
  of $n+1$ lines through the origin, then $\dim H^2(X) = n$: the Betti
  numbers of $X$ are $1, n+1$, and $n$, since the Euler characteristic
  is zero, each deleted line creates an independent class in first
  cohomology, and there can be no cohomology in degrees higher than
  two as $X$ is a two-dimensional affine variety. This produces an
  example as desired for $n \geq 2$.
\end{remark}

\subsection{The $\caD$-module defined by $\mfv$}\label{ss:ddv}
The proof of the theorems above is based on a stronger result
concerning the $\caD$-module whose solutions are invariants under the
flow of $\mfv$.  This construction generalizes $M(X)$ from
\cite{ESdm} in the case $X$ is Poisson and $\mfv$ is the Lie algebra of
Hamiltonian vector fields. Namely, we prove that this $\caD$-module is
holonomic when $X$ has finitely many leaves. We will explain a partial
converse in \S \ref{ss:inc}, and discuss holonomicity in more detail in
\S \ref{ss:hol-bicond} below.

Let $\caD_X$ be the canonical right $\caD$-module on $X$, which is
equipped with a left action by $\Vect(X)$.  Explicitly, under
Kashiwara's equivalence, if $X \into V$ is an embedding into a smooth
affine variety $V$, this corresponds to $I_X \cdot \caD_V \setminus \caD_V$
together with the left action by derivations which preserve $I_X$.

Define
\begin{equation}
  M(X,\mfv) := \mfv \cdot \caD_X \setminus \caD_X,
\end{equation}
where $\mfv \cdot \caD_X$ is the right submodule generated by the action of
$\mfv$ on $\caD_X$.  (We will also use the same definition when $X$ is
replaced by its completion $\hat X_x$ at points $x \in X$, even though
$\hat X_x$ does not have finite type.)

Explicitly, if $i: X \to V$ is an embedding into a smooth affine
variety $V$, let $\widetilde \mfv \subseteq \Vect(V)$ be the subspace
of vector fields which are parallel to $X$ and restrict on $X$ to
elements of $\mfv$. Then, the image of $M(X,\mfv)$ under Kashiwara's
equivalence is
\[
M(X,\mfv,i) = (I_X + \widetilde \mfv) \caD_V \setminus \caD_V.
\]

Let $\pi: X \to \Spec \bk$ be the projection to a point, and $\pi_0$
the functor of underived direct image from $\caD$-modules on $X$ to
those on $\bk$, i.e., $\bk$-vector spaces.
\begin{proposition}\label{p:pushfwd}
$\pi_0 M(X,\mfv) = (\cO_X)_{\mfv}$.
\end{proposition}
\begin{proof}
  Fix an affine embedding $X \into V$. Then, the underived direct
  image is
\[
\pi_0 M(X,\mfv) = (I_X + \widetilde \mfv) \caD_V \setminus \caD_V \otimes_{\caD_V} \cO_V =
(\cO_X)_{\mfv}. \qedhere
\]
\end{proof}
If $Z \subseteq X$ is an invariant closed subscheme,
we will repeatedly use the following relationship between $M(X,\mfv)$ and
$M(Z, \mfv|_Z)$:
\begin{proposition}\label{p:inv-sub-qt}
  If $i: Z \to X$ is the tautological embedding of an invariant closed
  subscheme, then there is a canonical surjection $M(X,\mfv) \onto i_*
  M(Z, \mfv|_Z)$.
\end{proposition}
\begin{proof}
  This follows because $i_* M(Z, \mfv|_Z) = ((\mfv +I_Z)\cdot
  \caD_X)\setminus \caD_X$, where $I_Z$ is the ideal of $Z$.
\end{proof}
\begin{remark}\label{r:vs}
  As pointed out in the previous subsection, one could more generally
  allow $\mfv$ to be an arbitrary subset of
  $\Vect(X)$. However, it is easy to see that the $\caD$-module is the
  same as for the Lie algebra generated by this subset. So, no
  generality is lost by assuming that $\mfv$ be a Lie algebra.
\end{remark}
\begin{notation}\label{ntn:la}
  By a Lie algebroid in $\Vect(X)$, we mean a Lie subalgebra which is
  also a coherent subsheaf.
\end{notation}
\begin{remark}\label{r:diffop-leq1}  One could more generally (although
  equivalently in a sense we will explain) allow $\mfv \subseteq
  \caD_X^{\leq 1}$ to be a space of differential operators of order
  $\leq 1$.  One then sets, as before, $M(X,\mfv) = \mfv \cdot \caD_X
  \setminus \caD_X$.  In this case, one obtains the same $\caD$-module
  not merely by passing to the Lie algebra generated by $\mfv$, but in
  fact one can also replace $\mfv$ by $\mfv \cdot \cO_X$.  Let
  $\sigma: \caD_X^{\leq 1} \to \Vect(X)$ denote the principal symbol.
  Then, we conclude that $\sigma(\mfv) \subseteq \Vect(X)$ is actually
  a Lie algebroid (cf.~Notation \ref{ntn:la}).

  This is actually equivalent to using only vector fields, in the
  following sense: Given any pair $(X, \mfv)$ with $\mfv \subseteq
  \caD_X^{\leq 1}$, one can consider the pair $(\bA^1 \times X, \hat
  \mfv)$ where, for $x$ the coordinate on $\bA^1$, $\hat \mfv$
  contains the vector field $\partial_x$ together with, for every
  differential operator $\theta \in \mfv$, $\sigma(\theta) -
  (\theta-\sigma(\theta)) x \partial_x$.  Since $(x \partial_x + 1)
  = \partial_x \cdot x \in (\partial_x \cdot \caD_{\bA^1})$, one
  easily sees that $M(\bA^1 \times X, \hat \mfv) \cong \Omega_{\bA^1}
  \boxtimes M(X,\mfv)$.  So, in this sense, one can reduce the study
  of pairs $(X, \mfv)$ to the study of affine schemes of finite type
  with Lie algebras of vector fields. In particular, our general
  results extend easily to the setting of differential operators of
  order $\leq 1$.
\end{remark}
\begin{remark}
  Similarly, one can reduce the study of pairs $(X, \mfv)$ to the case
  where $X$ is affine space.  Indeed, if $X \into \bA^n$ is any
  embedding, and $I_X$ is the ideal of $X$, we can consider the Lie algebroid
\[
I_X \cdot \caD_{\bA^n}^{\leq 1} + \mfv \subseteq \caD_{\bA^n}^{\leq 1}.
\]
This makes sense by lifting elements of $\mfv$ to vector fields on
$\bA^n$, and the result is independent of the choice.  We can then
apply the previous remark to reduce everything to Lie algebras of
vector fields on affine space.  (This is not really helpful, though:
in our examples, $\mfv$ is naturally associated with $X$ (e.g.,
Hamiltonian vector fields on $X$), so it is not natural to replace $X$
with an affine space.)
\end{remark}

\subsection{Holonomicity and proof of Theorems \ref{t:main1} and \ref{t:main1-alt}}
Recall that a nonzero $\caD$-module on $X$ is \emph{holonomic} if it
is finitely generated and its singular support is a Lagrangian
subvariety of $T^*X$ (i.e., its dimension equals that of $X$). We
always call the zero module holonomic.  (Derived) pushforwards of
holonomic $\caD$-modules are well-known to have holonomic
cohomology. Since a holonomic $\caD$-module on a point is
finite-dimensional, this implies that, if $M$ is holonomic and $\pi: X
\to \pt$ is the pushforward to a point, then $\pi_* M$ (by which we
mean the cohomology of the complex of vector spaces), and in
particular $\pi_0 M$, is finite-dimensional.  Therefore, if we can
show that $M(X, \mfv)$ is holonomic, this implies that $(\cO_X)_{\mfv}
= \pi_0 M(X, \mfv)$ is finite-dimensional, along with the full
pushforward $\pi_* M(X,\mfv)$. This reduces Theorem \ref{t:main1-alt}
and equivalently Theorem \ref{t:main1} to the statement:
\begin{theorem}\label{t:main2}
  If $(X,\mfv)$ has finitely many leaves, then $M(X,\mfv)$ is holonomic.
  In this case, the composition factors are intermediate extensions of
  local systems along the leaves.
\end{theorem}
The converse does not hold: see, e.g., Example \ref{ex:hol-nf}.
\begin{proof}
  The equations $\gr \mfv$ are satisfied by the singular support of
  $M(X,\mfv)$.  These equations say, at every point $x \in X$, that
  the restriction of the singular support of $M(X,\mfv)$ to $x$ lies
  in $(\mfv|_x)^\perp$.  Thus, if $Z \subseteq X$ is a leaf, then the
  restriction of the singular support of $M(X,\mfv)$ to $Z$ lies in
  the conormal bundle to $Z$, which is Lagrangian.  If $X$ is a finite
  union of leaves, it follows that the singular support of $M(X,\mfv)$
  is contained in the union of the conormal bundles to the leaves,
  which is Lagrangian.  The last statement immediately follows from
  this description of the singular support.
\end{proof}
We will be interested in the condition on $\mfv$ for $M(X,\mfv)$ to be
holonomic, which turns out to be subtle.
\begin{definition}
Call $(X,\mfv)$, or $\mfv$, holonomic if $M(X,\mfv)$ is.  
\end{definition}
We will often use the following immediate consequence:
\begin{proposition}\label{p:hol-fd}
If $\mfv$ is holonomic, then $\cO_{\mfv}$ is finite-dimensional.
\end{proposition}
\begin{remark} \label{r:hol-fd} The converse to Proposition
  \ref{p:hol-fd} does not hold in general (although we will have a
  couple of cases where it does: the Lie algebras of all vector fields
  (Proposition \ref{p:allvfds}) and of Hamiltonian vector fields
  preserving a top polyvector field (Corollary \ref{c:vtop-hol})).  A
  simple example where this converse does not hold is
  $(X,\mfv)=(\bA^2, \langle \partial_x \rangle)$ (where $x$ is one of
  the coordinates on $\bA^2$), where $M(X,\mfv) = \Omega_{\bA^1}
  \boxtimes \caD_{\bA^1}$ is not holonomic, but $\cO_{\mfv} = 0$. This
  example also has infinitely many leaves, namely all lines parallel
  to the $x$-axis.
\end{remark}

\subsection{Incompressibility and a weak converse}\label{ss:inc}
We say that a vector field $\xi$ preserves a differential form
$\omega$ if the Lie derivative $L_\xi$ annihilates $\omega$.
\begin{definition}
  Say that $\mfv$ flows \emph{incompressibly} along an irreducible
  invariant subvariety $Z$ if there exists a smooth point $z \in Z$
  and a volume form on the formal neighborhood of $Z$ at $z$ which is
  preserved by $\mfv$.  
\end{definition}
There is an alternative definition using divergence functions which does
not require formal localization, which we discuss in \S \ref{s:div};
see also Proposition \ref{p:inc-subvar}.(iii).
When $X$ is irreducible and $\mfv$ flows incompressibly on $X$, we
omit the $X$ and merely say that $\mfv$ flows incompressibly.  Note
that this is equivalent to flowing generically incompressibly.

In \S \ref{ss:inc-subvar-pf} we will prove
\begin{proposition}\label{p:inc-subvar}
  Let $X$ be an irreducible affine variety.
The following conditions are equivalent:
\begin{itemize}
\item[(i)] $\mfv$ flows incompressibly;
\item[(ii)]  $M(X,\mfv)$ is fully supported;
\item[(iii)] For all $\xi_i \in \mfv$ and $f_i \in \cO_X$ such that $\sum_i f_i \xi_i = 0$, one has $\sum_i \xi_i(f_i) = 0$.
\end{itemize}
Moreover, the equivalence (ii) $\Leftrightarrow$ (iii) holds when
$X$ is an arbitrary affine scheme of finite type, if one generalizes (ii)
to the condition: (ii') The annihilator of $M(X,\mfv)$ in $\cO_X$ is zero.
\end{proposition}
\begin{remark}\label{r:inc-subvar}
We can alternatively state (ii') and (iii) as follows, in terms of
global sections of $\mfv \cdot \caD_Z \subseteq \caD_Z$ (cf.~\S
\ref{ss:inc-subvar-pf} below): (ii') says that $(\mfv \cdot \caD_Z) \cap \cO_Z =
0$, and (iii) says that $(\mfv \cdot \cO_Z) \cap \cO_Z = 0$.
\end{remark}

Motivated by this proposition, we will generalize the notion of
incompressibility to the case of nonreduced subschemes in \S
\ref{ss:supp} below, to be defined by conditions (ii') or (iii) above.
\begin{example}\label{ex:poiss-inc}
  In the case that $X$ is a Poisson variety, $\mfv$ is the Lie algebra
  of Hamiltonian vector fields, and $Z \subseteq X$ is a symplectic
  leaf (i.e., a leaf of $\mfv$), then $\mfv$ flows incompressibly on
  $Z$, since it preserves the symplectic volume along $Z$.
\end{example}
\begin{definition}
  Say that $\mfv$ has \emph{finitely many incompressible leaves} if it
  has no degenerate invariant subvariety on which $\mfv$ flows
  incompressibly.
\end{definition}
As before, if $\mfv$ does not have finitely many incompressible
leaves, one does \emph{not} necessarily have infinitely many
incompressible leaves, or any at all (see Example \ref{ex:noleaves},
which does not have finitely many incompressible leaves, but has no
algebraic leaves).  

In \S \ref{ss:hol-finc-pf} below we will prove
\begin{theorem}\label{t:hol-finc}
\begin{itemize}
\item[(i)] For every incompressible leaf $Z \subseteq X$, letting $i:
  \bar Z \into X$ be the tautological embedding of its closure, the
  canonical quotient $M(X,\mfv) \onto i_* M(\bar Z, \mfv|_{\bar Z})$
  is an extension of a nonzero local system on $Z$ to $\bar Z$.
\item[(ii)] If $(X,\mfv)$ is holonomic, then it has finitely many
  incompressible leaves.
\end{itemize}
\end{theorem}
Note that the converse to (i) does not hold: see Example
\ref{ex:a3-inf-hol-qt}.  We will give a correct converse statement in
\S \ref{ss:hol-bicond} below.  Also, the converse to (ii) does not
hold, as we will demonstrate in Example \ref{ex:fin-infl}.

We conclude from the Theorems \ref{t:main2} and \ref{t:hol-finc} that
\begin{equation}\label{e:finleaf-hol-inc}
\text{finitely many leaves} \,\, \Rightarrow \,\, \text{holonomic}
 \,\, \Rightarrow \,\, \text{finitely many incompressible leaves},
\end{equation}
but neither converse direction holds, as mentioned
(see Examples \ref{ex:hol-nf} and
\ref{ex:fin-infl}, respectively).
However, we will see below that the second
implication is generically a biconditional for irreducible varieties
$X$, i.e., $X$ generically has finitely many incompressible leaves if
and only if $X$ is generically holonomic.
\begin{example}\label{ex:poiss-finsym}
When $X$ is Poisson and
$\mfv$ the Lie algebra of Hamiltonian vector fields, then Theorem
\ref{t:hol-finc} and Example \ref{ex:poiss-inc} imply that
$\mfv$ is holonomic \emph{if and only if} $X$ has finitely many
symplectic leaves. More precisely, if $Z \subseteq X$ is any
invariant subvariety, then in the formal neighborhood of a
generic point $z \in Z$, we can integrate the Hamiltonian flow and
write $\hat Z_z = V \times V'$ for formal polydiscs $V$ and $V'$,
where the Hamiltonian flow is along the $V$ direction, and transitive
along fibers of $(V \times V') \onto V'$. Then Hamiltonian flow
preserves the volume form $\omega_V \otimes \omega_{V'}$, where
$\omega_V$ is the canonical symplectic volume, and $\omega_{V'}$ is an
arbitrary volume form on $V'$.  Therefore, all $Z$ are incompressible.
(In particular, all leaves are incompressible, preserving the
canonical symplectic volume.) Then \eqref{e:finleaf-hol-inc} shows
that $H(X)$ is holonomic if and only if there are finitely many leaves.
\end{example}
\begin{example}\label{ex:fin-infl}
  We demonstrate that $(\cO_X)_{\mfv}$ need not be finite-dimensional
  if we only assume that $X$ has finitely many incompressible leaves.
  Therefore, $\mfv$ is not holonomic (although non-holonomicity also
  follows directly in this example). Let $X=\bA^2 \times (\bA
  \setminus \{0\}) \subseteq \bA^3$, with $\bA^2 = \Spec \bk[x,y]$ and
  $\bA \setminus \{0\} = \Spec \bk[z,z^{-1}]$.  Let $\mfv = \langle
  y^2 \partial_x, y \partial_y + z \partial_z, \partial_z\rangle$.
  Then this has an incompressible open leaf, $\{y \neq 0\}$,
  preserving the volume form $\frac{1}{y^2} dx \wedge dy \wedge
  dz$. The complement consists of the leaves $\{x=c, y=0\}$ for all $c
  \in \bk$, which are not incompressible since the restriction of
  $\mfv$ to each such leaf (or to their union, $\{y=0\}$) includes
  both $\partial_z$ and $z \partial_z$.

  We claim that the coinvariants $(\cO_X)_{\mfv}$ are
  infinite-dimensional, and isomorphic to $\bk[x] \cdot yz^{-1}$ via
  the quotient map $\cO_X \onto (\cO_X)_{\mfv}$.  Indeed, $
  y^2 \partial_x (\cO_X) = y^2 \cO_X$,
  $(y \partial_y + z \partial_z) \cO_X = \bk[x] \cdot \langle y^i z^j
  \mid i+j \neq 0 \rangle$, and $\partial_z (\cO_X) = \bk[x,y] \cdot
  \langle z^i \mid i \neq -1\rangle$.  The sum of these vector
  subspaces is the space spanned by all monomials in $x, y, z$, and $z^{-1}$
  except for $x^i yz^{-1}$ for all $i \geq 0$.
\end{example}
\begin{example}\label{ex:hol-nf}
  It is easy to give an example where $\mfv$ is holonomic but has
  infinitely many leaves: for $Y$ any positive-dimensional variety,
  consider $X=\bA^1 \times Y$, $\mfv := \langle \partial_x,
  x \partial_x \rangle$, where $x$ is the coordinate on $\bA^1$.  Then
  the leaves of $(X,\mfv)$ are of the form $\bA^1 \times \{y\}$ for $y
  \in Y$, but $M(X,\mfv)=0$, which is holonomic.
\end{example}
\begin{example}\label{ex:a3-inf-hol-sub}
  For a less trivial example, which is a generically nonzero holonomic
  $\caD$-module without finitely many leaves, let $X=\bA^3$ with
  coordinates $x, y$, and $z$, and let $\mfv$ be the Lie algebra of
  all incompressible vector fields (with respect to the standard
  volume) which along the plane $x=0$ are parallel to the
  $y$-axis. Then we claim that the singular support of $M(X,\mfv)$ is
  the union of the zero section of $T^* X$ and the conormal bundle of
  the plane $x=0$, which is Lagrangian, even though there are not
  finitely many leaves.  Actually, from the computation below, we see
  that $M(X,\mfv)$ is isomorphic to $j_!\Omega_{\bA^1 \setminus \{0\}}
  \boxtimes \Omega_{\bA^2}$, where $j: \bA^1 \setminus \{0\} \into
  \bA^1$ is the inclusion (which is an affine open embedding, so $j_!$
  is an exact functor on holonomic $\caD$-modules).  This is an
  extension of $\Omega_{\bA^3}$ by $i_*\Omega_{\bA^2}$, where $i:
  \bA^2 = \{0\} \times \bA^2 \into \bA^3$ is the closed embedding,
  i.e., there is an exact sequence
\[
0 \to i_* \Omega_{\bA^2} \into M(X,\mfv) \onto \Omega_{\bA^3} \to 0.
\]
Thus, there is a single composition factor on the open leaf and a
single composition factor on the degenerate (but not incompressible)
invariant subvariety $\{x=0\}$.

  To see this, note first that $\partial_y \in \mfv$. We claim 
  that $1+x\partial_x$ and $\partial_z$ are in $\mfv \cdot \caD_X$:
\begin{gather*}
  \partial_y \cdot y - (y \partial_y - x \partial_x) = 1 + x \partial_x; \\
  (1 + x \partial_x) \cdot \partial_z - (x \partial_z)
  \cdot \partial_x = \partial_z.
\end{gather*}
Thus, $\langle 1+x\partial_x, \partial_y, \partial_z \rangle \subseteq
\mfv \cdot \caD_X$.  Conversely, we claim that $\mfv \subseteq \langle
1+x\partial_x, \partial_y, \partial_z \rangle \cdot \caD_X$.  Indeed,
given an incompressible vector field of the form $\xi = xf\partial_x +
g \partial_y + xh \partial_z \in \mfv$ for $f,g,h \in \cO_X$, we can
write
\[
\xi = (1 + x \partial_x) \cdot f + \partial_y \cdot g + x \partial_z \cdot h,
\]
where the RHS is a vector field (and not merely a differential
operator of order $\leq 1$) because $\xi$ is incompressible.
Explicitly, the condition for this RHS to be a vector field, and the
condition for $\xi$ to be incompressible, are both that
$\partial_x(xf) + \partial_y(g) + \partial_z(xh) = 0$.

We conclude that $\langle 1+x\partial_x, \partial_y, \partial_z
\rangle \cdot \caD_X = \mfv \cdot \caD_X$. Therefore, $M(X,\mfv) \cong
j_!  \Omega_{\bA^1 \setminus \{0\}} \boxtimes \Omega_{\bA^2}$, as
claimed.
\end{example}
\begin{example}\label{ex:a3-inf-hol-qt}
  We can slightly modify Example \ref{ex:a3-inf-hol-sub}, so that
  (again for $X:=\bA^3$ and $i: \{0\} \times \bA^2 \into \bA^3$), $i_*
  \Omega_{\bA^2}$ appears as a quotient of $M(X,\mfv)$ rather than as
  a submodule. More precisely, we will have $M(X,\mfv) \cong j_*
  \Omega_{\bA^1 \setminus \{0\}} \boxtimes \Omega_{\bA^2}$. To do so,
  let $\mfv$ be the Lie algebra of all incompressible vector fields
  preserving the volume form $\frac{1}{x^2} dx \wedge dy \wedge dz$
  (cf.~Example \ref{ex:fin-infl}), which again along the plane $x=0$
  are parallel to the $y$-axis.  Note also that, in this example, the
  subvariety $\{0\} \times \bA^2$ is still not incompressible (since
  $\partial_y$ and $y \partial_y$ are both in $\mfv|_{0 \times \bA^2}$, and these cannot
  both preserve the same volume form), even though this subvariety now
  supports a quotient $i_* \Omega_{\bA^2}$ of $M(X)$.

  To see this, we claim that $\mfv \cdot \caD_Z = \langle 1 -
  x \partial_x, \partial_y, \partial_z \rangle \cdot \caD_Z$.  For 
  the
  containment $\supseteq$, we show that $1 - x \partial_x$ and $\partial_z$ are in
  $\mfv \cdot \caD_Z$. This follows from
\begin{gather*}
  \partial_y \cdot y - (y \partial_y + x \partial_x) = 1 - x \partial_x; \\
  (1 - x \partial_x) \cdot \partial_z + (x \partial_z)
  \cdot \partial_x = \partial_z.
\end{gather*}
Then, as in Example \ref{ex:a3-inf-hol-sub}, 
if $\xi = xf \partial_x + g \partial_y + xh \partial_z$ preserves the volume
form $\frac{1}{x^2} dx \wedge dy \wedge dz$, then 
\[
\xi = -(1 - x \partial_x) \cdot f + \partial_y \cdot g + x \partial_z \cdot h.
\]
Therefore, we also have the opposite containment, $\mfv \cdot \caD_Z
\subseteq \langle 1 - x \partial_x, \partial_y, \partial_z \rangle
\cdot \caD_Z$.  As a consequence, $M(X,\mfv) \cong j_* \Omega_{\bA^1
  \setminus \{0\}} \boxtimes \Omega_{\bA^2}$.  We therefore have a
canonical exact sequence
\[
0 \to \Omega_{\bA^3} \into M(X,\mfv) \onto i_* \Omega_{\bA^2} \to 0.
\]

\end{example}

\subsection{The transitive case}\label{ss:transitive}
In this section we consider the simplest, but important, example of
$\mfv$ and the $\caD$-module $M(X,\mfv)$, namely when $\mfv$ has maximal rank
everywhere:
\begin{definition}
  A pair $(X, \mfv)$ is called \emph{transitive at $x$} if $\mfv|_x =
  T_x X$.  We call the pair  $(X, \mfv)$ transitive if it is so at
  all $x \in X$.
\end{definition}
In other words, the transitive case is the one where every connected
component of $X$ is a leaf.  Note that, in particular,
 $X$ must be a smooth variety.  Also, we remark that
$X$ is generically transitive if and only if it is not degenerate.
\begin{proposition}\label{p:tr}
  If $(X,\mfv)$ is transitive and connected, then $M(X,\mfv)$ is a
  rank-one local system if $\mfv$ flows incompressibly, and
  $M(X,\mfv)=0$ otherwise.
\end{proposition}
\begin{proof}
  By taking associated graded of $M(X,\mfv)$, in the transitive
  connected case, one obtains either $\cO_X$ (where $X \subseteq T^*
  X$ is the zero section) or zero.  So $M(X, \mfv)$ is either a
  one-dimensional local system on $X$, or zero.  In the incompressible
  case, in a formal neighborhood of some $x \in X$, a volume form is
  preserved, so there is a surjection $M(\hat X_x, \mfv|_{\hat X_x})
  \onto \Omega_{\hat X_x}$, and hence in this case $M(X, \mfv)$ is a
  one-dimensional local system.  Conversely, if $M(X, \mfv)$ is a
  one-dimensional local system, then in a formal neighborhood of any
  point $x \in X$, it is a trivial local system, and hence it
  preserves a volume form there.
\end{proof}
\begin{example}\label{ex:cy}
  In the case when $X$ is connected and Calabi-Yau and $\mfv$
  preserves the global volume form (which includes the case where $X$
  is symplectic and $\mfv$ is the Lie algebra of Hamiltonian vector
  fields), then we conclude that $M(X,\mfv) \cong \Omega_X$. Thus, for
  $\pi: X \to \pt$ the projection to a point, $(\cO_X)_\mfv = \pi_0
  \Omega_X = H_{DR}^{\dim X}(X)$, the top de Rham cohomology. Taking
  the derived pushforward, we conclude that $\pi_* M(X,\mfv) = \pi_*
  \Omega_X = H_{DR}^{\dim X - *}(X)$.  In the Poisson case, where
  $(\cO_X)_{\mfv}$ is the zeroth Poisson homology, in \cite[Remark
  2.27]{ESdm} this motivated the term \emph{Poisson-de Rham homology},
  $HP^{DR}_*(X) = \pi_* M(X,\mfv)$, for the derived pushforward.  More
  generally, if $\mfv$ preserves a \emph{multivalued} volume form,
  then $M(X,\mfv)$ is a nontrivial rank-one local system and $\pi_*
  M(X,\mfv) = H_{DR}^{\dim X - *}(X, M(X,\mfv))$ is the cohomology of
  $X$ with coefficients in this local system (identifying $M(X,\mfv)$
  with its corresponding local system under the de Rham functor).  See
  the next example for more details on how to define such $\mfv$.
\end{example}
\begin{example}\label{ex:mvvol}
  The rank-one local system need not be trivial when $\mfv$ does not
  preserve a global volume form.  For example, let $X =
  (\bA^1\setminus\{0\}) \times \bA^1 = \Spec \bk[x,x^{-1},y]$.  Then
  we can let $\mfv$ be the Lie algebra of vector fields preserving the
  multivalued volume form $d(x^{r}) \wedge dy$ for $r \in \bk$.  It is
  easy to check that this makes sense and that the resulting Lie
  algebra $\mfv$ is transitive.  Then, $M(X,\mfv)$ is the rank-one
  local system whose homomorphisms to $\Omega_X$ correspond to this
  volume form, which is nontrivial (but with regular singularities)
  when $r$ is not an integer.  For $\bk=\bC$, the local system
  $M(X,\mfv)$ thus has monodromy $e^{-2 \pi i r}$.

  More generally, if $X$ is an arbitrary smooth variety of pure
  dimension at least two, and $\nabla$ is a flat connection on
  $\Omega_X$, we can think of the flat sections of $\nabla$ as giving
  multivalued volume forms, and define a corresponding Lie algebra
  $\mfv$ so that $\Hom_{\caD_X}(M(X,\mfv), \Omega_X)$ returns these
  forms on formal neighborhoods.  Precisely, we can let $\mfv$ be the
  Lie algebra of vector fields preserving formal flat sections of
  $\nabla$. We need to check that $\mfv$ is transitive, which is where
  we use the hypothesis that $X$ has pure dimension at least two: see
  \S \ref{ss:ham-df} and in particular Proposition
  \ref{p:div-enough-vfds} (alternatively, we could simply impose the
  condition that $\mfv$ be transitive, which is immediate to check in
  the example of the previous paragraph).  Then $M(X,\mfv) \cong
  (\Omega_X, \nabla)^* \otimes_{\cO_X} \Omega_X$, via the map sending
  the canonical generator $1 \in M(X,\mfv)$ to the identity element of
  $\End_{\cO_X}(\Omega_X)$.  Conversely, if $(X,\mfv)$ is transitive
  and $M(X,\mfv)$ is nonzero (hence a rank-one local system), then
  $\Hom_{\cO_X}(M(X,\mfv), \Omega_X)$ canonically has the structure of
  a local system on $\Omega_X$ with formal flat sections given by
  $\Hom_{\caD_X}(M(X,\mfv),\Omega_X)$, and one has a canonical
  isomorphism
\[
M(X,\mfv) \cong \Hom_{\cO_X}(M(X,\mfv), \Omega_X)^*
  \otimes_{\cO_X} \Omega_X.
\]

On the other hand, if $X$ is one-dimensional and $\mfv$ is transitive,
then $M(X,\mfv)$ cannot be a nontrivial local system, since there are
no vector fields defined in any Zariski open set preserving a
nontrivial local system.  More precisely, assuming $X$ is a connected
smooth curve, in order to be incompressible, $\mfv$ must be a
one-dimensional vector space. Then, if $\xi \in \mfv$ is nonzero, then
the inverse $\xi^{-1}$ defines the volume form preserved by $\mfv$.
\end{example}
\begin{corollary}\label{c:inc}
  If $(X,\mfv)$ is a variety, then $M(X,\mfv)$ is fully supported on
  $X$ if and only if $\mfv$ flows incompressibly on every irreducible
  component of $X$. In this case, the dimension of the singular
  support of $M(X,\mfv)$ on each irreducible component $Y \subseteq X$
  is generically $\dim Y + (\dim Y - r)$, where
  $r$ is the generic rank of $\mfv$ on $Y$.
\end{corollary}
\begin{corollary}\label{c:inc2}
  If $(X,\mfv)$ is an irreducible variety, then $\mfv$ is generically
  holonomic if and only if it is either generically transitive or not
  incompressible.
\end{corollary}
\begin{proof}[Proof of Corollary \ref{c:inc}]
  It suffices to assume $X$ is irreducible, since the statements can
  be checked generically on each irreducible component.  For generic
  $x \in X$, in the formal neighborhood $\hat X_x$, we can integrate
  the flow of $\mfv$ and write $\hat X_x \cong (V \times V')$, where
  $V$ and $V'$ are two formal polydiscs about zero, mapping $x \in
  \hat X_x$ to $(0,0) \in (V \times V')$, and such that $\mfv$
  generates the tangent space in the $V$ direction everywhere, i.e.,
  $\mfv|_{(v,v')} = T_{v} V \times \{0\}$ at all $(v,v') \in (V \times
  V')$.

  Since $\hat \cO_{X,x} \cdot \mfv = T_V \boxtimes \cO_{V'}$, inside
  $\mfv \cdot \hat \cO_{X,x}$ we have, for every $\xi \in T_V$, an
  element of the form $\xi + D(\xi)$, for some $D(\xi) \in \hat
  \cO_{X,x}$. Namely, this is true because, when $\xi \in \mfv$ and $f
  \in \hat \cO_{X,x}$, $\xi \cdot f = f \cdot \xi + \xi(f) \in \mfv
  \cdot \hat \cO_{X,x}$, and $T_V$ is contained in the span of such $f
  \cdot \xi$.

  Now assume that $\mfv$ preserves a volume form $\omega$ on $\hat
  X_x$.  Recall that this means that, for all $\xi \in \mfv$, one has
  $L_\xi \omega = 0$.  Since the right $\caD$-module action of vector
  fields $\xi \in \Vect(X)$ on $\Omega_X$ is by $\omega \cdot \xi :=
  -L_{\xi} \omega$, we conclude that $D(\xi) = L_\xi \omega / \omega$.
  Write $\omega = f \cdot \omega_V \wedge \omega_{V'}$ where
  $\omega_V$ and $\omega_{V'}$ are volume forms on $V$ and $V'$ and $f
  \in \hat \cO_{X,x}$ is a unit.  Then we conclude that $M(\hat X_x,
  \mfv|_{\hat X_x}) \cong \Omega_V \boxtimes \caD_{V'}$, the quotient
  of $\caD_{X,x}$ by the right ideal generated by
  $\omega_V$-preserving vector fields on $V$.

  Conversely, assume that $M(X,\mfv)$ is fully supported. Since $x$
  was generic, $M(\hat X_x, \mfv|_{\hat X_x})$ is also fully
  supported.  Thus, for every $\xi \in T_V$, there is a unique
  $D(\xi)$ such that $\xi + D(\xi) \in \mfv \cdot \hat \caD_{X,x}$
  (and in fact this is in $\mfv \cdot \hat \cO_{X,x}$).

  Let $\partial_1, \ldots, \partial_n$ be the constant vector fields
  on $V \times V'$.  We conclude that $\mfv \cdot \hat \caD_{X,x} = \{
  \xi + D(\xi): \xi \in T_V \boxtimes \cO_{V'}\} \cdot \Sym
  \langle \partial_1, \ldots, \partial_n \rangle$.  Since $M(\hat X_x,
  \mfv|_{\hat X_x})$ is fully supported, this implies that $\gr(\mfv
  \cdot \hat \caD_{X,x}) = T_V \cdot \Sym_{\hat \cO_{X,x}} T_{\hat
    X_x}$, and hence that $M(\hat X_x, \mfv|_{\hat X_x}) \cong
  \Omega_V \boxtimes \caD_{V'}$.  Then, $\mfv$ also preserves a formal
  volume form, since $\Hom(M(\hat X_x, \mfv|_{\hat X_x}), \Omega_{V
    \times V'}) \neq 0$. (Explicitly, for the unique (up to scaling)
  volume form $\omega_V$ on $V$ preserved by $\mfv|_V$, these are of
  the form $\omega_V \boxtimes \omega_{V'}$ for arbitrary volume forms
  $\omega_{V'}$ on $V'$.)

  For the final statement, the proof shows that, in the incompressible
  (irreducible) case, the dimension of the singular support is
  generically $\dim V + 2 \dim V'$, which is the same as the claimed
  formula when we note that $\dim V = r$ and $\dim V + \dim V' = \dim
  Y$.
\end{proof}
\subsection{Proof of Proposition
  \ref{p:inc-subvar}}\label{ss:inc-subvar-pf}
By Corollary \ref{c:inc}, conditions (i) and (ii) are equivalent, when
$X$ is an irreducible affine variety.  Now let $X$ be an arbitrary
affine scheme of finite type. We prove that (ii') and (iii) are
equivalent.

In view of Remark \ref{r:inc-subvar}, the implication (ii')
$\Rightarrow$ (iii) is immediate.  To make Remark \ref{r:inc-subvar}
precise, we should define $\mfv \cdot \cO_X$ as a subspace of global
sections of $\caD_X$. One way to do this is to take an embedding $i: X
\to V$ into a smooth affine variety $V$ as in \S \ref{ss:ddv}; in the
notation there, the global sections of $i_*(\mfv \cdot \caD_X)$ then
identify as
\begin{equation}\label{e:gs-vd}
\Gamma(V, i_*(\mfv \cdot \caD_X))
= I_X  \caD_V \setminus \bigl( (\widetilde \mfv+I_X) \cdot \caD_V\bigr).
\end{equation}
Then, by $\mfv \cdot \cO_X$ we mean the subspace 
\begin{equation}\label{e:vox}
\mfv \cdot \cO_X =
(I_X  \caD_V  \cap \widetilde \mfv \cdot \cO_V) \setminus
\widetilde \mfv \cdot \cO_V.
\end{equation}
Finally, by $\cO_X$ itself, we mean the subspace
\begin{equation}\label{e:ox}
\cO_X = (I_X  \caD_V \cap \cO_V) \setminus \cO_V.
\end{equation}
Then, it follows that (ii') is equivalent to $(\mfv \cdot \caD_X) \cap
\cO_X = 0$ and that (iii) is equivalent to $(\mfv \cdot \cO_X)
\cap \cO_X = 0$, as desired. In other words, it is equivalent to
ask that $(\widetilde \mfv \cdot \caD_V) \cap \cO_V \subseteq I_X$ and
$(\widetilde \mfv \cdot \cO_V) \cap \cO_V \subseteq I_X$.

We now prove that (iii) implies (ii'). Assume that $V = \bA^n = \Spec
\bk[x_1, \ldots, x_n]$.  Note that
\[
(\widetilde \mfv + I_X) \cdot \caD_V = (\widetilde \mfv + I_X) \cdot
\cO_V \cdot \Sym \langle\partial_1, \ldots, \partial_n \rangle.
\]
Thus, the fact that $\bigl((\widetilde \mfv + I_X) \cdot \cO_V\bigr) \cap \cO_V = I_X$, i.e., (iii), implies also that $\bigl(
(\widetilde \mfv + I_X) \cdot \caD_V \bigr) \cap \cO_V = I_X$, i.e., (ii').

\begin{remark}\label{r:inc-subvar-alt}
  For irreducible affine varieties, we can also show that (i) and
  (iii) are equivalent directly without using Remark
  \ref{r:inc-subvar}, and hence by Corollary \ref{c:inc}, that (ii)
  and (iii) are equivalent.  Suppose (i). By Corollary \ref{c:inc},
  $\mfv$ flows incompressibly on $Z$. Let $z \in Z$ be a smooth point
  and $\omega \in \Omega_{\hat Z_z}$ be a formal volume preserved by
  $\mfv$. Then, if $f_i \in \cO_Z$ and $\xi_i \in \mfv|_Z$ satisfy
  $\sum_i f_i \xi_i = 0$, we have $0 = L_{f_i \xi_i} \omega = \sum_i
  \xi_i(f_i)$, which proves (iii).

Conversely, suppose that (iii) is satisfied. Let $z \in Z$ be a smooth
point where the rank of $\mfv|_Z$ is maximal.  Then, in a neighborhood
$U \subseteq Z$ of $z$, $\cO_U \cdot \mfv$ is a free submodule of
$T_{U}$, and hence has a basis $\xi_1, \ldots, \xi_j$. In the language
of \S \ref{s:div}, one can define a divergence function $D:
\cO_U \cdot \mfv \to T_U, D(\sum_i f_i \xi_i) = \sum_i \xi_i(f_i)$.
Therefore, by Proposition \ref{p:reinc}, $\mfv$ flows incompressibly
on $U$, and hence on $Z$.
\end{remark}

\subsection{Proof of Theorem \ref{t:hol-finc}}\label{ss:hol-finc-pf}
Part (i) is an immediate consequence of Proposition \ref{p:inv-sub-qt}
and Corollary \ref{c:inc}.

For part (ii), suppose that $X$ does not have finitely many
incompressible leaves.  Then, there is a degenerate invariant
subvariety $i: Z \into X$ such that $\mfv$ flows incompressibly on
$Z$. By Proposition \ref{p:inv-sub-qt} and Corollary \ref{c:inc},
there is a nonholonomic quotient of $M(X,\mfv)$ supported on the
closure of $Z$.  So $M(X,\mfv)$ is not holonomic.

\subsection{Support and saturation}\label{ss:supp}
To proceed, note that in some cases, $M(X, \mfv)$ is actually
supported on a proper subvariety, e.g., in Example \ref{ex:hol-nf},
where it is zero; more generally, by Proposition \ref{p:inc-subvar},
this happens if and only if $\mfv$ does not flow incompressibly.  In
this case, it makes sense to replace $X$ with the support of
$M(X,\mfv)$, and define an equivalent system there. More precisely, we
define a scheme-theoretic support of $M(X,\mfv)$:
\begin{definition}\label{d:supp}
  The support of $(X,\mfv)$ is the closed subscheme $X_\mfv \subseteq X$
  defined by the ideal $(\mfv \cdot \caD_X) \cap \cO_X$ of $\cO_X$.
\end{definition}
To make sense of this definition, we work in the space of global
sections of $\mfv \cdot \caD_X$, using \eqref{e:vox} and \eqref{e:ox}.
Note that here it is \emph{essential} that we allow $X_\mfv$ to be
nonreduced (this was our motivation for working in the nonreduced
context).  

We immediately conclude
\begin{proposition}
Let $i: X_\mfv \to X$ be the natural closed embedding.
Then, there is a canonical isomorphism $M(X, \mfv) \cong i_* M(X_\mfv,
\mfv|_{X_\mfv})$.
\end{proposition}
The above remarks say that, when $X$ is a variety, $X=X_\mfv$ if and
only if $\mfv$ flows incompressibly. Moreover, $\mfv$ flows
incompressibly on an invariant subvariety $Z \subseteq X$ if and only
if $Z = Z_{\mfv|_Z}$. With this in mind, we extend the definition of
incompressibility to subschemes:
\begin{definition}
  We say that $\mfv$ flows incompressibly on an invariant
  subscheme $Z$ if $Z = Z_{\mfv|_Z}$. 
\end{definition}
With this definition, as promised, the conditions (i), (ii'), and
(iii) of Proposition \ref{p:inc-subvar} are equivalent for arbitrary
affine schemes of finite type.
\begin{proposition}\label{p:inc-supp}
  Let $Z \subseteq X$ be an irreducible closed subvariety.  Then there
  exists a quotient of $M(X,\mfv)$ whose support is $Z$ if and only if
  $Z$ is invariant and $\mfv$ flows incompressibly on some
  infinitesimal thickening of $Z$.  In this case, this quotient
  factors through the quotient $M(X,\mfv) \onto i_* M(Z',
  \mfv|_{Z'})$, for some infinitesimal thickening $Z'$, with inclusion
  $i: Z' \into X$.
\end{proposition}
Here, an \emph{infinitesimal thickening} of a subvariety $Z \subseteq
X$ is a subscheme $Z' \subseteq X$ such that $Z'_\red = Z$. Note that
it can happen that $\mfv$ flows incompressibly on $Z'$ but not on $Z$,
as in Example \ref{ex:a3-inf-hol-qt}.
We caution that, on the other hand, $M(X,\mfv)$ could have a
\emph{submodule} supported on $Z$ even if $\mfv$ does not flow
incompressibly on any infinitesimal thickening of $Z$: see Example
\ref{ex:a3-inf-hol-sub}.
\begin{proof}
  $M(X,\mfv) = \mfv \cdot \caD_X \setminus \caD_X$ admits a quotient
  supported on $Z$ if and only if, for some $N \geq 1$, $(\mfv + I_Z^N)
  \cdot \caD_X$ is not the unit ideal.  This is equivalent to saying
  that $M(Z',\mfv|_{Z'}) \neq 0$ for some infinitesimal thickening
  $Z'$ of $Z$.  This can only happen if $Z$ is invariant.  By
  definition, such a restriction is fully supported if and only if
  $\mfv$ flows incompressibly on $Z'$.  For the final statement, note
  that the quotient morphism must factor through a map $M(X,\mfv)
  \onto (\mfv + I_Z^N) \caD_X \setminus \caD_X$, and the
  latter is $M(Z', \mfv|_{Z'})$, where we define $Z'$ by $I_{Z'} =
  I_Z^N$.
\end{proof}
  Next, even if $X=X_\mfv$, there can be many choices of $\mfv$ that give
  rise to the same $\caD$-module.  This motivates
\begin{definition}
 The \emph{saturation} $\mfv^{s}$ of $\mfv$ is
 $\Vect(X) \cap (\mfv \cdot
  \caD_X)$. Precisely, in the language of
\S \ref{ss:inc-subvar-pf} for an embedding $i: X \into V$,
\[
\mfv^s = 
\Bigl( \Vect(V) \cap \bigl( (\widetilde \mfv + I_X) \cdot \caD_V \bigr)\Bigr)|_X.
\]
\end{definition}
It is easy to check that the definition of the saturation does not
depend on the choice of embedding.  We next define a smaller, but more
computable, saturation:
\begin{definition}\label{d:os}
  The \emph{$\cO$-saturation} $\mfv^{os}$ of $\mfv$ is $\Vect(X) \cap
  (\mfv \cdot \cO_X)$, precisely,
\[
\mfv^{os} := \{\sum_i f_i \xi_i \mid f_i \in \cO_X, \xi_i \in \mfv,
\text{ s.t. } \sum_i \xi_i(f_i) = 0\}.
\]
\end{definition}
Equivalently, for any embedding $X \subseteq V$ as above,
\[
\mfv^{os} = \Bigl( \Vect(V) \cap \bigl( (\mfv_V + I_X) \cdot \cO_V
\bigr)\Bigr)|_X.
\]
Note that, by definition, $\mfv^{os} \subseteq \cO_X \cdot \mfv$;
however, the same does \emph{not} necessarily hold for $\mfv^s$, as in
Examples \ref{ex:a3-inf-hol-sub} and \ref{ex:a3-inf-hol-qt}. In
particular, in those examples, $\mfv^s$ has rank two on the locus
$x=0$, whereas $\mfv^{os}$ has rank one.

However, generically on incompressible affine varieties,
$\mfv^{os}=\mfv^s$.  More precisely:
\begin{definition}
If $(X,\mfv)$ is incompressible, then call a vector field
$\xi\in \cO_X \cdot \mfv$
 \emph{incompressible} if, writing $\xi = \sum_i f_i \xi_i$ for $f_i \in \cO_X$
and $\xi_i \in \mfv$, one has $\sum_i \xi_i(f_i) = 0$.
\end{definition}
Note that we used incompressibility for the definition to make sense; otherwise
there could be multiple expressions $\sum_i f_i \xi_i$ for $\xi$ which
yield different values $\sum_i \xi_i(f_i)$.
\begin{remark}
  When $X$ is a variety, $\xi \in \cO_X \cdot \mfv$ is incompressible
  if and only if, for every irreducible component of $X$, at a smooth
  point with a formal volume preserved by $\mfv$, then $\xi$ also
  preserves that volume.
\end{remark}
\begin{proposition} \label{p:inc-os} If $\mfv$ flows incompressibly,
  then $\mfv^{os}$ is the subspace of $\cO_X \cdot \mfv$ of
  incompressible vector fields.  If $X$ is additionally a variety,
 then for some open dense subset
  $U \subseteq X$, $(\mfv|_U)^s = (\mfv|_U)^{os}$ is the subspace of
  $\cO_U \cdot \mfv$ of incompressible vector fields.
\end{proposition}
\begin{proof}
  For the first statement, if $X$ is incompressible and $f_i \in
  \cO_X, \xi_i \in \mfv$ are such that $\sum_i \xi_i(f_i) = 0$, then
  it follows that $\sum_i f_i \cdot \xi_i = \sum_i \xi_i \cdot f_i$.


  For the second statement, first note that, by Corollary \ref{c:inc},
  since $\mfv$ is incompressible and $X$ is a variety, on each
  irreducible component, $\mfv^s$ must generically have the same rank
  as $\mfv$. Now let $U \subseteq X$ be the locus of smooth points $x
  \in X$ such that, if $Y \subseteq X$ is the irreducible component
  containing $x$, the dimension $\mfv|_x$ is maximal along $Y$.  Then
  $\cO_U \cdot \mfv|_U$ is locally free.  It follows that this also
  equals $\cO_U \cdot (\mfv|_U)^s$.  Since $M(U,(\mfv|_U)^s) = M(U,
  \mfv|_U)$ is fully supported, $(\mfv|_U)^s$ is incompressible. By
  the first part, we therefore have $(\mfv|_U)^s \subseteq
  (\mfv|_U)^{os}$; the opposite inclusion is true by definition.
  Finally, note that, by definition, $U$ is open and dense.
\end{proof}
\begin{example}\label{ex:red-gen-sat}
  When $X = X_\mfv$ is reduced and irreducible, in the formal
  neighborhood of a generic point of $x \in X$, one has $\hat X_x
  \cong (V \times V')$ for formal polydiscs $V$ and $V'$, and
  $\mfv^{s}=\mfv^{os} = \cO_{V'} \cdot H(V)$ where $V$ is equipped
  with its standard volume form (this also gives an alternative proof
  of part of Proposition \ref{p:inc-os}). So, up to isomorphism,
  this only depends on the dimension of $X$ and the generic rank of
  $\mfv$.
\end{example}
\begin{remark}\label{r:supp-sat}
  There is a close relationship between the saturation and the support
  ideal.  In the language of Remark \ref{r:diffop-leq1}, if we
  generalize $\mfv$ to the setting of differential operators of order
  $\leq 1$, then the natural saturation becomes $(\mfv \cdot \caD_X)
  \cap \caD_X^{\leq 1}$.  In the case $\mfv \subseteq \Vect(X)$, this
  saturation contains both $\mfv^s$ and the ideal
  of $X_\mfv$; by a computation similar to that of \S
  \ref{ss:inc-subvar-pf}, in fact, this saturation is $\mfv^s \cdot
  \cO_X$.  
\end{remark}
\begin{remark}
  By Remark \ref{r:supp-sat}, one obtains an alternative formula for
  the support ideal, call it $I_{X_\mfv}$, of $X$: this is $I_{X_\mfv}
  = (\mfv^s \cdot \cO_X) \cap \cO_X$.  This can be viewed as a
  generalization of the equivalence of Proposition \ref{p:inc-subvar},
  (ii') $\Leftrightarrow$ (iii), in the case that $\mfv = \mfv^s$ is
  saturated.
\end{remark}
\subsection{Holonomicity criteria} \label{ss:hol-bicond}
\begin{theorem}\label{t:hol-bicond}
  The following conditions are equivalent:
\begin{itemize}
\item[(i)] $(X,\mfv)$ is holonomic;
\item[(ii)] For every (degenerate closed) invariant subscheme $Z'
  \subseteq X$ on which $\mfv$ flows incompressibly, for $i:Z := Z'_\red \to
  Z'$ the inclusion, $i^! M(Z',\mfv|_{Z'})$ is generically a local system;
\item[(iii)] $X$ has only finitely many invariant closed subvarieties $Z$
  on which $\mfv$ flows incompressibly in some infinitesimal
  thickening $i: Z \into Z' \subseteq X$, and for all of them,
 in formal neighborhoods of
generic $z \in Z$ there is a canonical isomorphism 
\[
i^! M(Z', \mfv_{Z'}) \cong \Omega_{\hat Z_z} \otimes ((i_*
\Omega_{\hat Z_z})^{\mfv|_{Z'}})^*.
\]
\end{itemize}
In this case, $M(X,\mfv)$ admits a filtration
\[
0 \subseteq M_{\geq \dim X}(X, \mfv) \subseteq M_{\geq \dim X-1}(X, \mfv)
\subseteq \cdots \subseteq M_{\geq 0}(X,\mfv) = M(X,\mfv),
\]
whose subquotients $M_{\geq j}(X,\mfv) / M_{\geq (j+1)}(X,\mfv)$ are
direct sums of indecomposable extensions of local systems on open
subvarieties of the dimension $j$ varieties appearing in (iii) by
local systems on subvarieties of their boundaries.
\end{theorem}
Here $(i_* \Omega_{\hat Z_z})^{\mfv|_{Z'}}$ is the
(finite-dimensional) vector space of distributions along $Z$ preserved
by the flow of $\mfv|_{Z'}$.  For example, in the case that there
exists a product decomposition $\widehat {Z'_z} \cong \hat Z_z \times
S$ for some zero-dimensional scheme $S$, for which the inclusion of
$\hat Z_z$ is the obvious one to $\hat Z_z \times \{0\}$, then $i_*
\Omega_{\hat Z_z} \cong (\Omega_{\hat Z_z} \otimes_\bk \cO_S^*)$,
where $\Omega_{\hat Z_z}$ is the space of formal volume forms on $\hat
Z_z$ and $\cO_S^*$ is the (finite-dimensional) space of algebraic
distributions on $S$.

We remark that the theorem also gives another proof of Proposition
\ref{p:tr} (which we don't use in the proof of the theorem), since a connected
transitive
variety $(X,\mfv)$ is a single leaf and therefore $\mfv$ is holonomic.

Using part (iii) of the Theorem, we immediately conclude
\begin{corollary}
  When $(X,\mfv)$ is holonomic, an invariant subscheme $Z' \subseteq
  X$ is incompressible if and only if, for generic $z \in Z :=
  Z'_\red$, with $i: Z \into Z'$ the inclusion, $(i_*\Omega_{\hat
    Z_z})^{\mfv|_{Z'}} \neq 0$.   
\end{corollary}
Note that, when $Z'$ is a variety, the corollary is tantamount to the
definition of incompressibility, and does not require holonomicity.

In particular, we can weaken the holonomicity criterion
of Theorem \ref{t:main2}, adding in the word ``incompressible'':
\begin{equation}
\text{no incompressible degenerate invariant closed subschemes}
\Rightarrow \text{holonomic}.
\end{equation}
For a counterexample to the converse implication, recall Example
\ref{ex:a3-inf-hol-qt}.
\begin{proof}[Proof of Theorem \ref{t:hol-bicond}]
  Since holonomic $\caD$-modules are always of finite length and their
  composition factors are intermediate extensions of local systems,
  and since in our case it is clear that any local systems must be on
  invariant subvarieties, it is immediate that (i) $\Rightarrow$ (iii)
  $\Rightarrow$ (ii); we only need to explain the formula in (iii).
  First, note that, by Kashiwara's equivalence (i.e., via the
  restriction functor of $\caD$-modules from $Z'$ to $Z$), the
  categories of $\caD$-modules on $Z'$ and on $Z$ are canonically
  equivalent.  Then, the multiplicity space $((i_* \Omega_{\hat
    Z_z})^{\mfv|_{Z'}})^*$ is explained by the canonical isomorphism
\[
(i_* \Omega_{\hat Z_z})^{\mfv|_{Z'}} \cong \Hom_{\hat \caD_{X,z}}(M(\widehat{Z'_z}, 
\mfv|_{\widehat{Z'_z}}),  i_* \Omega_{\hat Z_z}),
\]
looking at the image of the canonical generator of $M(\widehat{Z'_z},
\mfv|_{\widehat{Z'_z}})$, and viewing $\caD$-modules on $Z'$ as
$\caD$-modules on the ambient space $X$.  

So, we prove that (ii) implies (i).  Suppose (ii) holds.  We prove
holonomicity by induction on the dimension of $X$.  There is an open
dense subset $Y \subseteq X$ such that $M(X,\mfv)|_Y$ is a local
system (viewed as a $\caD$-module on $Y_\red$).  Take $Y$ to be
maximal for this property, i.e., the set-theoretic locus where
$M(X,\mfv)$ is a local system in some neighborhood.

Let $j: Y \into X$ be the open embedding.  Then by adjunction, since
$j^* M(X,\mfv) = M(X,\mfv)|_Y$ is holonomic, we obtain a canonical map
$H^0 j_! M(X, \mfv)|_Y \to M(X,\mfv)$.  The cokernel is supported on
the closed invariant subvariety $Z := X \setminus Y$, which has
strictly smaller dimension than that of $X$.  By Proposition
\ref{p:inc-supp}, the quotient factors through $M(Z',\mfv|_{Z'})$ for
some infinitesimal thickening $Z'$ of $Z$.  Then, by induction, $M(Z',
\mfv|_{Z'})$ is holonomic.  This implies the result.

The final statement follows from the inductive construction of the
previous paragraph, if we note that the image of $H^0j_! M(X,\mfv)|_Y$
is an extension of the local system $M(X,\mfv)|_Y$ by local systems on
boundary subvarieties, none of which split off the extension.
\end{proof}
Note that, by Example \ref{ex:a3-inf-hol-sub}, in general the extensions
appearing (iii) can contain composition factors supported on invariant
subvarieties which do not themselves appear in (iii).


\subsection{Global generalization}\label{ss:nonaffine}
Now, suppose that $X$ is not necessarily affine. Since $X$ does not in
general admit (enough) global vector fields, we need to generalize
$\mfv$ to a \emph{presheaf} of vector fields, i.e., a sub-presheaf of
$\bk$-vector spaces of the tangent sheaf. As we will see, even
for affine $X$, this is more natural and more flexible: for example,
even in the case of Hamiltonian vector fields, we will see that
Zariski-locally Hamiltonian vector fields need not coincide with
Hamiltonian vector fields, so that the natural presheaf $\mfv$ is not
even a sheaf, let alone constant; see Remark \ref{r:zarlh} below.

Nonetheless, all of the main examples and results 
of this paper are already interesting for affine varieties and
do not require this material, so the reader
interested only in the affine case can feel free to skip this subsection.

Let $X$ be a not necessarily affine variety and $\mfv$ a presheaf of
Lie algebras of vector fields on $X$.
For any open affine subset $U \subseteq X$, we can define the
$\caD$-module on $U$, $M(U,\mfv(U))$, as above.  Recall that this is
defined as a certain quotient of $\caD_U$. Therefore, to show that the
$M(U,\mfv(U))$ glue together to a $\caD$-module on
$X$, it suffices to check that the restriction to $U \cap U'$ of the
submodules of $\caD_U$ and $\caD_{U'}$ whose quotients are
$M(U,\mfv(U))$ and $M(U',\mfv(U'))$, respectively, are the same.  This
does \emph{not} hold in general, but it does hold
if one has the following condition:
\begin{definition} 
  Say
  that $(X,\mfv)$ is (Zariski) \emph{$\caD$-localizable}
  if, for every open affine subset $U \subseteq X$ and every open
  affine $U' \subseteq U$,
\begin{equation}\label{e:d-loc}
\mfv(U') \caD_{U'} = \mfv(U)|_{U'} \caD_{U'}.
\end{equation}
\end{definition}
\begin{remark}
If $X$ is already affine, the definition is still meaningful (and this is the
case we will primarily be interested in here). 
In this case we can restrict to $U=X$ in \eqref{e:d-loc}.
\end{remark}
\begin{example}\label{ex:const-local}
  If $X$ is irreducible and $\mfv$ is a constant sheaf, then it is
  immediate that $\mfv$ is $\caD$-localizable.  More generally, for
  reducible $X$ and $\mfv \subseteq \Vect(X)$, we can consider the
  associated presheaf $\mfv(U) := \mfv|_U$, and this is
  $\caD$-localizable.  If the irreducible components of $X$ are $X_j$,
  then the sheafification of this $\mfv$ is $\mfv(U) = \bigoplus_{j
    \mid X_j \cap U \neq \emptyset} \mfv(X_j \cap U)$.
\end{example}
We will use below the following basic
\begin{lemma}\label{l:vd-sheaf}
  Let $X$ be an affine scheme of finite type 
and $\mfv \subseteq \Vect(X)$ an arbitrary subset
  of vector fields.  Then for every affine open $U \subseteq X$, one
  has the equality of sheaves on $U$,
\[
(\mfv \cdot \caD_X)|_U = \mfv|_U \cdot \caD_U.
\]
In particular, as a sheaf, the sections of $\mfv \cdot \caD_X$ on  $U$
coincide with the global sections of $\mfv|_U \cdot \caD_U$.

Similarly, for every $x \in X$, we have  $(\mfv \cdot \caD_X)|_{\hat X_x} = \mfv|_{\hat X_x} \cdot \hat \caD_{X,x}$.
\end{lemma}
\begin{proof}
  We use \eqref{e:gs-vd}. In these terms, for $X \into V$ an embedding
  into a smooth affine variety $V$, let $U' \subseteq V$ be an affine
  open subset such that $U' \cap X = U$.  Then $(\mfv \cdot
  \caD_X)|_U$ identifies with the $\caD$-module restriction of
  \eqref{e:gs-vd} to $U'$, which is then $\mfv|_U \cdot \caD_U$.
We conclude the statements of the first paragraph.  The second
paragraph is similar.
\end{proof}

Given a presheaf $\mathcal{C}$, let $\Sh(\mathcal{C})$ be its
sheafification.  
\begin{proposition}\label{p:local}
  Suppose that $(X, \mfv)$ is $\caD$-localizable.
Then the following hold:
\begin{enumerate}
\item[(i)] The $M(U,\mfv(U))$ glue together to a
  $\caD$-module $M(X,\mfv)$ on $X$.
\item[(ii)] For every open affine $U$ and every open affine $U' \subseteq U$,
  $M(X,\mfv)|_{U'} = M(U',\mfv(U'))$.
\item[(iii)] $(X,\Sh(\mfv))$ is also $\caD$-localizable,
 and $M(X,\Sh(\mfv)) = M(X,\mfv)$.
\end{enumerate}
\end{proposition}
\begin{proof}[Proof of Proposition \ref{p:local}]
  For (i), note that \eqref{e:d-loc} applied to $U' := U \cap V$
  implies that $M(U, \mfv(U))$ and $M(V, \mfv(V))$ glue.  Then, (ii)
  is an immediate consequence of \eqref{e:d-loc}.

  It remains to prove (iii). Suppose that $U$ is an affine open, $U'
  \subseteq U$ is affine open, and $\xi \in \Sh(\mfv)(U')$.  Let $u
  \in U'$. By definition, there exists a neighborhood $U'' \subseteq
  U'$ of $u$ such that $\xi|_{U''} \in \mfv(U'')$.  By
  \eqref{e:d-loc}, $\xi|_{U''} \in \mfv(U)|_{U''} \cdot \caD_{U''}$.
  Thus, by Lemma \ref{l:vd-sheaf}, $\xi$ is a section of the
  $\caD$-module $\mfv(U)|_{U'} \cdot \caD_{U'} = (\mfv(U) \cdot
  \caD_U)|_{U'}$ on $U'$. This proves the first statement.  This also
  proves the second statement, since $U' \subseteq U$ and $\xi \in
  \Sh(\mfv)(U')$ were arbitrary.
\end{proof}
\begin{remark}\label{r:vsloc} As in Remark \ref{r:vs}, we could
  have allowed $\mfv$ to be an arbitrary presheaf of vector fields
  (rather than a sheaf of Lie algebras of vector fields).  However, it
  is easy to see that it is then $\caD$-localizable if and only if the
  presheaf of Lie algebras generated by it is, and that the resulting
  $\caD$-module is the same.  So, no generality is lost by requiring
  that $\mfv$ be a presheaf of Lie algebras.
\end{remark}
Using the above, in the $\caD$-localizable setting, the results of
this section extend to nonaffine schemes of finite type.  We omit
further details (but we will discuss $\caD$-localizability more in \S
\ref{s:ex-dloc} below).

\section{Generalizations of Cartan's simple Lie algebras}\label{s:Csla}
In this section we state and prove general results on Lie algebras
of vector fields on affine varieties which generalize the simple Lie
algebras of vector fields on affine space as classified by Cartan.
Namely, we will consider the Lie algebras of all vector fields; of
Hamiltonian vector fields on Poisson varieties; of Hamiltonian vector
fields on Jacobi varieties (this generalizes both the previous example
and the setting of contact vector fields on contact varieties); and of
Hamiltonian vector fields on varieties equipped with a top polyvector
field, or more generally equipped with a divergence function. 
The last example, which seems to not have been studied before,
generalizes the volume-preserving or divergence-free vector fields on
$\bA^n$ or on Calabi-Yau varieties.  We also consider invariants of
these Lie algebras under the actions of finite groups (we will
continue this study in \S$\!$\S \ref{s:fqcyvar} and \ref{s:sym}).

Namely, in this section we compute the leaves under the flow of these
vector fields and determine when they are holonomic, and hence their
coinvariants are finite-dimensional.

We will state all examples in the affine setting; in
\S \ref{s:ex-dloc} below we will explain how to generalize them to
the nonaffine setting (which will at least work for the cases of all
vector fields and Hamiltonian vector fields).

\subsection{The case of all vector fields}\label{ss:allvfds}
Consider the case where $\mfv$ is the Lie algebra of all vector
fields. In this case we have a basic result:
\begin{proposition}\label{p:supp-allvfds}
  The support, $Z = X_{\Vect(X)}$, of $\Vect(X)$ is the locus where
  all vector fields vanish, i.e., the scheme of the ideal
  $(\Vect(X)(\cO_X))$.  Moreover, 
\[
M(X,\Vect(X)) = \caD_Z := 
  \Vect(X)(\cO_X) \cdot \caD_X \setminus \caD_X,
  \]
and $(\cO_X)_{\Vect(X)} = \cO_Z$.
\end{proposition}
The support is evidently incompressible, and is the union of
zero-dimensional leaves at every point. Therefore, $\Vect(X)$ is
holonomic if and only if this vanishing locus is finite.
\begin{proof}
  Given $\xi \in \Vect(X)$, the submodule $\mfv \cdot \caD_X$ contains
  $[\xi, f]=\xi(f)$ for all $f \in \cO_X$. These generate the ideal
  $(\Vect(X)(\cO_X))$ over $\cO_X$, which defines the vanishing scheme of
  $\Vect(X)$.  Conversely, notice that the principal symbol of any
  product of vector fields lies in the submodule $(\Vect(X)(\cO_X))
  \cdot \caD_X$.  Thus, $\mfv \cdot \caD_X = (\Vect(X)(\cO_X)) \cdot \caD_X$.
 The last statement follows immediately.
\end{proof}
This motivates the
\begin{definition}
  A point $x \in X$ is \emph{exceptional} if all vector fields on $X$
  vanish at $x$.
\end{definition}
Clearly, all exceptional points are singular, but not conversely: for
example, if $X = Y \times Z$ where $Z$ is smooth and of purely
positive dimension, then $X$ will have no exceptional points, regardless
of how singular $Y$ is.
\begin{proposition}\label{p:allvfds}
  The following are equivalent:
\begin{enumerate}
\item[(i)] The quotient $(\cO_X)_{\Vect(X)}$ is finite-dimensional;
\item[(ii)] $X$ has finitely many exceptional points;
\item[(iii)]  $\Vect(X)$ (i.e., $M(X,\Vect(X))$) is holonomic.
\end{enumerate}
\end{proposition}
\begin{proof}
  First, (ii) and (iii) are equivalent by Proposition
  \ref{p:supp-allvfds}, since $\caD_Z$ is holonomic if and only if $Z$
  has dimension zero, i.e., set-theoretically $Z$ is finite.  By the
  proposition, with $Z$ the support of $\Vect(X)$, then $Z_\red$ is
  the locus of exceptional points of $X$ and $M(X,\Vect(X)) = \caD_Z$,
  so the equivalence follows.  Similarly, these are equivalent to (i),
  since $(\cO_X)_{\Vect(X)} = \cO_Z$.
\end{proof}
\begin{remark}
  Note that the implication (i) $\Rightarrow$ (iii) above, a converse
  to Proposition \ref{p:hol-fd}, is special to the case $\mfv =
  \Vect(X)$.  See, e.g., Remarks \ref{r:hol-fd} and \ref{r:hp0-van}.
\end{remark}
\begin{corollary}
If $X$ has a finite exceptional locus $Z \subseteq X$
(i.e., $\mfv$ is holonomic),
then 
\[
M(X, \Vect(X)) \cong \bigoplus_{z \in Z} \delta_z \otimes (\hat
\cO_{X,z})_{Vect(\hat \cO_{X, z})}.
\]
\end{corollary}
\begin{proof}
This follows immediately by formally localizing at each exceptional point.
\end{proof}
\begin{corollary}
  Under the same assumptions as in the previous corollary, if $\pi: X
  \to \pt$ is the projection to a point,
\[
\pi_*  M(X,\Vect(X)) = \pi_0 M(X,\Vect(X)) \cong \bigoplus_{z \in Z} (\hat
\cO_{Z,z})_{Vect(\hat \cO_{Z, z})}.
\]
\end{corollary}
\begin{proof}
  This follows since $\pi_* \delta_x = \pi_0 \delta_x = \bk$ for any
  point $x \in X$.
\end{proof}
\begin{example}
  Suppose that $X$ has finitely many exceptional points.  Then, the
  dual space $((\cO_X)_{\Vect(X)})^* = (\cO_X^*)^{\Vect(X)}$, of
  functionals invariant under all vector fields, includes the
  evaluation functionals at every exceptional point.  These are
  linearly independent.  However, they need not span all invariant
  functionals. In other words, the multiplicity spaces $(\hat
  \cO_{X,x})_{Vect(\hat \cO_{X, x})}$ in the corollaries need not be
  one-dimensional.

  For example, if one takes a curve $X \subset \bA^2$ of the form
  $P(x,y) + Q(x,y) = 0$ in the plane with $P(x,y)$ and $Q(x,y)$
  homogeneous of degrees $n$ and $n+1$, then we claim that, if $n \geq
  5$ and $P$ and $Q$ are generic, all vector fields on $X$ vanish to
  degree at least two at the singularity at the origin.  Therefore,
  the coinvariants $(\cO_X)_{\Vect(X)}$ have dimension at least three,
  even though $0$ is the only singularity of $X$.

  Indeed, up to scaling, any vector field which sends $P$ to a
  constant multiple of $P$ up to higher degree terms is of the form $a
  \Eu + v$, where $a \in \bk$ and $v$ vanishes up to degree at least
  two at the origin.  Suppose that such a vector field preserves the
  ideal $(P+Q)$, i.e., that it sends $P + Q$ to a multiple of
  $P+Q$. We claim that $a=0$. Otherwise, we can assume up to scaling
  that $a=1$. Then $(\Eu+v)(P+Q) = f(P+Q)$ for some polynomial $f$. By
  comparing the parts of degree $n$, we conclude that $f(0)=n$.
  Writing $f=n+bx+cy+g$, where $g$ vanishes to degree at least two at
  the origin, we conclude that $Q = (-v + (bx+cy)\Eu)P$.
  So there exists a quadratic vector field $w = -v + (bx+cy)\Eu$
  which takes $P$ to $Q$.  The space of all quadratic vector
  fields is six-dimensional, whereas the space of all possible $Q$ is
  of dimension $n+2$.  So for $n \geq 5$, we obtain a contradiction,
  since $P$ and $Q$ are assumed to be generic.

Here is an explicit example for the smallest case, $n=5$, of such a
$P$ and $Q$: Let $P=x^5+y^5$ and $Q=x^3y^3$.  Then it is clear that
the equation $Q = -v(P) + (bx+cy)P$ cannot be satisfied for any
quadratic vector field $v$ and any $b, c \in \bk$.
\end{example}
\begin{example}\label{ex:inf-exc}
  One example of a variety with infinitely many exceptional points,
  and hence infinite-dimensional $(\cO_X)_{\Vect(X)}$ and
  non-holonomic $\Vect(X)$, is a nontrivial family of affine cones of
  elliptic curves: one can take $X = \Spec \bk[x,y,z,t] / (x^3 + y^3 +
  z^3 + txyz)$, which is a family over $\bA^1=\Spec \bk[t]$ whose
  fibers are affine cones of elliptic curves in $\bP^2$. Then, we
  claim that all singular points $x=y=z=0$ are exceptional.

  The proof is as follows:  Take any vector field on $X$ and lift it
to a vector field $\xi$ on $\bA^4$ parallel to $X$.  Then
$\xi(x^3 + y^3 + z^3 + txyz) = f (x^3+y^3+z^3+txyz)$ for some 
  $f \in \cO_{\bA^4}$.  Replacing $\xi$ by $\xi -  (1/3)f \cdot
  (x \partial_x + y \partial_y + z\partial_z)$, we can assume that
  $f = 0$.  Restricting to $t=t_0$, we obtain
\begin{equation}\label{e:xieqn}
\xi|_{t=t_0} (x^3 + y^3 + z^3 + t_0 xyz) = - \xi(t)|_{t=t_0}\cdot xyz.
\end{equation}
Suppose that $\xi$ did not vanish at $(0,0,0,t_0)$.  We can assume
$\xi$ is homogeneous with respect to the grading $|x|=|y|=|z|=1$ and
$|t|=0$.  Then $\xi|_{t=t_0}$ is either constant or linear. By
\eqref{e:xieqn}, $\xi|_{t=t_0}$ annihilates $x^3+y^3+z^3$, but no
constant or linear vector field can do that, which is a contradiction.

\end{example}

\subsection{The Poisson case}\label{ss:poisson}
Suppose that $X$ is an affine Poisson scheme of finite type, i.e.,
$\cO_X$ is a Poisson algebra.  Let $\pi$ be the Poisson bivector field
on $X$. Then, we can let $\mfv$ be the Lie algebra of Hamiltonian
vector fields, $H(X)=H_\pi(X)$.  In particular, these vector fields
are $\xi_f := \pi(df)$ for $f \in \cO_X$. In this case, $(\cO_X)_\mfv =
\HP_0(\cO_X)$, the zeroth Poisson homology of $\cO_X$. 
As pointed out in Example \ref{ex:poiss-finsym}, $H(X)$ is holonomic if
and only if $X$ has finitely many symplectic leaves. 

There are several natural larger Lie algebras to consider than $H(X)$.
Note that $H(X)$ is the space of vector fields obtained by contracting
$\pi$ with exact one-forms.  So, one can consider instead
$LH(X)=LH_\pi(X) = \pi(\tilde \Omega^1_X)$, the space of vector fields
obtained by contracting $\pi$ with \emph{closed} one-forms modulo
torsion (note that contracting $\pi$ with torsion yields zero, since
$\cO_X$ is torsion-free).  Here we will denote the resulting vector
field by $\eta_\alpha := \pi(\alpha)$. Thus, when $X$ is generically
symplectic, $LH(U)/H(U) \cong H^1_{DR}(U)$ for all open affine $U
\subseteq X$.  (Recall from the beginning of \S \ref{s:gen} that, over
$\bk=\bC$, if $U$ is smooth, this coincides with the first topological
cohomology of $U$).

Here, $LH$ stands for ``locally Hamiltonian;'' in a smooth affine open
subset, in the case that $\bk=\bC$, these are the vector fields which
are locally Hamiltonian in the analytic topology. In general, in a
smooth open subset, these are the vector fields which, restricted to a
formal neighborhood of a point, are Hamiltonian.  However, as
explained in the next example, in formal neighborhoods of singular points
not all locally Hamiltonian vector fields are Hamiltonian:
\begin{example}\label{ex:fn-sing-poiss}
  In the formal neighborhood of singular points, locally Hamiltonian
  vector fields need not be Hamiltonian, since the first de Rham
  cohomology modulo torsion need not vanish in such a neighborhood,
  and as mentioned above, when $X$ is generically symplectic, then
  $LH(X)/H(X) \cong H^1(\tilde \Omega_X^\bullet)$.

  Here is an example where this cohomology does not vanish.
  Suppose $Z \subseteq \bA^n$ is a complete
  intersection with an isolated singularity at $z \in Z$. By
  \eqref{e:gre-fla} below,  in this case $\tilde
  \Omega_{Z,z}^\bullet$ is acyclic except in degree $k=\dim Z$, where
  $\dim H^k(\tilde \Omega_{Z,z}^\bullet) = \mu_z - \tau_z$, where
  $\mu_z$ and $\tau_z$ are the Milnor and Tjurina numbers of $z$ (see
  \S \ref{s:cy-cplte-int-is} below; we will not use the general
  definition here).  In the case when $Z \subseteq \bA^2$ is a reduced
  curve cut out by $Q \in \bk[x,y]$ with an isolated singularity at
  the origin, then all one-forms modulo torsion are closed, but they
  are not all exact in general. Explicitly, $H^1(\tilde
  \Omega_{Z,0}^\bullet) \cong (Q, \partial_x Q, \partial_y Q)_0 /
  (\partial_x Q, \partial_y Q)_0$, where $(-)_0 \subseteq \hat
  \cO_{\bA^2,0}$ is the ideal in the completed local ring at the
  origin.

Specifically, take $Q = x^3 + x^2 y + y^4$, where
\begin{multline*}
(Q,\partial_x Q, \partial_y Q) = (3x^2+2xy,x^2+4y^3,x^3+x^2y+y^4)=
(3x^2+2xy,x^2+4y^3,y^4) \\ \neq (3x^2+2xy,x^2+4y^3)=(\partial_x
Q,\partial_y Q).
\end{multline*}
One therefore obtains a nonexact (closed) one-form.  Such a form is
$\alpha := x \cdot dy$: one can compute that
\[
\alpha \wedge dQ = (-3x^3 -2x^2 y) \cdot dx \wedge dy \equiv 2y^4
\cdot dx \wedge dy \pmod { d \bk[\![x,y]\!] \wedge dQ + (Q) dx
\wedge dy + (x,y)^5 dx \wedge dy },
\]
and this is not equivalent to zero modulo $d \bk[\![x,y]\!] \wedge dQ
+ (Q) dx \wedge dy + (x,y)^5 dx \wedge dy$.

Then, consider the Poisson variety $X = Z \times \bA^1$ with the
Poisson structure $(\partial_x \wedge \partial_y)(dQ)
\wedge \partial_t$, with $t$ the coordinate on $\bA^1$. This is
generically symplectic, so provides an example where $LH(X) \neq
H(X)$.  Specifically, the vector field $\eta_\alpha =
(-3x^3-2x^2y) \partial_t$ is locally Hamiltonian on $X$, but in the
formal neighborhood of the origin it is not Hamiltonian.  By the above
computation, this spans $LH(\hat X_0) / H(\hat X_0)$.
%
%
\end{example}

Note that the fact that $LH(X)$ and $H(X)$ are Lie algebras follow
from the fact that $[LH(X), LH(X)] \subseteq H(X)$, since
$\{\eta_\alpha, \eta_\beta\} = \xi_{i_{\eta_\alpha} \beta}$ for closed
one-forms $\alpha$ and $\beta$. 


Next, one can consider $P(X)=P_\pi(X)$, the space of all Poisson
vector fields, i.e., those $\xi$ such that $L_\xi(\pi)=0$.  Clearly,
we have $H(X) \subseteq LH(X) \subseteq P(X)$.  If $X$ is symplectic
(which for us in particular means $X$ is smooth), then it is
well-known that $LH(X) = P(X)$, but this may not be true in general
(even if $X$ has finitely many symplectic leaves: see Example
\ref{ex:poisson-not-ham}).  However, there is a certain generalization
of this equality to the mildly singular case, as explained in the next
remark.
\begin{remark}\label{r:n-cod2-poiss}
In the case that $X$ is normal and generically symplectic, then the
  following conditions are equivalent:
\begin{itemize}
\item[(i)] $X$ is symplectic on its smooth locus;
\item[(ii)] On each irreducible component,
$X$ is symplectic outside of a codimension two subset.
\end{itemize}
This is because the degeneracy locus of a Poisson structure is given
by a single equation $\pi^{\wedge \lceil\dim Y/2\rceil} = 0$, so on
the smooth locus this consists of divisors (if it is generically
nondegenerate).

If we assume that either of these conditions is satisfied, then
letting $X^\circ \subseteq X$ be the smooth locus (which is not affine
unless $X=X^\circ$) we claim that $P(X) = P(X^\circ) = LH(X^\circ)$,
where here by $P(X^\circ)$ we mean global Poisson vector fields on the
nonaffine $X^\circ$, and by $LH(X^\circ)$ we mean the collection of
vector fields $\eta_\alpha$ for $\alpha \in \Gamma(X^\circ,
\Omega_{X^{\circ}})$ a closed one-form regular on $X^\circ$.

Indeed, in this case, all vector fields which are regular on $X^\circ$
extend to all of $X$. Thus $P(X)=P(X^\circ)$.  Moreover, if $\xi \in
P(X)$ is a global Poisson vector field, then dividing by the Poisson
bivector, we obtain a closed one-form regular on $X^\circ$, and
conversely.
\end{remark}


The leaves of $X$ under both $H(X)$ and $LH(X)$ are the symplectic
leaves.  For $H(X)$, this is the definition of symplectic leaves; for
$LH(X)$, this is true because, since all one-forms (and in particular
all closed one forms) are spanned over $\cO_X$ by exact one-forms, the
evaluations at each point of the contraction of $\pi$ with either span
the same subspace of the tangent space.  That is, $H(X)|_x = LH(X)|_x$
for all $x \in X$, as subspaces of $T_x X$.  In fact, $H(X)$ and
$LH(X)$ define the same $\caD$-module, since they have the same
$\cO$-saturation, as defined in \S \ref{ss:supp}:
\begin{proposition} \label{p:poiss-hlh} The $\cO$-saturations are
  equal: $H(X)^{os} = LH(X)^{os}$.  Hence, $M(X,H(X)) \cong
  M(X,LH(X))$.
\end{proposition}
\begin{proof}
  Given any closed one-form $\alpha := \sum_i f_i dg_i \in T_X^*$, for
  $f_i, g_i \in \cO_X$, we claim that $\eta_\alpha = \sum_i \xi_{g_i}
  \cdot f_i$.  This follows because $\sum_i [\xi_{g_i}, f_i] = \sum_i
  \xi_{g_i}(f_i) = \pi(d\alpha) = 0$.  Hence $LH(X) \cdot \cO_X
  \subseteq H(X) \cdot \cO_X$. For the opposite inclusion, note that
  $H(X) \subseteq LH(X)$.
\end{proof}
In the case that $X$ has finitely many symplectic leaves, then $P(X)$
also has these as its leaves, since in this case every Poisson vector
field must be parallel to the symplectic leaves.  On the other hand,
it can happen that $P(X)$ has finitely many leaves but not $LH(X)$:
\begin{example}
If
$\pi = x \partial_x \wedge \partial_y$ on $\bA^2$, then there are
infinitely many symplectic leaves: the $y$-axis is a degenerate
invariant subvariety with respect to $LH(X)$. On the other hand, the
vector field $\partial_y$ is Poisson, so the $y$-axis is a leaf with
respect to $P(X)$.  
\end{example}
For $LH(X)$, the same argument as for $H(X)$ shows that, in the
notation of Proposition \ref{p:decomp}, all of the $X_i$ are
incompressible, and hence $LH(X)$ is holonomic if and only if it has
finitely many leaves (the symplectic leaves); or one can use
Proposition \ref{p:poiss-hlh}.  So, again, Theorem \ref{t:main2} is
the same as Theorem \ref{t:main1-alt}.

On the other hand, it can happen that $P(X)$ is holonomic even though
it does not have finitely many leaves:
\begin{example}
  If $X$ is a variety equipped with the zero Poisson structure, then
  $P(X)$ is the Lie algebra of all vector fields, and as explained in
  \S \ref{ss:allvfds}, this is holonomic if and only if there are
  finitely many exceptional points.  This can happen without having
  finitely many leaves, e.g., if one takes a product $X= \bA^1 \times
  Y$ where $Y$ has infinitely many exceptional points (cf.~Example
  \ref{ex:inf-exc}). Moreover, this is an example where the $X_i$ are
  not incompressible (if $x$ is the coordinate on $\bA^1$, $P(X)$
  contains both $\partial_x$ and $x \partial_x$, so cannot be
  incompressible on any of the $X_i = \bA^1 \times Y_{i-1}$).
\end{example}

\begin{example}\label{ex:poiss-cplte-int}
  If $Y$ is an $n$-dimensional Calabi-Yau variety (e.g., $Y = \bA^n$)
  and $X = Z(f_1, \ldots, f_{n-2}) \subseteq Y$ is a surface which is
  a complete intersection $f_1=\cdots=f_{n-2}=0$, then there is a
  standard \emph{Jacobian} Poisson structure on $X$, given by
  $i_{\Xi}df_1 \wedge \cdots \wedge df_{n-2}$, where $\Xi =
  \vol_Y^{-1}$ is the inverse to the volume form on $Y$, which we then
  contract with the exact $n-2$-form $df_1 \wedge \cdots \wedge
  df_{n-2}$. It is then standard that the result is a Poisson bivector
  field.  Then $H(X)$ is holonomic if and only if $X$ has only
  isolated singularities.  Already in the case $Y = \bA^3$ and $X =
  Z(f)$ for $f$ a (quasi)homogeneous surface with an isolated
  singularity at zero, this is quite interesting;
  $\HP_0(\cO_X)=(\cO_X)_{H(X)}$ was computed in \cite{AL} (although,
  as we will explain in \S \ref{s:cy-cplte-int-is}, it follows from
  older results of \cite{Gre-GMZ}); we plan to compute $M(X,H(X))$ in
  \cite{ES-ciiss}. See Example \ref{ex:cy-cplte-int} and \S
  \ref{s:cy-cplte-int-is}.
\end{example}
\begin{example}\label{ex:poiss-prod}
  If $X$ and $Y$ are Poisson schemes of finite type, then for any of
  the three Lie algebras defined above, the coinvariants are
  multiplicative in the sense that $(\cO_{X \times Y})_{H(X \times Y)}
  = (\cO_X)_{H(X)} \otimes (\cO_Y)_{H(Y)}$ and similarly for $LH$ and
  $P$.  Similarly, the leaves of $X \times Y$ are the products of
  leaves from $X$ and of leaves from $Y$.  These facts follow from the
  following formula, which also holds for $LH$ and $P$ replacing $H$:
\begin{equation}\label{e:prod-vfds}
H(X) \oplus H(Y) \subseteq
H(X \times Y) 
\subseteq
 (\cO_X \boxtimes H(Y)) \oplus (H(X) \boxtimes \cO_Y).
\end{equation}
The first inclusion holds because, for $f \in \cO_X$ and $g \in
\cO_Y$, $\xi_{(f \otimes 1)+ (1 \otimes g)}=\xi_f+\xi_g$.  The second
follows because, for $f \in \cO_X$ and $g \in \cO_Y$,
$\xi_{f \otimes g}(h) = f \xi_g + g \xi_f$.
To extend \eqref{e:prod-vfds} to the case of $LH(X \times Y)$, it
remains only to consider also the action of Hamiltonian vector fields
of closed one-forms modulo torsion generating $H^1_{DR}(X \times Y) =
H^1_{DR}(X) \oplus H^1_{DR}(Y)$ (assuming for simplicity that $X$ and
$Y$ are connected). So it suffices to consider Hamiltonian vector
fields of closed one-forms modulo torsion on $X$ and $Y$
separately. One concludes that \eqref{e:prod-vfds} holds for $LH$
replacing $H$.  Finally, for $P(X \times Y)$, one also has
\eqref{e:prod-vfds} with $P$ replacing $H$, since $\pi_{X \times Y} =
\pi_X \oplus \pi_Y$ and $\Vect(X \times Y) = (\Vect(X) \boxtimes
\cO_Y) \oplus (\cO_X \boxtimes \Vect(Y))$.
\end{example}
\begin{example}\label{ex:poisson-not-ham}
  Here we give an example of a variety $X$ with finitely many
  symplectic leaves for which $LH(X) \subsetneq P(X)$.  Namely,
  suppose $X$ is a homogeneous cubic hypersurface, $Q=0$,
  in $\bA^3$ with an isolated
  singularity at the origin, i.e., the cone over a smooth curve of
  genus one.  Then $X$ is equipped with the Poisson
  bivector given by contracting the top polyvector field $\partial_x
  \wedge \partial_y \wedge \partial_z$ on $\bA^3$ with $dQ$, where $x,
  y$, and $z$ are the coordinate functions on $\bA^3$.  This has two
  symplectic leaves: the origin and its complement.

  We claim that the Euler vector field is Poisson but not locally
  Hamiltonian.  This is because the Poisson bracket preserves total
  degree, so the Euler vector field is Poisson, but it cannot be
  Hamiltonian since the Poisson bivector vanishes to degree two at the
  origin, i.e., $\pi(df \wedge dg) \subseteq \mathfrak{m}_0^2$ for all
  $f,g \in \cO_{X}$, with $\mathfrak{m}_0$ the maximal ideal of
  functions vanishing at the origin.  Hence all Hamiltonian vector
  fields vanish to degree two at the origin as well.

  For example, $X$ could be the hypersurface $x^3+y^3+z^3 =0$, which
  is the cone over the Fermat curve. Then $\{x,y\}=3z^2$, $\{y,z\} =
  3x^2$, and $\{z,x\} = 3y^2$, and it is clear that the Euler vector
  field is Poisson but not (locally) Hamiltonian.
\end{example}
\begin{remark}\label{r:hp0-van}
  We note that, unlike for all vector fields, the converse to
  Proposition \ref{p:hol-fd} does not hold in the Poisson case.
  Indeed, one can consider $\bA^3$ with the Poisson structure
  $\partial_x \wedge \partial_y$, which has infinitely many leaves
  (hence is not holonomic) but vanishing $\HP_0$.
\end{remark}
 
Finally, if $X$ is an affine Poisson scheme of finite type with
finitely many symplectic leaves, and $f: X \to Y$ is a finite map,
then the argument of \cite{ESdm}  showed that the Lie algebra of
Hamiltonian vector fields of Hamiltonian functions from $f^*\cO_Y$ has
finitely many leaves. We recover the result from \emph{op.~cit.} that
$\cO_X / \{\cO_X, \cO_Y\}$ is finite-dimensional. This includes the
case, for example, where $X = V$ is a symplectic vector space, and $Y
= V/G$ for $G < \Sp(V)$ a finite subgroup (or even any finite subgroup
$G < \GL(V)$).  If $G < \Sp(V)$ then we obtain the $G$-invariant
Hamiltonian vector fields, $H(X)^G$.  Note that, in this case, if
$q: X \to X/G$ is the projection, then $q_* M(X,H(X)^G)^G \cong M(X/G,H(X/G))$.

 \subsection{Jacobi schemes}\label{ss:jacobi}
 A Jacobi structure \cite{Lic-jac} is a generalization of a
 Poisson structure, which includes both symplectic and contact
 manifolds (see the examples below), and can be thought of as a
 degenerate or singular version of both.  By definition, it is a Lie bracket on $\cO_X$ which need not satisfy the Leibniz rule, but instead satisfies that $\{f, -\}$ is a differential operator of order $\leq 1$ for all $f \in \cO_X$. Equivalently, the Lie bracket is given by
a pair
 of a bivector field $\pi$ and a vector field $u$ via the
 formula
\[
\{f,g\} = \pi(df \wedge dg) + u(fdg - gdf).
\]
Here, by a  degree $k$ polyvector field, we mean a
skew-symmetric multiderivation of $\cO_X$ of degree $k$, i.e., a
linear map $\wedge^k \cO_X \to \cO_X$ which is a derivation in each
component.

The Jacobi identity is then equivalent to the identities
\[
[u, \pi] = 0, \quad [\pi, \pi] = 2u \wedge \pi,
\]
where $[-,-]$ is the Schouten-Nijenhuis bracket on polyvector
fields. 

To any affine Jacobi scheme $X$ of finite type, one naturally associates the
Lie algebra of Hamiltonian vector fields $\xi_f$ for $f \in \cO_X$,
given by the principal symbol of the differential operator $\{f, -\}$,
i.e.,
\[
\xi_f = \pi(df) + f u, \quad \text{i.e.,} \quad \xi_f(g) = \{f,g\} + g
u(f).
\]
It is well-known and easy to verify that one has the identity
\[
[\xi_f, \xi_g] = \xi_{\{f,g\}},
\]
so this indeed forms a Lie algebra. Call it $H(X) := H_{\pi,u}(X)$.


We can also define a version $P(X) := P_{\pi,u}(X)$ of vector fields
\emph{preserving} the Jacobi structure, i.e., vector fields $\xi$ such
that $\xi(\{f,g\}) = \{\xi(f),g\} + \{f, \xi(g)\} = 0$ for all $f,g
\in \cO_X$. However, unlike before, it is no longer true that $H(X)
\subseteq P(X)$.  In particular, to have $\xi_f \in P(X)$, we require
that $[u,\xi_f]=\xi_{u(f)} = 0$.  So to have $H(X) \subseteq P(X)$, we
would need to have $u=0$, i.e., the structure has to be Poisson.
\begin{remark}
  It seems that we cannot define an analogue of $LH(X)$ in this
  setting since there is no way to obtain Hamiltonian vector fields
  from closed one-forms. In a neighborhood of a smooth point, one
  could consider vector fields that restrict in a formal neighborhood
  of the point to a Hamiltonian vector field, but in general this will
  not coincide with the definition of $LH(X)$ in the Poisson case, in
  neighborhoods of singular points where the first de Rham cohomology
  does not vanish in the formal neighborhood; see Example
  \ref{ex:fn-sing-poiss}.
\end{remark}
\begin{remark}\label{r:jac-prod}  
  Unlike the Poisson case, given Jacobi varieties $X$ and $Y$, there
  is no natural way to define a Jacobi structure on the product $X
  \times Y$: if one set $\pi_{X \times Y} = \pi_X \oplus \pi_Y$ and
  $u_{X \times Y} = u_X \oplus u_Y$, then the identity $[\pi,\pi] = 2u
  \wedge \pi$ would no longer be satisfied: $\pi_X \wedge u_Y$ and
  $\pi_Y \wedge u_X$ would appear on the RHS but not the LHS.
  However, one can still equip $X \times Y$ with the Lie algebra of
  vector fields $\mfv_X \oplus \mfv_Y$; in this general situation
  (i.e., for any $\mfv_X$ and $\mfv_Y$), one always has $(\cO_{X
    \times Y})_{\mfv_X \oplus \mfv_Y} \cong (\cO_X)_{\mfv_X} \otimes
  (\cO_Y)_{\mfv_Y}$.
\end{remark}
\begin{example} \label{ex:j-tr} The analogue of symplectic varieties
  in this setting is a smooth Jacobi variety for which $H(X)$ has full
  rank everywhere, i.e., it has only one leaf (assuming $X$ is
  connected). This is called a \emph{transitive} Jacobi variety.

  As pointed out in, e.g., \cite{MS-gm} (this is in the smooth
  context, but the result is proved using a formal neighborhood and
  works in general), there are two types of connected transitive
  varieties.  One is called \emph{locally conformally symplectic}, and
  is the situation where $\pi$ is nondegenerate (recall we assumed $X$
  was smooth). Therefore, $X$ is even-dimensional. In this case, $u$
  is equivalent to the data of a closed one-form $\phi$ satisfying $d
  \omega = \phi \wedge \omega$, where $\omega$ is the inverse of
  $\pi$, and $\phi = u(\omega)$. Then, in the formal neighborhood of
  any point $x \in X$, we can write $\phi = df$ for some function $f$,
  and then $H(X)$ preserves the formal volume form
  $(e^{-f}\omega)^{\wedge \dim X}$ (cf.~Example \ref{ex:st-lcs}
  below).  This need not be a global volume form, so $M(X,H(X))$ is a
  rank-one local system which need not be trivial.
 
  The other type of transitive Jacobi variety is an odd-dimensional
  contact variety.  In this case, the Jacobi structure is equivalent
  to the structure of a \emph{contact one-form} $\alpha$, i.e., a
  one-form such that $\vol_X := \alpha \wedge (d\alpha)^{\wedge (\dim X-1)/2}$
  is a nonvanishing volume form.  This determines $u$ and $\pi$
  uniquely by the formulas
\[
u(d\alpha)=0, u(\alpha) = 1, \quad \pi(\alpha, \beta) = 0, \quad \pi(\beta \wedge d\alpha) = -\beta + u(\beta)\alpha, \forall \beta \in T_X^*.
\]
By the next example, in this case $\mfv$ does not flow incompressibly, so
by Proposition \ref{p:tr}, $M(X,H(X)) = 0$. 
On the other hand, we will
see that $P(X)$ does flow incompressibly and transitively, preserving the
volume form $\vol_X$, 
 so $M(X,P(X)) = \Omega_X$ and $\pi_*M(X,P(X)) \cong H_{DR}^{\dim X - *}(X)$. In particular, $(\cO_X)_{P(X)} = H_{DR}^{\dim X}(X)$.
\end{example}
\begin{example}\label{ex:st-cont}
  The standard example of a contact variety is $\bA^{2d+1}$ with the
  standard contact structure, $\alpha = dt + \sum_i x_i dy_i$.  Also,
  note that an arbitrary contact variety restricts to one isomorphic
  to this in the formal neighborhood of any point.  We claim that no
  volume form is preserved by $H(\bA^{2d+1})$, and hence the flow of
  $H(X)$ on an arbitrary contact variety is not incompressible.
  Indeed, let $\Eu$ be the weighted Euler vector field on $\bA^{2d+1}$
  assigning weights $|x_i|=1=|y_i|$ and $|t|=2$, i.e., $\Eu = 2t
  \frac{\partial}{\partial t} + \sum_i x_i \frac{\partial}{\partial
    x_i} + y_i \frac{\partial}{\partial y_i}$.  Then, we have $\pi = 
  -\sum_i \frac{\partial}{\partial x_i} \wedge
  (\frac{\partial}{\partial y_i} - x_i \frac{\partial}{\partial t})$
  and $u = \frac{\partial}{\partial t}$.  In this case, $\xi_1 =
  \frac{\partial}{\partial t}$, $\xi_{x_i} = -\frac{\partial}{\partial
    y_i} +  x_i \frac{\partial}{\partial t}$, $\xi_{y_i} =
  \frac{\partial}{\partial x_i}$, and $\xi_{t} = - \sum_i x_i \frac{\partial}{\partial x_i}$.  In
  particular, the Lie algebra $H(X)$ does not preserve any volume form
  (if it did, for this form to be preserved by $\xi_1, \xi_{x_i}$, and
  $\xi_{y_i}$, it would have to preserve the constant vector fields,
  and hence the form would have to be the standard volume form, i.e.,
  the one determined by the contact structure; however this form is
  not preserved by $\xi_t$.)

  Finally, note that, in the above case, $P(\bA^{2d+1})$, the Lie
  algebra of all vector fields that commute with both $\pi$ and $u$,
  is the subspace of Hamiltonian vector fields $\xi_f$ where $f$ is
  independent of $t$. So $P(\bA^{2d+1}) \subsetneq H(\bA^{2d+1})$.
  This still flows transitively, since it includes the constant vector
  fields as above.  As a result, for arbitrary odd-dimensional contact
  varieties, $P(X) \subsetneq H(X)$.  In fact, $P(X)$ does flow
  incompressibly, since it preserves the standard volume form (it is
  clear that it preserves the inverse top polyvector field,
  $\pm \pi^{\wedge (\dim X - 1)/2} \wedge u$).
\end{example}
\begin{example}\label{ex:st-lcs}
  By the Darboux theorem, every locally conformally symplectic variety
  $X$ of dimension $2d$ has the form, in a formal neighborhood of a
  point $x \in X$, $\omega = e^f \omega_0$ and $\phi = df$, where
  $\omega_0$ is the standard symplectic form on $\hat \bA^{2d} \cong
  \hat X_x$. In this case $\pi = e^{-f} \pi_0$ where $\pi_0$ is the
  standard Poisson bivector on $\bA^{2d}$, and $u=\pi(df)$ is $e^{-f}$
  times the Hamiltonian vector field of $f$ under the standard
  symplectic structure.  Thus, $H(\hat X_x)$ is identical with
  the Lie algebra of Hamiltonian vector fields preserving the standard
  symplectic form $\omega_0$ (in this formal neighborhood), so it
  flows incompressibly.
  However, as noted above, $H(\hat X_x) \not \subseteq P(\hat X_x)$.
  In fact, in this case, as in the case of odd-dimensional contact
  varieties, $P(\hat X_x) \subsetneq H(\hat X_x)$. Indeed, $P(\hat
  X_x)$ consists of $\xi_g$ such that $u(g)=0$, i.e., $\{f,g\}=0$.
\end{example}
We see as a consequence of the above that, in general, the leaves of $H(X)$ 
consist of odd-dimensional contact varieties and locally conformally
symplectic varieties.  The former are not incompressible (without
passing to an infinitesimal neighborhood), whereas the latter are.  As
a consequence, we conclude from Corollary \ref{c:inc} that
\begin{proposition}
  Let $X$ be a Jacobi variety. Then $X = X_{H(X)}$ if and only if the
  generic rank of $H(X)$ is even on each irreducible component.
\end{proposition}
(Recall from Definition \ref{d:supp} that $X_{H(X)}$ is the support of
$M(X,H(X))$ on $X$.)
\begin{example}
  Here is an example of a Jacobi variety where there is an
  odd-dimensional leaf having an infinitesimal neighborhood which is
  incompressible.  Let $X = \bA^2$ with $\pi = -x \partial_x
  \wedge \partial_t$ and $u = \partial_t$.  Then $H(X)$ has rank two
  except along $x=0$, where it has rank one.  Moreover, the
  distribution $\phi := \partial_x(\delta_{x=0}) \boxtimes dt$ is
  preserved by $H(X)$: for $\xi_{x^it^j}$ with $i\geq 2$ this clearly
  annihilates $\phi$; then $\xi_{xt^j} = jx^2t^{j-1} \partial_x$ and
  $\xi_{t^j} = jxt^{j-1} \partial_x + t^j \partial_t$ also do (recall
  that the action of differential operators on distributions is a
  right action; the action of vector fields is given by $(\psi \cdot
  \xi)(f) = \psi(\xi(f))$ for $\psi$ a distribution and $f$ a
  function).  The final vector field, $\xi_{t^j}$, can alternatively
  be rewritten in $H(X) \cdot \cO_X$ as
\[
\xi_{t^j} = j(x\partial_x - 1) t^{j-1} + \partial_t \cdot t^j,
\]
and note that $x\partial_x - 1$ and $\partial_t$ annihilate $\phi$, which
implied that  $\xi_{t^j}$ does.
\end{example}
\begin{question}
Let $X_\even$ be
the closure of the locus where the rank of $\mfv$ is even.  Then, is
the set-theoretic support, $(X_{H(X)})_\red$, of $(X,H(X))$ equal to
$X_\even$?  If the answer is negative, is there an example where
$H(X)$ has everywhere odd rank, but $M(X,H(X)) \neq 0$?
\end{question}

  \subsection{Varieties with a top polyvector field}\label{ss:vtop}
  Motivated by the idea that a Poisson structure is a singular and/or
  degenerate generalization of a symplectic structure, we define a
  similar analogue of Calabi-Yau structures, and their associated Lie
  algebras of incompressible vector fields.  These are also motivated
  by the relationship between incompressibility and holonomicity.

  In the Poisson case, one replaces a nondegenerate two-form by a
  possibly degenerate two-bivector, which in the nondegenerate case is
  inverse to the symplectic form.  Thus, by analogy, we replace a
  volume form by a top polyvector field, which is allowed to vanish on
  some locus. On the nondegenerate, smooth locus, one recovers a
  symplectic form by taking the inverse of the polyvector field.

  Specifically, let $X$ be an affine variety of dimension $n$ equipped
  with a global top polyvector field, i.e., a multiderivation $\Xi:
  \wedge^{n} \cO_X \to \cO_X$. 
  Then, as in the Poisson case, there are three natural Lie algebras
  to consider: the Lie algebra $H_\Xi(X)$ of vector fields obtained by
  contracting $\Xi$ with exact $(n-1)$-forms; the Lie algebra
  $LH_\Xi(X)$ of vector fields obtained by contracting $\Xi$ with
  closed $(n-1)$-forms; and the Lie algebra $P_\Xi(X)$ of all
  incompressible vector fields, i.e., vector fields $\xi$ such that
  $L_\xi(\Xi) = 0$ (vector fields \emph{preserving} $\Xi$). Note that,
  in this case, when $X$ is irreducible and $\Xi$ is nonzero, it is
  immediate that all three flow incompressibly on $X$.
  \begin{example}\label{ex:cy2} As in Example \ref{ex:cy}, if $X$ is
    affine Calabi-Yau and $\Xi$ is the inverse of the volume form,
    then all three Lie algebras coincide and equal the Lie algebra of
    volume-preserving vector fields.  Then $M(X,H(X)) = \Omega_X$ and
    $\pi_* M(X,H(X)) \cong H_{DR}^{\dim X - *}(X)$.
\end{example}
  As for generically symplectic varieties with their associated
  (locally) Hamiltonian vector fields, for arbitrary irreducible
  $(X,\Xi)$ with $\Xi \neq 0$, one has $LH_\Xi(U) / H_\Xi(U) \cong
  H^{\dim X - 1}_{DR}(U)$ for all open affine $U \subseteq X$.
  Moreover, when $U$ is additionally smooth, $LH_\Xi(X)$ coincides
  with those vector fields which, in formal neighborhoods of all $x
  \in U$, are Hamiltonian.
\begin{remark}
  As in the Poisson case (see Example \ref{ex:fn-sing-poiss}), in the
  formal neighborhood of a singular point $x \in X$, not all locally
  Hamiltonian vector fields need be Hamiltonian, since $H^{\dim X -
    1}_{DR}(\hat X_x)$ need not vanish.  Indeed, as in Example
  \ref{ex:fn-sing-poiss}, when $X = \bA^1 \times Z$ where $Z$ is a
  complete intersection with an isolated singularity at $z \in Z$,
  then $\dim H^{\dim X - 1}_{DR}(\hat X_{(t,z)}) = \mu_z-\tau_z$,
  which need not be zero (already for the case of a hypersurface in
  $\bA^n$). Then, equipped with the polyvector field $\Xi_{\bA^1}
  \boxtimes \Xi_Z$ where $\Xi_Z$ is as in Example
  \ref{ex:cy-cplte-int} (which in the case $Z = \{Q=0\} \subseteq
  \bA^n$ is $\Xi_{\bA^n}(dQ)$), one concludes that $LH(\hat X_{(t,z)})
  / H(\hat X_{(t,z)}) \cong H^{\dim X -1}(\hat X_{(t,z)}) \neq 0$.
\end{remark}
As in the Poisson case, these are Lie algebras since $[LH(X), LH(X)]
\subseteq H(X)$, as we explain.  Given a $(n-2)$-form (modulo torsion)
$\alpha \in \tilde \Omega^{n-2}_X$, let $\xi_{\alpha} := \Xi(d\alpha)$
be its associated Hamiltonian vector field. Similarly, given a closed
$(n-1)$-form modulo torsion, $\gamma \in \tilde \Omega^{n-1}_X$, let
$\eta_\gamma := \Xi(\gamma)$ be its associated locally Hamiltonian
vector field.  Then the fact that $[LH(X), LH(X)] \subseteq H(X)$
follows from the formula, where $\alpha$ and $\beta$ are closed
$(n-1)$-forms modulo torsion,
\begin{equation}\label{e:toppol-lie}
[\eta_\alpha, \eta_\beta] = \xi_{i_{\eta_\alpha}(\beta)},
\end{equation}
which can be verified in a formal neighborhood of a smooth point of
$X$ where $\Xi$ is nonvanishing, and hence which holds globally.

As in the Poisson case (Proposition \ref{p:poiss-hlh}), $H(X)$ and
$LH(X)$ define the same $\caD$-modules on $X$:
\begin{proposition}\label{p:vtop-hlh} The $\cO$-saturations are equal:
  $H(X)^{os} = LH(X)^{os}$. Thus, $M(X,H(X))
  \cong M(X,LH(X))$.
\end{proposition}
\begin{proof}
  Given a closed $n-1$ form $\alpha = \sum_i f_i d \beta_i$, we see
  that $\eta_\alpha = \sum_i \eta_\beta \cdot f_i$, since $\sum_i
  \eta_\beta(f_i) = \Xi(d\alpha) = 0$. Thus, $LH(X) \cdot \cO_X
  \subseteq H(X) \cdot \cO_X$, and the proposition follows since $H(X)
  \subseteq LH(X)$.
\end{proof}
Next, we compute the leaves of $H_\Xi(X)$ and of $LH_\Xi(X)$.  All
non-open leaves turn out to be points.  We will use a general
\begin{definition}
  Given a Lie algebra of vector fields $\mfv$ on $X$, the
  \emph{degenerate locus} of $\mfv$ is the locus of $x \in X$ such
  that $\mfv|_x \neq T_x X_\red$.
\end{definition}
Note that the degenerate locus includes the singular locus of $X_\red$
(which equals $X$ in this subsection, although the preceding
definition makes sense more generally).
\begin{remark}
  If $X$ is irreducible, then we claim that the degenerate locus is
  the same as the locus of $x$ such that $\dim \mfv|_x < \dim X$,
  i.e., such that $\mfv$ does not have maximal rank.  Thus, in terms
  of Proposition \ref{p:decomp}, the degenerate locus is the union of
  $X_i$ for $i < \dim X$. To prove the claim, we only
  have to show that, along the singular locus, the rank of $\mfv$ is
  strictly less than $\dim X$.  This is true at generic singular
  points, where the singular locus is smooth, since $\mfv$ must be
  parallel to the singular locus.  Then, the result follows for the
  entire singular locus, by replacing $X$ by its singular locus and
  inducting on the dimension of $X$.
\end{remark}
Now return to our assumption that $(X,\Xi)$ is a variety with a top
polyvector field $\Xi$.  For $\mfv = H_\Xi(X), LH_\Xi(X)$, or
$P_\Xi(X)$, it is clear that the degenerate locus is the union of the
singular locus with the vanishing locus of $\Xi$.  We will also call
this \emph{the degenerate locus of $\Xi$}.

\begin{theorem} \label{t:inc-leaves} Let $(X,\Xi)$ be a
  variety equipped with a top polyvector field.  If $\mfv := H_\Xi(X)$
  or $LH_\Xi(X)$, then every degenerate point is a (zero-dimensional)
  leaf.  That is, $\mfv|_x \neq T_x X$ implies $\mfv|_x = 0$.
\end{theorem}
We remark that the theorem is in stark contrast to the previous
subsections, where in general there can exist leaves of positive
dimension less than the dimension of $X$. For surfaces, where $\Xi$ is
the same as a Poisson structure, the theorem reduces to the statement
that all symplectic leaves have dimension zero or two.
\begin{proof}[Proof of Theorem \ref{t:inc-leaves}]
  It suffices to show that $\Xi$ vanishes on the singular locus of
  $X$. Let $Z$ be an irreducible component of the singular locus.
  Then $\dim Z < \dim X$, and $\mfv$ is parallel to $Z$.  Hence,
  $(\wedge^{\dim Z} \mfv)|_Z = 0$ (this holds at smooth points of $Z$,
  hence generically on $Z$, and hence on all of $Z$). \qedhere

\end{proof}
\begin{corollary}\label{c:inc-degloc}
  For $(X,\mfv)$ as in the theorem, assuming also that $X$ is purely
  of positive dimension, the following are equal:
\begin{enumerate}
\item[(i)] The degenerate locus of $\mfv$;
\item[(ii)] The set-theoretic support of the ideal generated by
  $\mfv (\cO_X)$;
\item[(iii)] The set of points $x$ such that
  $(\hat \cO_{X,x})_{\mfv} \neq 0$.
\end{enumerate}
\end{corollary}
\begin{proof} It is easy to see that (ii) and
  (iii) coincide with the vanishing locus of $\mfv$ since $X$ is
  positive-dimensional.  
  The theorem implies that this coincides with (i).
\end{proof}
\begin{corollary}\label{c:vtop-leaves}
  For $(X,\mfv)$ as in the theorem, $X$ is the union of finitely
  many open leaves and the degenerate (set-theoretic) locus of $\Xi$.
  There are finitely many leaves if and only if the degenerate locus
  is finite.
\end{corollary}
\begin{proof}
  The connected components of the open locus where $\mfv|_x = T_x X$
  are the open leaves (of which there are finitely many), 
and the vanishing locus of $\mfv|_x$ is the
  union of all points which are leaves. By the theorem, the union of
  these is all of $X$. 
\end{proof}
\begin{corollary}\label{c:vtop-hol}
Let $\mfv:= H_\Xi(X)$ or $LH_\Xi(X)$. 
Then, the following are equivalent:
\begin{itemize}
\item[(i)] $(\cO_X)_{\mfv}$ is finite-dimensional;
\item[(ii)] The degenerate locus of $\Xi$  is finite;
\item[(iii)] $\mfv$ is holonomic.
\end{itemize}
\end{corollary}
\begin{proof}
  By the corollary, $X$ has finitely many leaves if and only if it has
  finitely many zero-dimensional leaves.  Since zero-dimensional
  leaves are automatically incompressible, this shows that (ii) and
  (iii) are equivalent. Moreover, since zero-dimensional leaves always
  support linearly independent evaluation functionals in
  $((\cO_X)_{\mfv})^*$, (i) implies (ii) and (iii).
  The implication (iii) $\Rightarrow$ (i) is immediate.
\end{proof}
Note that, in contrast to $H_\Xi(X)$ and $LH_\Xi(X)$, $P_\Xi(X)$ can
be holonomic even without having finitely many leaves (e.g., in the
case when $\Xi = 0$, this happens if and only if $X$ has finitely many
exceptional points).

One example of a variety with a top polyvector field is an
even-dimensional (affine) Poisson variety, with $\Xi = \pi^{\wedge \dim
  X/2}$, for $\pi$ the Poisson bivector field.  Note that $P_\Xi(X)
\supseteq P_{\pi}(X)$.  We claim that this is a proper containment if
and only if $\dim X > 2$. For $\dim X = 2$ it is clear these are
equal. Otherwise, since $\Xi \neq 0$ if and only if $\pi$ is
generically symplectic, passing to a formal neighborhood of a point,
the statement reduces to the case $X=\bA^{2n}$ with $n > 1$ and the
usual symplectic structure, where it is well-known and easy to check.

\begin{example}
  As noted in example \ref{ex:cy}, if $X$ is a symplectic variety,
  then in particular it is Calabi-Yau and $M(X,H_\pi(X)) =
  M(X,H_\Xi(X))=\Omega_X$, whether we use the Poisson bivector $\pi$
  or the top polyvector field $\Xi = \wedge^{\dim X/2} \pi$
  (cf.~Example \ref{ex:cy2}).  However, for general Poisson varieties,
  again setting $\Xi = \wedge^{\dim X/2}$, this does not hold.  For
  example, if the Poisson bivector field $\pi$ has generic rank two
  and $\dim X \geq 4$, then the top exterior power, $\Xi = \pi^{\wedge
    (\dim X / 2)}$, is zero, so $H_\Xi = LH_\Xi = 0$, and $P_\Xi =
  \Vect(X)$, but this is clearly not true of $H_\pi, LH_\pi$, and
  $P_\pi$, and the coinvariants will differ in general.
\end{example}
\begin{example} \label{ex:cy-cplte-int} Generalizing Example
  \ref{ex:poiss-cplte-int}, we can let $(Y,\Xi_Y)$ be any
  $n$-dimensional variety with a top polyvector field, and let
  $X=Z(f_1,\ldots,f_{k}) \subseteq Y$ be a complete intersection. Then
  we can set $\Xi_X = i_{\Xi_Y}(df_1 \wedge \cdots \wedge df_{k})$,
  which is a top polyvector field on $X$. 
  Then, by Corollary \ref{c:vtop-hol}, $H(X)$ is holonomic if and only
  if $X$ has only isolated singularities, and the degenerate locus of
  $Y$ meets $X$ at only finitely many points.  In this case, we
  explicitly compute $(\cO_X)_{H(X)}$ in \S \ref{s:cy-cplte-int-is}.
\end{example}
\begin{remark} \label{r:inc-prod} Unlike Example \ref{ex:poiss-prod},
  a product formula does not hold for the above Lie algebras of vector
  fields on $X \times Y$, when $X$ and $Y$ are equipped with top
  polyvector fields $\Xi_X$ and $\Xi_Y$ and $X \times Y$ is equipped
  with the tensor product $\Xi_X \boxtimes \Xi_Y$.  First of all, for
  the Lie algebras $P$, note that, in general,
\[
P(X \times Y) \not \subseteq (P(X) \boxtimes \cO_{Y}) \oplus (\cO_{X}
\boxtimes P(Y)).
\]
For example, when $X$ and $Y$ admit vector fields $\Eu_X, \Eu_Y$ such
that $L_{\Eu_X}(\Xi_X)=\Xi_X$ and $L_{\Eu_Y}(\Xi_Y)=\Xi_Y$, then
$\Eu_X - \Eu_Y$ is in the LHS but not the RHS above. (This holds, for
example, when $X$ and $Y$ are conical with top polyvector fields
$\Xi_X$ and $\Xi_Y$ which are homogeneous of nonzero weight under the
scaling action, replacing the standard Euler vector fields by suitable
nonzero multiples).

Using this, one can see that a product formula does not hold for
coinvariants: suppose $(\cO_X)_{P(X)} \ncong
(\cO_X)_{\Vect(X)}$. Suppose that $\xi \in \Vect(X)$ is a vector field
such that $\xi(\cO_X) \not \subseteq P(X)(\cO_X)$ and $L_\xi(\Xi_X) =
\Xi_X$.  Then $P(X \times X)(\cO_{X \times X})$ contains $(\xi
\boxtimes 1 - 1 \boxtimes \xi)(\cO_X \boxtimes 1) = \xi(\cO_X)$, but
this is not contained in $(P(X)(\cO_X) \boxtimes \cO_X) + (\cO_X
\boxtimes P(X)(\cO_X))$. Since also $P(X \times X)$ contains
horizontal and vertical vector fields, $P(X) \boxtimes 1$ and $1
\boxtimes P(X)$, we conclude that $(\cO_{X \times X})_{P(X \times X)}$
is quotient of $(\cO_X)_{P(X)}^{\boxtimes 2}$ by a nontrivial vector
subspace.

For an explicit example, we could let $X$ be the cuspidal curve
$x^2=y^3$ in the plane $\bA^2$. Then, $P(X) = \langle 2x \partial_y +
3y^2 \partial_x \rangle$ and hence $(\cO_X)_{P(X)}$ surjects (in fact
isomorphically by a special case of Corollary \ref{c:is-qh};
cf.~Remark \ref{r:is-qh}) to $(\cO_X)/(2x,3y^2)$, which is
two-dimensional; on the other hand, since $\Vect(X)$ contains the
Euler vector field $3x\partial_x +2 y \partial_y$, $(\cO_X)_{\Vect(X)}
= (\cO_X)/(x,y)$ is one dimensional.  In particular, in this case,
$(\cO_{X^2})_{P(X^2)}$ is two-dimensional, whereas
$(\cO_X)_{P(X)}^{\otimes 2}$ is four-dimensional.

For the Lie algebras of Hamiltonian and locally Hamiltonian vector
fields, let $(Y,\Xi_Y)$ be any (affine) variety with $(\cO_Y)_{H(Y)} =
0$ (by Corollary \ref{c:inc-degloc} and Example \ref{ex:cy2}, this is
equivalent to $Y$ being Calabi-Yau with $H^{\dim Y}(Y) = 0$), and
$(X,\Xi_X)$ be a positive-dimensional (affine) variety. Then, we claim
that $(\cO_{X \times Y})_{H(X \times Y)} \cong (\cO_X)/(H(X) \cdot
\cO_X) \boxtimes \cO_{Y}$, where now $(H(X) \cdot \cO_X)$ is the
\emph{ideal} generated by $H(X) \cdot \cO_X$.  That is, we claim that
the vector space $H(X \times Y) \cdot \cO_{X \times Y}$ is $(H(X)
\cdot \cO_X) \boxtimes \cO_Y$.

To see this, note that the ideal $(H(X) \cdot \cO_X)$ is identified
with the image of the contraction of $\Xi_X$ with top differential
forms on $X$.  Now, on the product variety $X \times Y$, top
differential forms are spanned by exterior products of top
differential forms on $X$ with top differential forms on $Y$.  The
same is true for top polyvector fields: a derivation of $\cO_X \otimes
\cO_Y$ is uniquely determined by its restriction to $\cO_X \otimes 1$
and $1 \otimes \cO_Y$, by the formula $D(f \otimes g) = D(f) \otimes g
+ f \otimes D(g)$.  Thus, skew-symmetric multiderivations of degree
$\dim X + \dim Y$ are of the form $\Xi_X \boxtimes \Xi_Y$ for $\Xi_X$
and $\Xi_Y$ top polyvector fields on $X$ and $Y$, respectively.

Therefore, the contraction of top polyvector fields on $X \times Y$
with top differential forms lies in the ideal $(H(X) \cdot \cO_X)
\otimes \cO_Y$ (in fact, they are equal, in view of the assumption
that $(\cO_Y)_{H(Y)} = 0$, or by the next argument). Thus we get the
inclusion of $H(X \times Y) \cdot \cO_{X \times Y}$ in $(H(X) \cdot
\cO_X) \otimes \cO_Y$.

Conversely, for any element $f \in (H(X) \cdot \cO_X) \subseteq
\cO_X$, suppose that $f = \Xi_X(\alpha)$ for some top differential
form $\alpha$.  For any $g \in \cO_Y$, write $g = \Xi_Y(d\beta)$ for
some $(\dim Y-1)$-form $\beta$.  Then, $(f \otimes g) = (\Xi_X \wedge
\Xi_Y)(\alpha \wedge d\beta)$.  Therefore, $(f\otimes g) \in H(X
\times Y) \cdot \cO_{X \times Y}$. This gives the opposite inclusion.

Note that the ideal $(H(X) \cdot \cO_X)$ is supported at the
zero-dimensional leaves of $X$, which by Theorem \ref{t:inc-leaves} is
the degenerate locus of $\Xi_X$.  More generally, for arbitrary $X$
and $Y$, the leaves of $H(X \times Y)$ and $LH(X \times Y)$ consist of
the open leaves obtained as products of open leaves in $X$ with open
leaves in $Y$, and zero-dimensional leaves at every point of the
degenerate locus.
\end{remark}

Finally, as in the Poisson case, one can also consider, for every map
$f: X \to Y$, the smaller Lie algebra of vector fields obtained by
contracting $\Xi$ with exact (or closed) $(n-1)$-forms pulled back
from $Y$.  The leaves of the resulting Lie algebra consist of open
leaves, which are the restriction of the open leaves in $X$ to the
noncritical locus, together with zero-dimensional leaves at the
critical points of $f$ together with the degenerate locus of $\Xi$.
This example includes, for every subgroup $G < \SL(n)$ (or even
$\GL(n)$), the map $X = \bA^n \to \bA^n / G = Y$.  The coinvariants of
$\cO_{\bA^n}$ under the resulting Lie algebra is finite-dimensional if
and only if the critical locus of $f$ is finite, i.e., no nontrivial
element of $G$ has one as an eigenvalue; equivalently, this says that
$G$ acts freely on the $2n-1$-sphere of unit vectors in $\bC^n$.  More
generally, we can take a quotient of an arbitrary pair $(X, \Xi)$ by a
finite group of automorphisms preserving $\Xi$, and the coinvariants
of the resulting Lie algebra are finite-dimensional if and only if the
degenerate locus of $X$ is finite and all elements of the group have
only isolated fixed points.

One can alternatively consider, for a finite group quotient $X \onto
X/G$, the Lie algebras $H(X)^G$, $LH(X)^G$, and $P(X)^G$. We can do
this slightly more generally, where $G$ only preserves $\Xi$ up to
scaling (then $G$ still acts on $H(X), LH(X)$, and $P(X)$).
\begin{proposition}\label{p:hg-vtop-lf}
  Suppose $\dim X \geq 2$ and let $G$ be a finite group of
  automorphisms of $X$ which acts on $\Xi$ by rescaling.  Let $\mfv$
  be $H(X)$ or $LH(X)$.  Then the leaves of $\mfv^G$ consist of the
  points of the degenerate locus of $X$, together with the connected
  components of the subvarieties of the open leaf whose stabilizers
  are fixed subgroups of $G$.

  If the degenerate locus of $X$ is finite, then $\mfv^G$ has finitely
  many leaves, and the same result holds for $\mfv=P(X)$.
\end{proposition}
Call a subgroup $K < G$ \emph{parabolic} if it occurs as the
stabilizer of a point in $X$, i.e., it is the stabilizer of one of the
leaves of $\mfv^G$.
\begin{proof} 
  It is clear that $\mfv^G$ must flow parallel to the given
  subvarieties.  Therefore, since $H(X) \subseteq LH(X)$, we only have
  to show that $H(X)^G$ flows transitively along each of the given
  subvarieties. Also, the last statement is immediate from this, the
  fact that $P(X)$ preserves the given subvarieties (since the
  degenerate locus is finite, it cannot flow along it), and $H(X)
  \subseteq P(X)$.
  
  Since $H(X)$ is $\caD$-localizable, the same is true for $H(X)^G$, so
  we can remove the vanishing locus of $\Xi$ and assume that $X$ is
  Calabi-Yau.

  Let $K < G$ be parabolic and let $Z$ be a connected component of
  $\{x \in X \mid \Stab_G(x) = K\}$, as mentioned in the proposition.
  We have to show that, for $z \in Z$, $H(X)^G$ spans $T_z Z$.
 
  Fix $z \in Z$ and $w \in T_z Z$. We will find $\xi \in H(X)^G$ such
  that $\xi|_z = w$.  Since $X$ is Calabi-Yau, there exists $\xi \in
  H(X)$ such that $\xi|_z = w$.  Let $\phi \in
  \tilde \Omega_X^{\dim X-2}$ be such that $\xi=\xi_\phi$.  Let $f \in
  \cO_X$ be such that $f(z)=1$ and $f(y)=0$ for all $y \in G \cdot z
  \setminus\{z\}$, and moreover such that $df|_{G \cdot z} = 0$.

  Now, consider $\eta := |K|^{-1} \sum_{g \in G}g^* \xi_{f \phi} \in
  H(X)^G$.  Then $(\xi_\psi)|_z = w$, as desired. \qedhere

\end{proof}
Using Theorem
\ref{t:main1-alt}, we immediately conclude:
\begin{corollary}
In the situation of Proposition \ref{p:hg-vtop-lf},
the coinvariants $(\cO(X))_{H(X)^G}$ are
finite-dimensional.
\end{corollary}
Note that, when $X$ is normal and $G$ acts by automorphisms on
$(X,\Xi)$ (preserving $\Xi$) with critical locus of codimension at
least two, then $P(X)^G = P(X/G)$.  This is because, by Hartogs'
theorem, vector fields on $X/G$ are the same as $G$-invariant vector
fields on $X$, and such vector fields preserve $\Xi_X$ if and only if
they preserve $\Xi_{X/G}$.  In particular, we conclude in this case
that $(\cO_{X/G})_{P(X/G)} = (\cO(X))_{P(X)^G}^G$, and that this,
as well as $(\cO_X)_{P(X)^G}$ itself, are finite-dimensional if and
only if the degenerate locus of $X$ is finite.  Moreover,
$M(X/G,P(X/G)) \cong q_* M(X,P(X)^G)^G$, where $q: X \to X/G$ is the
projection.  We caution, however, that $H(X/G)$ and $LH(X/G)$ are in general
much smaller than $P(X/G)$ (even for $X$ Calabi-Yau), owing to the
fact that $G$-invariant $k$-forms on $X$ do not in general descend to
$k$-forms on $X/G$ when $k > 1$.  In fact, by Theorem
\ref{t:inc-leaves}, $(\cO_{X/G})_{H(X/G)}$ and
$(\cO_{X/G})_{LH(X/G)}$, as well as $(\cO_X)_{H(X/G)}$ and
$(\cO_X)_{LH(X/G)}$, are finite-dimensional if and only if the
critical locus of $G$ is \emph{finite} and $X$ has a finite degenerate
locus.

\subsection{Divergence functions and incompressibility}
\label{s:div}
The preceding example can be generalized to the setting of degenerate
versions of \emph{multivalued} volume forms (i.e., Calabi-Yau structures)
rather than of ordinary volume forms.  We formulate this in terms of \emph{divergence functions}, which also yield an alternative definition of incompressibility (Proposition \ref{p:reinc}).

We assume throughout that $X$ is irreducible and reduced.
Recall the definitions of polyvector fields $T_X^\bullet$ and differential
forms $\Omega_X^\bullet$ and $\tilde \Omega_X^\bullet$ from \S \ref{s:gen}.
\begin{definition}\label{d:df}
  Let $N \subseteq T_X$ be an
  $\cO_X$-submodule.  A \emph{divergence function} $D$ on $N$ is a morphism
  of sheaves of vector spaces 
  $D: N \to \cO_X$ satisfying $D(f \xi) = f D(\xi) + \xi(f)$
  for all $\xi \in N$ and $f \in \cO_X$.  When $N = T_X$, we call this
  a divergence function on $X$.
\end{definition}
As we will explain, divergence functions should be viewed as a
degenerate, \emph{multivalued} version of Calabi-Yau structures: they
simultaneously generalize flat sections of flat connections on the
canonical bundle (which includes volume forms), discussed
in Example \ref{ex:mvvol}, and top polyvector
fields on possibly singular schemes of finite type, discussed in 
\S \ref{ss:vtop}. 

For the latter, given $(X, \Xi)$, we let $N \subseteq T_X$ be the
submodule of $\xi \in T_X$ such that $L_{\xi}(\Xi)$ is a multiple of
$\Xi$.  This is a submodule in view of the identity $L_{f \xi}(\Xi) =
f L_{\xi}(\Xi) - \xi(f) \cdot \Xi$, which can be checked in local
formal coordinates where $\Xi$ is nondegenerate (where we can take
$\Xi$ to be the inverse to the standard volume form on the formal
neighborhood of the origin in affine space).  Next, define $D$ by the
formula $D(\xi)\cdot \Xi = -L_{\xi}(\Xi)$.  Note that,
on the nondegenerate locus of $\Xi$, call it $X^\circ \subseteq X$, we have
$N|_{X^\circ} = T_{X^\circ}$, since $X^\circ$ is symplectic.

 Next, we explain how divergence functions generalize multivalued
 volume forms:
\begin{proposition}\label{p:div-conncan}
If $X$ is normal and of pure dimension $n$, then the following are
in natural bijection:
\begin{enumerate}
\item[(i)] Divergence functions $D$ on $N \subseteq T_X$;
\item[(ii)] Connections $N \times \tilde \Omega^{n}_X \to \tilde \Omega_{X}^n$ on 
$\tilde \Omega^{n}_X$
 along $N$.
\item[(iii)] Connections $N \times T^{n}_X \to T^{n}_X$ on 
$T^{n}_X$
 along $N$.
\end{enumerate}
The equivalence between (i) and (ii) 
is given by the correspondences, for $\xi \in N$ and $\omega \in \tilde \Omega_X^{n}$,
\begin{gather}
D \mapsto \nabla^D, \quad \nabla^D_\xi(\omega) = L_\xi(\omega) - D(\xi) \cdot \omega; \\
\nabla \mapsto D_\nabla, \quad D_\nabla(\xi) = L_\xi - \nabla_\xi \in \End_{\cO_X}(\tilde \Omega_X^{n}) \cong \cO_X.
\end{gather}
The equivalence between (i) and (iii) is given by the formulas, for $\xi \in N$
and $\Xi \in T_X^{n}$,
\begin{gather}
D \mapsto \nabla^D, \quad \nabla^D_\xi(\Xi) = L_\xi(\Xi) + D(\xi) \cdot \Xi; \\
\nabla \mapsto D_\nabla, \quad D_\nabla(\xi) = \nabla_\xi - L_\xi \in \End_{\cO_X}(T_X^{n}) \cong \cO_X.
\end{gather}
Finally, the constructions $D \mapsto \nabla^D$ are valid even when
$X$ is not normal.
\end{proposition}
We will need the elementary
\begin{lemma} \label{l:norm-tfc2} 
Suppose that $X$ is normal and that $F$ is a torsion-free coherent sheaf on $X$ which is a line bundle outside of codimension two.  Then $\End(F) = \cO_X$.
\end{lemma}
\begin{proof}
 For any $a \in \End(F)$, on
  some open subset $U \subseteq X$ where $F$ is a line bundle and $X
  \setminus U$ has codimension at least two, $a|_U \in \End(\cO_U) =
  \Gamma(\cO_U)$. By normality, this is the restriction of a function
  $f_a \in \cO_X$.  Since $\cO_X \subseteq \End(F)$, we conclude that
  $f_a-a \in \End(F)$ has zero restriction to $U$, and hence is zero
  since $F$ is torsion-free.
\end{proof}
\begin{proof}[Proof of Proposition \ref{p:div-conncan}]
  Suppose that $D$ is a divergence function. Then $\nabla^D_\xi(f
  \cdot \omega) = f\nabla^D_\xi(\omega) + \xi(f) \cdot
  \omega$. Similarly, $\nabla^D_{f\xi}(\omega) = f \nabla^D(\omega) +
  \xi(f)-\xi(f) = f \nabla^D(\omega)$. We deduce that $\nabla^D_\xi$
  is a connection.  Similarly, if $\nabla$ is a connection on $\tilde
  \Omega_X$, then first of all $L_\xi(f\omega) - \nabla_{\xi}(f\omega)
  = f \bigl( L_\xi(\omega) - \nabla_\xi(\omega) \bigr)$, so
  $D_\nabla(\xi)$ is indeed a well-defined $\cO_X$-module endomorphism
  of $\tilde \Omega_X$.  By Lemma \ref{l:norm-tfc2}, this is the same
  as an element of $\cO_X$.
  Then, $D_\nabla(f\xi) = f
  D_\nabla(\xi) + \xi(f)$, so $D_\nabla$ is a divergence function.
  One immediately checks that $D_{\nabla^D} = D$ and
  $\nabla^{D_\nabla} = \nabla$.

  The proof of the equivalence between (i) and (iii) is similar, so we
  omit the details.  For the final statement, note that well-definition of 
$\nabla^D$ did not require normality.
\end{proof}
\begin{remark}
In fact, in Proposition \ref{p:div-conncan}, we can replace $T_X^n$
and $\tilde \Omega_X^n$ by any torsion-free coherent sheaves which
coincide with $T_X^n$ and $\tilde \Omega_X^n$, respectively, outside
of codimension two; the proof then goes through unchanged.
\end{remark}
\begin{remark}
  For not necessarily normal $X$, but still of pure dimension $n$,
  Proposition \ref{p:div-conncan} generalizes to give an equivalence
  between divergence functions of the form $D: N
  \to \End_{\cO_X}(\tilde \Omega_X^{n}) \supseteq \cO_X$ and connections $N
  \times \tilde \Omega^{n}_X \to \tilde \Omega^{n}_X$.  Similarly, we obtain an
  equivalence between divergence functions valued in
  $\End_{\cO_X}(T_X^{n}) \supseteq \cO_X$ and connections on
  $T_X^{n}$ along $N$.
\end{remark}

Divergence functions yield the following alternative formulation of
the incompressibility condition.
 Let $\cO_X \cdot \mfv$ denote the $\cO_X$-linear span of $\mfv$
and similarly for $\cO_{X'}$ where $X'$ is an open subvariety of $X$
(we will also use this notation for formal neighborhoods, etc.).
\begin{proposition}\label{p:reinc}
  Let $X$ be an arbitrary affine variety and $\mfv$ a Lie algebra of
  vector fields $\mfv$ on $X$.  Then, the flow of $\mfv$ along $X$ is
  incompressible if and only if there exists an open dense subset
  $X^\circ \subseteq X$ and a divergence function on $\cO_{X^\circ}
  \cdot \mfv|_{X^\circ}$ annihilating $\mfv|_{X^\circ}$.  In this
  case, in the formal neighborhood of every point of $X^\circ$, there
  exists a volume form preserved by $\mfv$.
\end{proposition}
We can restate this in terms of connections using:
\begin{proposition}\label{p:div-lie}
In terms of Proposition \ref{p:div-conncan}, when $X$ is normal and of
pure dimension, a
divergence function $D$ on $N$ annihilates $\mfv \subseteq T_X$ if and only
if $\nabla^D_\xi = L_\xi$ for all $\xi \in \mfv$.
\end{proposition}
The proof is immediate from the definition of $\nabla^D$ and $D_\nabla$.

\subsubsection{Flat divergence functions}
In terms of Proposition \ref{p:div-conncan}, we can describe what it means
for a divergence function to be flat. As before, assume that $X$ is a variety
of pure dimension $n$.  Assume that $N \subseteq T_X$ is a Lie subalgebroid.

Consider the extension of $D$ to an operator $\tilde D:
\wedge_{\cO_X}^\bullet N \to \wedge_{\cO_X}^{\bullet-1} N$ given by
\begin{multline}\label{e:divfn-cplx}
\xi_1 \wedge \cdots \wedge \xi_k 
\mapsto
\sum_{i=1}^k (-1)^{i-1} D(\xi_i)
\xi_1 \wedge \cdots \wedge \hat \xi_i \wedge
\cdots \wedge \xi_k
\\ + \sum_{i,j} (-1)^{i+j-1} [\xi_i, \xi_j] \wedge \xi_1 \wedge \cdots \wedge \hat \xi_i  \wedge \cdots \wedge \hat \xi_j \wedge \cdots \wedge \xi_k. 
\end{multline}
Note that, since we take the exterior algebra over $\cO_X$, 
one must check that the formula is well-defined, i.e., that one obtains the same result if we multiply $\xi_i$ by $f$ as if we multiply $\xi_j$ by $f$, for all $i < j$ and all $f \in \cO_X$.  This is easy to check.
\begin{definition}
Call a divergence function $D$ flat if
the associated operator \eqref{e:divfn-cplx} has square zero: $\tilde D^2=0$.
\end{definition}
\begin{example}\label{ex:neqtx}
Suppose that $N=T_X$ and $X$ is smooth.  Then we can replace \eqref{e:divfn-cplx} with
\begin{equation}\label{e:divfn-cplx-smth}
\Omega_X^{\dim X -
  \bullet} \otimes_{\cO_X} T_X^n,
\end{equation} 
equipped with the derivation $d_D = d \otimes\Id + \Id \otimes \bar \nabla^D$, where $\bar \nabla^D: T_X^n \to \Omega^1_X \otimes_{\cO_X} T_X^n$ is the usual $\bk$-linear operator associated to the connection $\nabla^D$. This is isomorphic
to \eqref{e:divfn-cplx} by contracting $\Omega^\bullet$ with $T_X^n$.
Thus, $d_D^2=0$ 
if and only if $D$ is flat.
\end{example}
\begin{example}\label{ex:nlocfree}
More generally than Example \eqref{ex:neqtx}, 
  suppose that $X$ is smooth and $N$ is locally free of rank $n-k$ and
  the vanishing locus of a collection of (linearly independent)
  one-forms $df_1, \ldots, df_k$. Then, we can consider the $k$-form
  $\alpha = df_1 \wedge \cdots \wedge df_k$, and replace
  \eqref{e:divfn-cplx} by 
\begin{equation}\label{e:divfn-cplx-locfree}
(\Omega_X^{\dim X - k -
  \bullet} \wedge \alpha) \otimes_{\cO_X} T_X^n.
\end{equation} 
This is equipped with the derivation $d_D$ defined as before, and with
this derivation, the contraction map produces an isomorphism of
\eqref{e:divfn-cplx-locfree} with \eqref{e:divfn-cplx}.  Thus, it
remains true that $d_D^2=0$ if and only if $D$ is flat.  Moreover, by
Frobenius's theorem, in a formal neighborhood of a smooth point $x \in
X$, such $f_1, \ldots, f_k$ always exist since $N$ is integrable.
\end{example}
\begin{proposition} \label{p:flat-lie}
Let $D: N \to \cO_X$ be a divergence function with
  $N \subseteq T_X$ a Lie subalgebroid, and let $\mfv := \{\xi \in N
  \mid D(\xi)=0\}$.  Suppose moreover that $N = \cO_X \cdot
  \mfv$. Then $D$ is generically flat if and only if $\mfv$ is a Lie
  algebra.
\end{proposition}
By generically flat, we mean that, restricted to an open dense subset
of $X$, $D$ is flat.  Note that the condition $N = \cO_X \cdot \mfv$
is automatic if we replace $X$ with a formal neighborhood $\hat X_x$
for generic $x \in X$ and define $\mfv \subseteq T_{\hat X_x}$ as
above, since $N$ is integrable, so we can write $\hat X_x \cong V
\times V'$ for formal polydiscs $V, V'$ such that $N$ identifies with
the subsheaf of $T_{\hat X_x}$ in the $V$ direction.
\begin{proof}
  First, if $D$ is generically flat, then on some open dense subset of
  $X$, given any $\xi, \eta \in \mfv$, we have $\tilde D^2(\xi \wedge
  \eta) = 0$ (since $D(\xi)=0=D(\eta)$), which implies that $[\xi,
  \eta]\in \mfv$ as well.

Consider now the reverse implication.  It suffices to restrict to a
formal neighborhood of a smooth point $x \in X$ (on each connected
component of $X$).  Then, as noted in Example \ref{ex:nlocfree}, we
can assume $N$ is the vanishing locus of $k$ nonvanishing one-forms
$df_1, \ldots, df_k$.  Set $\alpha = df_1 \wedge \cdots \wedge df_k$
and replace \eqref{e:divfn-cplx} by \eqref{e:divfn-cplx-locfree}.  By
Proposition \ref{p:div-lie}, $\mfv$ consists of those $\xi \in N$ such
that $\nabla^D_\xi = L_\xi$ on $\Omega^n_{\hat X_x}$, or equivalently
on $(\Omega^{n-k}_{\hat X_x} \wedge \alpha)$.

Assume that $\mfv$ is a Lie algebra.  Then, for $\xi, \eta \in
\mfv$,
\[
[\nabla^D_\xi, \nabla^D_\eta] = [L_\xi, L_\eta] = L_{[\xi,\eta]} = \nabla^D_{[\xi, \eta]}.
\]
Note that this also implies that $[\nabla^D_\xi, \nabla^D_\eta] =
\nabla^D_{[\xi, \eta]}$ for all $\xi, \eta \in \hat \cO_{X,x} \cdot
N$, since this equality remains true when replacing $\xi$ by $f \cdot
\xi$ for $f \in\hat \cO_{X,x}$, and it is biadditive in $\xi$ and
$\eta$.  Since $N = \cO_X \cdot \mfv$, and hence $N|_{\hat X_x} = \hat
\cO_{X,x} \cdot \mfv|_{\hat X_x}$, the equality holds for all $\xi,
\eta \in \hat \cO_{X,x}$.  Now, the identity $[\nabla^D_\xi,
\nabla^D_\eta] = \nabla^D_{[\xi, \eta]}$ on $\Omega^n_{\hat X_x}$
implies in the standard way that the derivation $d_D$ on
$(\Omega_{\hat X_x}^{\dim X - k - \bullet} \wedge \alpha)
\otimes_{\cO_X} T_{\hat X_x}^n$ has square zero. Namely, one can
verify that $d_D^2$ is given by contraction with the two-form $\alpha$
given by
\[
\alpha(\xi \wedge \eta) = [\nabla^D_\xi, \nabla^D_\eta] -
\nabla^D_{[\xi, \eta]} \in \End(T_{\hat X_x}^n) = \hat
\cO_{X,x}. \qedhere
\]
\end{proof}
\subsubsection{Hamiltonian vector fields on
varieties with flat divergence functions}\label{ss:ham-df}
Now we define, analogously to \S \ref{ss:vtop}, Hamiltonian and
incompressible vector fields preserving flat divergence functions
(i.e., preserving the formal volumes associated to them).

Let $X$ be a variety of pure dimension $n$ and $N \subseteq T_X$ an
$\cO_X$-submodule, and $D: N \to \cO_X$ be a flat divergence
function. Then first we have the Lie algebra $P(X,D) \subseteq N$ of
all incompressible vector fields in $N$. Note that the $\cO_X$-linear
span of $P(X,D)$ need not be all of $N$.

Next, given any element $\tau \in \wedge_{\cO_X}^2 N$, consider the
image $\theta_\tau := \tilde D(\tau) \in N$. By construction,
$D(\theta_\tau) = 0$.  We call $\theta_\tau$ the Hamiltonian vector
field of $\tau$.  Since $[\theta_\tau, \theta_{\tau'}] =
\theta_{L_{\theta_\tau}(\tau')}$, these form a Lie subalgebra of
$P(X,D)$,
\[
H(X,D) := \langle \theta_\tau \rangle
\subseteq P(X,D).
\]
\begin{example}
  If $X$ is Calabi-Yau and $D$ is the associated divergence function,
  we again recover $H(X,D)=P(X,D)=H(X)$, the Lie algebra of
  volume-preserving vector fields.
\end{example}
As long as $N$ has rank at least two, then $H(X,D)$ has enough vector
fields, in the sense that $\cO_X \cdot H(X,D) = N$; more precisely:
\begin{proposition}\label{p:div-enough-vfds}
  Suppose that the image of $N$ at the tangent fiber $T_x X$ has
  dimension at least two.  Then $H(X,D)|_x = N|_x$, i.e., $H(X,D)
  \subseteq N$ spans the same tangent space at $x$ as $N$.  In
  particular, if $N = T_X$ and $X$ has pure dimension at least two,
  then $H(X,D)$ is transitive.
\end{proposition}
As a consequence, the same result holds for $P(X,D) \supseteq H(X,D)$.
\begin{proof}
  Let $x \in X$ be a point, and $\xi, \eta \in N$ two vector fields
  linearly independent at $x$.  Let $f \in \cO_X$ be a function such
  that $\xi(df)(x)=1$ and $\eta(df)(x)=0$.  Then $\bigl(\tilde D(f \xi
  \wedge \eta) - f \tilde D(\xi \wedge \eta)\bigr)|_x = \eta|_x$.
\end{proof}
On the other hand, if $N$ has rank one, then $P(X,D)$ can be zero,
e.g., when $X$ is a smooth curve and $D$ is a divergence function
preserving a multivalued volume form which is not single valued
(cf.~Example \ref{ex:mvvol}).
\begin{example} Consider the case of
 Example \ref{ex:nlocfree}, i.e., where 
 $N$ is locally free of rank $n-k$ and the zero locus of (linearly
 independent) exact one-forms $df_1, \ldots, df_k$.  Set $\alpha :=
 df_1 \wedge \cdots \wedge df_k$ and replace \eqref{e:divfn-cplx} by
 \eqref{e:divfn-cplx-locfree}.
Given any element $\beta \in (\Omega^{n-k-2}_X\wedge \alpha)
\otimes T_X^n$, we can define the Hamiltonian vector field
\[
\xi_{\beta} = \ctr(\nabla^D(\beta)),
\]
where $\ctr$ is the operator
\[
\ctr: \tilde \Omega^\bullet_X \otimes_{\cO_X} T_X^n \to
T_X^{n-\bullet}, \ctr(\omega \otimes \tau) = i_\tau(\omega).
\]
These vector fields coincide with $H(X,D)$ as defined above, since
 \eqref{e:divfn-cplx-locfree} is isomorphic to
\eqref{e:divfn-cplx} via the contraction operation.

Next,
call an element of $(\Omega^{n-k-\bullet}_X \wedge \alpha) \otimes_{\cO_X} T_X^n$ 
$\nabla^D$-closed
if it is in the kernel of $\nabla^D$.  Then, if $\gamma \in
(\Omega^{n-k-1}_X\wedge \alpha)
\otimes T_X^n$ is $\nabla^D$-closed,
we can define the locally Hamiltonian vector field
\[
\eta_\gamma := \ctr(\gamma).
\]
These vector fields coincide with $P(X,D)$ as defined above, since via
the contraction isomorphism of complexes \eqref{e:divfn-cplx-locfree}
and \eqref{e:divfn-cplx}, the vector fields $\eta_\gamma$ are
precisely those elements of $N$ with zero divergence.
\end{example}

\begin{example}
  Suppose $(X, \Xi)$ is a variety of pure dimension $n$ equipped with
  a generically nonvanishing top polyvector field $\Xi$ as in \S
  \ref{ss:vtop}, and define $N$ and $D$ as at the beginning of \S
  \ref{s:div}.  Then we see immediately that $P(X) = P(X,D)$,
  consisting of the vector fields $\xi$ such that $L_\xi \Xi = 0$. On
  the other hand, $H(X)$ need not equal $H(X,D)$ in general. Indeed,
  although it is true that, given $\alpha \in \Omega^{n-2}_X$, then
  $\theta_{i_\Xi(\alpha)} = \xi_{\alpha}$, we have $\wedge^2 N \neq
  i_{\Xi}(d\Omega^{n-2}_X)$ in general, and sometimes $H(X) \neq
  H(X,D)$.
\end{example}

\subsubsection{Proof of Proposition \ref{p:reinc}}\label{ss:reinc-pf}
We can assume that $X=X^\circ$ is smooth and that $\mfv$ has constant
(i.e., maximal) rank. Therefore $\Omega_X = \tilde \Omega_X$, and we
omit the tilde from now on.  We show that $\mfv$ flows incompressibly
on $X$ if and only if there exists a connection $\nabla$ on $\Omega_X$
along $N:= \cO_X \cdot \mfv$ such that $\nabla_\xi=L_\xi$ for all $\xi
\in \mfv$.

First, suppose that $\mfv$ flows incompressibly on $X$. Let $x \in X$
be a point and $\omega \in \Omega_{\hat X_x}$ a formal volume form
preserved by $\mfv$. Let $\nabla$ be the unique flat connection whose
flat sections are multiples of $\omega$. Then $\nabla_\xi \omega = 0 =
L_\xi \omega$ for all $\xi \in \mfv$. Therefore, the restriction of
$\nabla$ to $N$ is as desired.

Conversely, suppose that $\nabla$ is a connection on $\Omega_X$ along
$N$ such that $\nabla_\xi = L_\xi$ for all $\xi \in \mfv$.  Since
$\mfv$ is a Lie algebra, Proposition \ref{p:flat-lie} implies that
$\nabla$ is generically flat.  
Thus, at a generic point $x  \in X$,
$\hat N_x$ is free over $\hat \cO_{X,x}$, and we can
write $T_{\hat X_x} = \hat N_x \oplus L$ for some complementary free
$\hat \cO_{X,x}$-submodule $L$.  Then the connection $\nabla$ can be
extended to a flat connection on $T_{\hat X_x}$.  Let $\omega
\in\Omega_{\hat X_x}$ be a nonzero flat formal section of
$\nabla$. Then $\nabla_\xi(\omega)=0$ for all $\xi \in T_X$.  Hence
$L_\xi(\omega)=0$ for all $\xi \in \mfv$.  Therefore, $\omega$ is
preserved by $\mfv$.

\subsection{Smooth curves}\label{ss:curves}
Let $X$ be a smooth connected curve.  In this section we explicitly
compute $M(X,\mfv)$.  We may assume that $\mfv$ is nonzero.  Let $Z
\subseteq X$ be the vanishing locus of $\mfv$, which is
zero-dimensional.  Let $X^\circ:= (X \setminus Z) \subseteq X$ be the
complement.
\begin{lemma}
If $\mfv$ is one-dimensional, then $M(X,\mfv)|_{X^\circ} = \Omega_{X^\circ}$. Otherwise, $M(X,\mfv)|_{X^\circ} = 0$.
\end{lemma}
\begin{proof}
  By our assumptions, $\mfv|_{X^\circ}$ is transitive.  Moreover, if
  $\mfv$ is one-dimensional, then any nonzero element $\xi \in \mfv$
  is a top polyvector field on $X$ vanishing on $Z$, so $\xi^{-1}$
  defines a nondegenerate volume form on $X^\circ$ preserved by
  $\mfv$. Therefore we conclude that
  $M(X^\circ,\mfv)\cong\Omega_{X^\circ}$ by Proposition \ref{p:tr}.
  On the other hand, if $\mfv$ is at least two-dimensional, then if
  $\xi_1, \xi_2 \in \mfv$ are linearly independent, then on some open
  subset $U \subseteq X^\circ$, $\xi_1^{-1}$ and $\xi_2^{-1}$ both
  define nondegenerate volume forms which are not scalar multiples of
  each other.  Then there can be no volume form on $U$ preserved by
  both, even restricted to $\hat U_x$ for every $x \in U$.
\end{proof}
\begin{proposition} If $\dim \mfv \geq 2$, then $M(X,\mfv) \cong
  \bigoplus_{z \in Z} \delta_z \otimes (\hat
  \cO_{X,z})_{\mfv}$. Moreover, $\dim (\hat \cO_{X,z})_{\mfv}$ is the
  minimum order of vanishing of vector fields of $\mfv$ at $z$.
\end{proposition}
Note in particular that each $\dim (\hat \cO_{X,z})_{\mfv}$ is positive.
\begin{proof}
  By the lemma, we immediately conclude that $M(X,\mfv)$ is a direct
  sum of copies of delta-function $\caD$-modules at points of $Z$,
  which is finite.  Then, the result follows from the fact that
\begin{equation}
  \Hom(M(X,\mfv), \delta_z) = \Hom(\caD_X, \delta_z)^{\mfv} = 
((\hat \cO_{X,z})^*)^\mfv. \qedhere
\end{equation}
\end{proof}
Now, assume that $\mfv = \langle \xi \rangle$, so that $M(X,\mfv) =
\xi \cdot \caD_X \setminus \caD_X$ for $\xi \in \Vect(X)$.  Then, by the lemma
and the argument of the proposition, we have an exact sequence
\begin{equation}
  0 \to j_! \Omega_{X^\circ} = \Omega_X \into M(X,\mfv) \onto \bigoplus_{z \in Z}
  \delta_z \otimes (\hat
  \cO_{X,z})_{\mfv} \to 0.
\end{equation}
It turns out that this sequence is maximally nonsplit. Namely, at each
$z \in Z$, $\Ext(\delta_z, \Omega_X) = \bk$, since $X$ is a smooth
curve.
\begin{proposition}
  When $\mfv$ is one-dimensional, then $M(X,\mfv) = N \oplus
  \bigoplus_{z \in Z} \delta_z \otimes (\hat \cO_{X,z})_{\mfv}/\bk$,
  where $N$ is an indecomposable $\caD$-module fitting into an exact
  sequence
\begin{equation}
  0 \to j_! \Omega_{X^\circ} = \Omega_X \into N
  \onto \bigoplus_{z \in Z} \delta_z \to 0.
\end{equation}
As before, $\dim (\hat \cO_{X,z})_{\mfv}$ is the minimum order of
vanishing of vector fields of $\mfv$ at $z$.
\end{proposition}
\begin{proof}
  By formally localizing at $z \in Z$, it is enough to assume that
  $\mfv = \langle x^k \partial_x \rangle$ for $\bA^1 = \Spec \bk[x]$
  and $k \geq 1$. In this case, it suffices to prove that
\[
\Hom(\caD_{\bA^1} / x^k \partial_x \cdot \caD_{\bA^1}, \Omega_{\bA^1}) = 0.
\]
But, no volume form on $\bA^1$ is annihilated by $L_{x^k \partial_x}$
(even in a formal neighborhood of zero): the rational volume form
annihilated by $L_{x^k \partial_x}$ is $x^{-k} dx$.  The last statement follows
as in the previous proof.
\end{proof}

\subsection{Finite maps}\label{ss:fmaps}
Let $f: X \to Y$ be a finite surjective map of affine varieties. In
this section we explain how to construct more examples using finite
maps, which generalizes the aforementioned Lie algebras of Hamiltonian
vector fields of Hamiltonians pulled back from $Y$. We will not need
the material of this section for the remainder of the paper.
\begin{definition}
Let $\Vect_X(Y) \subseteq \Vect(Y)$ be the
subspace of vector fields $\xi$ on $Y$ such that there
exists a vector field $f^* \xi$ on
$X$ such that $f_* (f^* \xi|_x) = \xi|_{f(x)}$ for all $x \in X$.
\end{definition}
Algebraically, $\Vect_X(Y)$ consists of the derivations of $\cO_Y$ which
extend to derivations of $\cO_X$.

Since $f$ is finite and $X$ and $Y$ are reduced, it is generically a
covering map.  Therefore, when $f^* \xi$ exists, it is unique.
\begin{example}\label{ex:normcrit2}
  If $X$ is a normal variety and the critical locus of $f$ has
  codimension at least two, then by Hartogs' theorem, vector fields
  on $X$ outside the singular and critical locus extend to all of $X$.
  Therefore, $\Vect_X(Y) = \Vect(Y)$, since $f$ is a covering map when
  restricted to this latter locus.
\end{example}
Suppose that $X$ and $Y$ are varieties and
$\mfv_Y \subseteq \Vect_X(Y)$. Let $\mfv_X := f^* \mfv_Y$.
\begin{proposition}
\begin{itemize}
\item[(i)] $(X,\mfv_X)$ has finitely many leaves if and only if
$(Y,\mfv_Y)$ does. 
\item[(ii)] $(X,\mfv_X)$ has finitely many incompressible leaves 
of and only if $(Y,\mfv_Y)$ does.
\item[(iii)] $(X,\mfv_X)$ has finitely many zero-dimensional leaves if and only if $(Y,\mfv_Y)$ does.
\end{itemize}
\end{proposition}
\begin{proof}
  Restricted to any invariant subvariety $Z \subseteq X$, $f$ is still
  finite and therefore generically a covering map.  This reduced the
  statement to the case where $f$ is a covering map of smooth
  varieties. Then, the statements (i)--(ii) follow from the basic
  facts that (i) $X$ is generically transitive if and only if $Y$ is;
  (ii) $X$ is incompressible if and only if $Y$ is. Statement (iii)
  follows from the fact that $f$ restricts to a finite map from the
  vanishing locus of $\mfv_X$ onto the vanishing locus of $\mfv_Y$.
\end{proof}
\begin{example}
  In the situation of Example \ref{ex:normcrit2}, $X$ has finitely
  many leaves under the flow of all vector fields if and only if the
  same is true of $Y$, and $X$ has finitely many exceptional points if
  and only if $Y$ does.  Thus, $(\cO_X)_{\Vect(X)}$ is
  finite-dimensional if and only if $(\cO_Y)_{\Vect(Y)}$ is.
\end{example}
\begin{example}
  If $f: X \to Y$ is a finite Poisson map of varieties with finitely
  many symplectic leaves and $X$ is normal, one recovers the
  observation at the end of \S \ref{ss:poisson} in the setting of
  Poisson maps (note that the critical locus of $f$ is automatically
  of codimension at least two, since $f$ is nondegenerate over the
  open leaves of $Y$). Thus, one recovers \cite[Theorem 3.1]{ESdm} in
  this setting, i.e., that $f^*H(Y)$ is holonomic (similarly one
  obtains that $f^*LH(Y)$ and $f^*P(Y)$ are holonomic).  Here, we only
  used the conditions that $X$ is normal and $Y$ has finitely many
  symplectic leaves to assure that $H(Y) \subseteq \Vect_Y(X)$; to
  drop these assumptions, one can observe that $H(Y) \subseteq
  \Vect_Y(X)$, since $f^* \xi_h = \xi_{f^* h}$ (which also allows one
  to drop the condition that $Y$ is Poisson altogether, using Hamiltonian vector fields on $X$ of the form $\xi_{f^* h}$); similarly we
  can conclude in this setting that $LH(Y) \subseteq \Vect_Y(X)$.
\end{example}
\begin{example}
  Suppose $f: X \to Y$ is a finite map of varieties equipped with top
  polyvector fields $\Xi_X$ and $\Xi_Y$ such that $f_*(\Xi_X|_x) =
  \Xi_Y|_{f(x)}$ for all $x \in X$ (an ``incompressible'' finite map).
  If $X$ is normal, $\Xi_Y$ has a finite degenerate locus, and the
  dimension of $X$ is at least two, one concludes that $f^* H(Y)$ is
  holonomic (as well as $f^* LH(Y)$ and $f^* P(Y)$), and hence that
  $(\cO_X)_{f^* H(Y)}$ is finite-dimensional; this recovers an
  observation at the end of \S \ref{ss:vtop} in a special case.  As in
  the previous remark, we can drop the assumptions that $X$ is normal
  and $\Xi_Y$ has a finite degenerate locus, since those were only
  used to show that $H(Y) \subseteq \Vect_Y(X)$, but this is
  automatic since we can pull back closed $(n-1)$-forms from $Y$ to
  $X$ (this also applies to $LH(Y)$, but not necessarily to $P(Y)$).

  In the case $Y=X/G$ where $G$ is a finite group acting on $(X,
  \Xi_X)$, one similarly recovers the observation from the end of \S
  \ref{ss:vtop}, that $(\cO_{X/G})_{P(X/G)} = (\cO_X)_{f^*P(X/G)}^G =
  (\cO_X)_{P(X)^G}^G$, as well as $(\cO_X)_{P(X)^G}$, are
  finite-dimensional if and only if $\Xi_X$ has a finite degenerate
  locus, i.e., if and only if $(\cO_X)_{P(X)}$ is finite-dimensional.
\end{example}

\section{Globalization and Poisson vector fields}
\label{s:ex-dloc}
\subsection{Hamiltonian vector fields are $\caD$-localizable}
In order to prove that our main examples are $\caD$-localizable (for
all vector fields and Hamiltonian vector fields), we prove the
following more general result, which roughly states that a Lie algebra
of vector fields generated by a coherent sheaf $E$ of ``potentials''
is $\caD$-localizable (in the Poisson case with $\mfv=H(X)$, or in the
case $\mfv=\Vect(X)$, $E=\cO_X$, as we will explain):
\begin{theorem}\label{t:h-loc}
  Let $E$ be a coherent sheaf on an affine variety $X$ equipped with a
  map $v: E \to T_X$ of $\bk$-linear sheaves, such that, for all $e
  \in E$, the bilinear map
\[
\pi_e(f,g) := v(f \cdot e)(g) - f \cdot v(e)(g)
\]
defines a skew-symmetric biderivation $\cO_X^{\otimes 2} \to \cO_X$.
Then (the Lie algebra generated by) $v(E)$ is $\caD$-localizable.
\end{theorem}
The condition of the theorem can alternatively be stated as: $v: E \to
T_X$ is a differential operator of order $\leq 1$ whose principal
symbol $\sigma(v): E \to T_X \otimes T_X$ is skew-symmetric.
\begin{proof}
  Let $X \subseteq \bA^n$ be an embedding into affine space, and let
  $x_1, \ldots, x_n$ be the coordinate functions on $\bA^n$.  Let 
  $U \subseteq X$ be an open affine subset.
We need to show that, for every $g
  \in \cO_{U}$ and $e \in E(X)$, then $v(g \cdot e) \in v(E(X)) \cdot
  \caD_U$.  Let $V \subseteq \bA^n$ be an open affine subset such that $V \cap X = U$.  We claim that, in $\caD_U$, for all $f \in \cO_V$,
\begin{equation}\label{e:loc}
  v(f \cdot e) = v(e) \cdot f + \sum_{i=1}^n \bigl(v(x_i \cdot e) - v(e) \cdot x_i \bigr) \cdot \frac{\partial f}{\partial x_i},
\end{equation}
which immediately implies the statement.  To prove \eqref{e:loc}, we
first rewrite it (putting vector fields on the left-hand side and
functions on the right hand side) as
\[
(v(f \cdot e) - f \cdot v(e)) - \sum_{i=1}^n \frac{\partial
  f}{\partial x_i} \bigl(v(x_i \cdot e) - x_i v(e) \bigr) = v(e)(f) +
\sum_{i=1}^n \bigl( -\frac{\partial f}{\partial x_i} v(e)(x_i) +
(v(x_i \cdot e) - x_i v(e))(\frac{\partial f}{\partial x_i}) \bigr).
\]
So the statement is equivalent to showing that both sides of the above
desired equality are zero. For the LHS, this follows from the fact
that, for fixed $e \in E$, the map $f \mapsto v(f \cdot e) - f \cdot
v(e)$ is a derivation of $f$; in more detail, this implies that this
is obtained from a linear map $\Omega^1 \to T_X, df \mapsto v(f \cdot
e) - f \cdot v(e)$, and then we write $df = \sum_i \frac{\partial
  f}{\partial x_i} dx_i$.  For the RHS, the fact that $v(e) \in T_X$
is a derivation implies that $v(e)(f) + \sum_{i=1}^n -\frac{\partial
  f}{\partial x_i} v(e)(x_i) = 0$, just as before. It remains to show
that
\[
\sum_{i=1}^n \bigl( v(x_i \cdot e) - x_i \cdot v(e)
\bigr)\bigl(\frac{\partial f}{\partial x_i}\bigr) = 0.
\]
Using the definition of $\pi_e$, we can rewrite the LHS of this
expression as
\[
\sum_i \pi_e\bigl(x_i, \frac{\partial f}{\partial x_i}\bigr).
\]
Now, viewing $\pi_e$ as a bivector field (i.e., a skew-symmetric
biderivation), this can be rewritten as
\[
\sum_i \pi_e \bigl(dx_i \wedge d(\frac{\partial f}{\partial x_i})\bigr) =
\pi_e d(df) = 0. \qedhere
\]
\end{proof}
\begin{corollary}\label{c:all-local}
  $(X,\Vect(X))$ is $\caD$-localizable.  More generally, if $E
  \subseteq \Vect(X)$ is a coherent subsheaf, then (the Lie algebra
  generated by) $E$ is $\caD$-localizable.
\end{corollary}
\begin{proof}
Take $v = \Id$ in the theorem.
\end{proof}
\begin{corollary}
  Let $X$ be either Poisson, Jacobi, or equipped with a top polyvector
  field. Then the presheaf $\caH(X)$ of Hamiltonian vector fields is
  $\caD$-localizable.  Moreover, in the Poisson and top polyvector
  field cases, the presheaf $\cLH(X)$ of locally Hamiltonian vector
  fields is also $\caD$-localizable, and defines the same
  $\caD$-module.

  Similarly, when $X$ is equipped with a coherent subsheaf $N
  \subseteq \cT_X$ and a divergence function $D: N \to \cO_X$, then
  the presheaf $\mathcal{H}(X,D)$ is $\caD$-localizable, setting $E :=
  \wedge^2_{\cO_X} N$.
\end{corollary}
\begin{proof}
  In the Poisson and Jacobi cases, we can take $E = \cO_X$ and $v(f) =
  \xi_f$. Then it is easy to check that $\pi_e$ is a skew-symmetric
  biderivation for all $e \in E$, so the theorem implies that $\caH_X$
  is $\caD$-localizable.  In the case of a top polyvector field $\Xi$,
  we take $E = \tilde \Omega_X^{n-2}$ and again let $v(\alpha) = \xi_\alpha =
  \Xi(d\alpha)$.

  For the second statement, it suffices to recall from Propositions
  \ref{p:poiss-hlh} and \ref{p:vtop-hlh} that, in the Poisson and top
  polyvector field cases, $H(X) \cdot \cO_X = LH(X) \cdot \cO_X$ for
  all affine $X$.

  The final statement follows in the same manner.
\end{proof}
  On the other hand, $P(X)$ need not be $\caD$-localizable: see \S
  \ref{ss:pvfd-loc} for a detailed discussion.
\begin{remark}\label{r:zarlh}
  We note that, in general, $\caH_X$ is \emph{not} a sheaf,
  and neither is $\cLH_X$.  

  For an example where $\caH_X$ and $\cLH_X$ are not sheaves, let $X$
  be the complement in $\bA^3$ of the plane $x+y=0$, equipped with the
  Poisson structure given by the potential $f(x,y,z) =
  \frac{xy}{x+y}$, i.e.,
\[
\{x,y\}=0, \{y,z\}= \frac{y^2}{(x+y)^2}, \{z,x\} = \frac{x^2}{(x+y)^2}.
\]
Consider the vector field $\xi := (x+y)^{-2} \partial_z$.  This is
regular, and on the open set where $x \neq 0$, it is the Hamiltonian
vector field of $x^{-1}$, and on the open set where $y \neq 0$, it is
the Hamiltonian vector field of $-y^{-1}$.  But it is not globally
Hamiltonian, since if $\xi = \xi_f$ for some $f \in \cO_X$, then
$\frac{1}{x^2+y^2} = \{f,z\}$, but the RHS must live in
$\frac{(x^2,y^2)}{x^2+y^2}$, where $(x^2,y^2)$ is the ideal generated
by $x^2$ and $y^2$.  This is impossible.

The same argument shows that $\xi$ is not given by a global one-form:
otherwise, $\frac{1}{x^2+y^2} = f_1 \{x, z\} + f_2 \{y, z\}$ for some
$f_1, f_2 \in \cO_X$, and one concludes as before that this is
impossible.

On the other hand, in the case that $X$ is generically symplectic, it
follows that $\caH_X$ and $\cLH_X$ are sheaves, since in this case any
vector field which is Hamiltonian in some neighborhood must be given
by a unique Hamiltonian function up to locally constant functions, and
this is then defined and Hamiltonian on the regular locus of that
function (and similarly in the locally Hamiltonian case).

Note similarly that, in the case of a variety with a top polyvector
field $\Xi$, $\caH_X$ and $\cLH_X$ are sheaves, since if $\Xi$ is
nonzero, then on its nonvanishing locus a Hamiltonian vector field is
once again given by a unique Hamiltonian.
\end{remark}

\begin{remark}\label{r:f-loc}
  In the examples above, the presheaves also are equipped naturally
  with spaces of sections on formal neighborhoods $\hat X_x$ of every
  point $x \in X$; the presheaf condition requires only that these
  contain the restrictions of sections on open subsets containing
  $x$. Thus it makes sense to define the notion of \emph{formal
    $\caD$-localizability},
  i.e., that, for every open affine $U$ and $x \in U$,
\begin{equation}\label{e:f-d-loc}
\mfv(\hat X_x) \caD_{\hat X_x} = \mfv(U)|_{\hat X_x} \caD_{\hat X_x}.
\end{equation}
Formal localizability implies usual localizability: indeed, if $\mfv$
is formally localizable, and $\xi \in \mfv(U')$ for some $U' \subseteq
U$, then at every $x \in U'$, it follows that $\xi|_{\hat
  X_x} \in \mfv(U) \cdot \caD_{\hat X_x}$, and hence $\xi \in \mfv(U)
\cdot \caD_U$, by Lemma \ref{l:vd-sheaf}.

Theorem \ref{t:h-loc} extends to show that, under the assumptions
there, $\mfv$ is formally $\caD$-localizable, by formally localizing
the embedding $X \to \bA^n$ to $\hat X_x \to \hat \bA^n_0$. Then, the
same proof applies.  We conclude as before that the presheaves of
(locally) Hamiltonian vector fields are formally $\caD$-localizable,
as well as $\Vect(X)$ and all coherent subsheaves thereof.
\end{remark}

\subsection{$\caD$-localizability of Poisson vector
  fields}\label{ss:pvfd-loc}
An interesting question raised in the previous subsection is whether
$P(X)$ is $\caD$-localizable.  This turns out to have an interesting
answer, which we discuss here. The material of this subsection will
not be needed for the rest of the paper, and our motivation is partly
to illustrate the nontriviality of $\caD$-localizability. We will
first give the statements and examples, and postpone the proofs of the
propositions to the end of the subsection, for the purpose of
emphasizing the statements and counterexamples to their
generalization.
\begin{proposition}\label{p:ploc}
Let $X$ be an irreducible affine Poisson
  variety on which $P(X)$ flows  incompressibly. If $P(X)$ is
  $\caD$-localizable, then the generic rank of $P(X)$ 
  must equal that of $P(U)$ for every open subset $U \subseteq X$.  

  Conversely, suppose that $X$ is a smooth affine Poisson variety on
  which the rank of $P(U)$ equals that of $H(U)$ everywhere, for all
  affine open $U \subseteq X$, and that this rank is constant on $X$. Then,
  for all affine open $U \subseteq X$, one has an equality of
  $\cO$-saturations $P(U)^{os} = H(U)^{os}$.  Hence, $P(X)$ is
  $\caD$-localizable.
\end{proposition}
The assertion of the second paragraph 
follows from the more general
\begin{lemma}\label{l:inc-rank-sat}
  Suppose $\mfv \subseteq \mathfrak{w}$ is an inclusion of Lie
  algebras of vector fields on a smooth affine variety $X$.  Suppose
  that the rank of $\mfv$ is constant and equals that of
  $\mathfrak{w}$ everywhere, and moreover that $\mathfrak{w}$ flows
  incompressibly.  Then $\mfv^{os} = \mathfrak{w}^{os}$. In
  particular, $M(X,\mfv) = M(X, \mathfrak{w})$.
\end{lemma}
\begin{proof}
  Since the ranks of $\mfv$ and $\mathfrak{w}$ are constant and equal,
  we conclude that, for every $x \in X$, there exists an open subset
  $U \subseteq X$ containing $x$ such that $\cO_U \cdot \mfv|_U = \cO_U \cdot
  \mathfrak{w}$, and hence in fact $\cO_X \cdot \mfv = \cO_X \cdot
  \mathfrak{w}$.  Now, if $\mathfrak{w}$ flows incompressibly, and
  hence also $\mfv$, then $\mfv^{os}=\mathfrak{w}^{os}=$ the subspace
  of $\cO_X \cdot \mfv$ of incompressible vector fields, by Proposition
  \ref{p:inc-os}.
\end{proof}
\begin{remark}
Lemma \ref{l:inc-rank-sat} generalizes to affine schemes of finite type, if
we replace the rank condition by the condition that $\cO_X \cdot \mfv =
\cO_X \cdot \mathfrak{w}$.
\end{remark}
We also give a localizability result that does not require
$X$ to be smooth, in the situation of Remark \ref{r:n-cod2-poiss}, where
$P(X) = LH(X^\circ)$ for $X^\circ$ the smooth locus of $X$.
\begin{notation}\label{ntn:gdr}
  Given any not necessarily affine scheme $Y$, we will let
  $H^\bullet_{DR}(Y) := H^\bullet(\Gamma(\tilde \Omega_Y))$ denote the
  cohomology of the complex of global sections of de Rham differential
  forms modulo torsion.
\end{notation}
For all $x \in X$,
let $\cO_{X,x}$ be the \emph{uncompleted} local ring of $X$ at $x$.  
\begin{proposition}\label{p:ploc-is}
Suppose $X$ is Poisson, normal, and symplectic on its smooth locus. Let $S$
be its singular locus.
\begin{enumerate}
\item[(i)] For every $s \in S$, let $E_s \subseteq S$ be the union of all irreducible components of $S$ containing $s$. 
Suppose that, for all $s \in S$, the natural map
\begin{equation}\label{e:is-l-surj}
H^1_{DR}(X \setminus E_s) \oplus H^1_{DR}(\Spec \cO_{X,s}) \to
H^1_{DR}( \Spec \cO_{X,s}  \setminus E_s)
\end{equation}
is surjective.  Then, $X$ is Poisson localizable. Moreover, for all $s \in S$,
\begin{equation}\label{e:imp-ploc}
P(\Spec \cO_{X,s}) =P(X)|_{\Spec \cO_{X,s}} \cdot \cO_{X,s} + LH(\Spec \cO_{X,s}).
\end{equation}
\item[(ii)] Now suppose that $S$ is finite and $\bk=\bC$. Then the
  hypothesis of (i) is satisfied if, for all $s \in S$ and all affine
  Zariski open neighborhoods $U$ of $s$, the natural map on
  \emph{topological} cohomology,
\begin{equation}\label{e:is-l-top-surj}
H^1_{\tpl}(X \setminus \{s\}) \oplus H^1_{\tpl}(U) \to H^1_{\tpl}(U \setminus \{s\})
\end{equation}
is surjective. In particular, in this case, $X$ is Poisson localizable,
and \eqref{e:imp-ploc} holds.
\end{enumerate}
\end{proposition}
\begin{example}
  When $X$ has a contracting $\bG_m$ action (where this is the
  multiplicative group), i.e., $\cO_X$ is nonnegatively graded with
  $\bk$ in degree zero, then $H^\bullet(X)=\bk$, and in particular
  $H^2(X)=0$. Therefore, if $X$ has an isolated singularity at the
  fixed point for the action, using the Mayer-Vietoris sequence
  \eqref{e:mv-ord} below, \eqref{e:is-l-top-surj} is surjective.
  Thus, if $X$ is also normal and generically symplectic, the
  conditions of the proposition are satisfied, so $P(X)$ is
  $\caD$-localizable. Also, in this case, $P(U) = P(X)|_U + LH(U)$ for
  all open sets (and for those $U$ which don't contain the singularity
  we have $P(U)=H(U)$, since then $U$ is symplectic).

  For such an example where $P(U)/LH(U)$ is nonzero, let $X$ be the
  locus $x^3+y^3+z^3=0$ (or a more general elliptic singularity); then
  $P(U)/LH(U)$ is generated by the Euler vector field in $P(X)$ for
  all open affine $U$.
\end{example}
\begin{example}
  Here is a simple example of a non-normal $X$ which is not
  $\caD$-localizable. Suppose $X = \Spec \bk[x^2,x^3,y,xy]$ and
  $\{x,y\}=y$. This is generically symplectic but not normal. Then we
  claim that every global Poisson vector field vanishes at
  $y=0$. Indeed, $\xi = f \partial_x + g \partial_y$ is Poisson if and
  only if $\frac{\partial f}{\partial x} + \frac{\partial g}{\partial
    y} = \frac{g}{y}$. Writing $g=yh$, we obtain $\frac{\partial
    f}{\partial x} + y\frac{\partial h}{\partial y} = 0$.  So $y \mid
  \frac{\partial f}{\partial x}$. Since $\xi$ is a vector field on
  $X$, $f$ vanishes at the origin, and hence $y \mid f$.  This proves
  the claim. On the other hand, in the complement $U$ of any
  hyperplane through the origin, $\partial_x$ is a Poisson vector
  field; but this can only be in $P(X) \cdot \caD_U$ when the
  hyperplane was $y=0$. Thus $P(X)$ is not $\caD$-localizable.
\end{example}
\begin{remark}
  In the case $X$ is smooth, if it has finitely many symplectic
  leaves, it is in fact symplectic.  However, there are many cases
  where $X$ is smooth and generically symplectic, and $P(X)$ has
  finitely many leaves even though $X$ has infinitely many symplectic
  leaves; e.g., $\pi = x \partial_x \wedge \partial_y$ on $\bA^2$, as
  mentioned in \S \ref{ss:poisson}.
\end{remark}
We give an example where $X$ is smooth but $P(X)$ is not $\caD$-localizable:
\begin{example} Let $\mfg$ be the Lie algebra $\mfg :=
  \mathfrak{sl}_2$ and let $X = \mfg^*$, equipped with the induced
  Poisson bracket on $\cO_X = \Sym \mfg$.  Then, all global Poisson
  vector fields are Hamiltonian, since $H^1(\mfg, \Sym \mfg) = 0$
  (this implies that all derivations $\mfg \to \Sym \mfg$ are inner,
  and hence all derivations of $\Sym \mfg$, i.e., vector fields on
  $\mfg^*$, are Hamiltonian).  It is clear that the Poisson bivector
  has rank two, except at the origin, where the rank is zero; hence
  this is the rank of $P(X)$.  However, we claim that the rank of the
  space of generic Poisson vector fields is three.  Indeed, write
  $\mfg = \langle e, h, f \rangle$ with the standard bracket $[e,f]=h,
  [h,e]=2e, [h,f]=-2f$.  So $e, h, f \in \cO_X$ are linear
  coordinates.  Let $C = 2ef + \frac{1}{2} h^2 \in \cO_X$ be the
  Casimir function, so $\{C,g\}=0$ for all $g \in \cO_X$.  Then, if we
  localize where $e\neq 0$, we can consider the coordinate system
  $(e,h,C)$ and take the directional derivative in the $C$ direction,
  which in the original coordinates $(e,h,f)$ is $\xi :=
  \frac{1}{2e}\partial_f$.  Since the Poisson bivector field is
  tangent to the planes where $C$ is constant, this vector field is
  Poisson, which is also immediate from explicit computation (it is
  enough to check that $\{\xi(x), y\} + \{x, \xi(y)\} = \xi\{x,y\}$
  for $x,y \in \cO_X$, which clearly reduces to the case $x=f, y=h$,
  where $\{\xi(f), h\} = \frac{1}{e} = \xi(2f) = \xi\{f,h\}$.)
\end{example}
\begin{proof}[Proof of Proposition \ref{p:ploc}]
  By incompressibility and Corollary \ref{c:inc}, the generic rank of
  $P(X)$ equals $2 \dim X$ minus the dimension of the singular support
  of $M(U',P(X)|_{U'})$ for small enough open subsets $U'$ (viewing
  $P(X)|_{U'}$ as a vector space).  Thus $\caD$-localizability implies
  that this must also equal the generic rank of $P(U)$ for every open
  subset $U \subseteq X$.

  The second statement follows from Lemma \ref{l:inc-rank-sat},
  provided we can show that $P(X)$ flows incompressibly. By
  assumption, $P(X)$ flows parallel to the symplectic leaves.  But, to
  be Poisson, such vector fields must preserve the symplectic form
  along the leaves, and hence they are incompressible along the
  leaves.  Thus, as for $H(X)$ (see Example \ref{ex:poiss-finsym}),
  one concludes that $P(X)$ flows incompressibly on $X$.
\end{proof}
\begin{proof}[Proof of Proposition \ref{p:ploc-is}] 
(i) Suppose that $\xi \in P(\Spec \cO_{X,s})$.  As explained in Remark 
\ref{r:n-cod2-poiss}, this means that $\xi = \eta_\alpha$ where $\alpha$ is
a closed one-form on $\Spec \cO_{X,s} \setminus E_s$. By the hypothesis \eqref{e:is-l-surj}, we
can write
\begin{equation}\label{e:ploc-is-pf1}
\alpha = \alpha_{\Spec \cO_{X,s}} + \alpha_{X \setminus E_s} + df,
\end{equation}
where $\alpha_{\Spec \cO_{X,s}}$ is a closed one-form modulo torsion
on $\Spec \cO_{X,s}$,
$\alpha_{X \setminus E_s}$ is a closed one-form modulo torsion
on $X \setminus E_s$, and $f \in
\Gamma(\Spec \cO_{X,s} \setminus E_s)$.  By normality, $f \in \cO_{X,s}$.  Thus
$\xi_f \in H(\Spec \cO_{X,s})$.  Note that $\eta_{\alpha_{\Spec \cO_{X,s}}} \in LH(\Spec \cO_{X,s})$, by definition. As in Remark \ref{r:n-cod2-poiss}, 
we obtain that $\eta_{X \setminus E_s} \in P(X)$. Therefore, applying the operation $\beta \mapsto \eta_\beta$ to both sides of \eqref{e:ploc-is-pf1}, we
obtain \eqref{e:imp-ploc}, since $\xi$ was arbitrary.

As a consequence, we deduce that $P(\Spec \cO_{X,s}) \subseteq P(X)|_{\Spec \cO_{X,s}} \cdot \cO_{X,s}$.  Now, $s \in X$ was an arbitrary singular point.  At smooth points $x \in X$, we have $P(\Spec \cO_{X,x}) = H(\Spec \cO_{X,x}) \subseteq
H(X)|_{\Spec \cO_{X,x}} \cdot \cO_{X,x}$. 

Now, for arbitrary open affine $U \subseteq X$, $P(X)|_U \cdot \caD_U$
is a sheaf on $U$, by Lemma \ref{l:vd-sheaf}.  By the above,
$P(U)|_{\Spec \cO_{X,x}} \subseteq P(X)|_{\Spec \cO_{X,x}} \cdot
\caD_{X,x}$, where the latter is the Zariski localization of $\caD_X$
at $x$. By Lemma \ref{l:vd-sheaf}, $P(X)|_{\Spec \cO_{X,x}} \cdot
\caD_{X,x} = (P(X)|_U \cdot \caD_U)|_{\Spec \cO_{X,x}}$.  We conclude
that, for all $x \in U$,
\[
P(U)|_{\Spec \cO_{X,x}} \subseteq (P(X)|_U \cdot \caD_U)|_{\Spec \cO_{X,x}},
\]
and since $P(X)|_U \cdot \caD_U$ is a sheaf, this implies that $P(U)
\subseteq P(X)|_U \cdot \caD_U$.  As $U$ was arbitrary, we conclude
Poisson localizability.

(ii) In order to prove \eqref{e:is-l-surj}, it suffices to prove the
statement when $\Spec \cO_{X,s}$ is replaced by sufficiently small
Zariski open neighborhoods $U$ of $s$.  This is because every closed
one-form modulo torsion in $\Spec \cO_{X,s} \setminus E_s$ is actually
regular on $U \setminus E_s$ for some Zariski open neighborhood $U$ of
$s$, and we are free to shrink it.

Now, assuming that $S$ is finite, $E_s = \{s\}$ for all $s \in S$. By
the preceding paragraph, it suffices to show that
\eqref{e:is-l-top-surj} implies that the map
\begin{equation}\label{e:is-l-usurj}
H^1_{DR}(X \setminus \{s\}) \oplus H^1_{DR}(U) \to
H^1_{DR}(U \setminus \{s\})
\end{equation}
is surjective. To see this, we first note that, for $Y$ smooth but not necessarily affine, we have an isomorphism by Grothendieck's theorem,
\[
\mathbf{H}^\bullet_{DR}(Y) \cong H^\bullet_{\tpl}(Y),
\]
where $\mathbf{H}^\bullet_{DR}(Y)$ denotes the \emph{hypercohomology}
of the complex of sheaves $\Omega_Y^\bullet = \tilde \Omega_Y^\bullet$.

Next, there is a natural map $H^1_{DR}(U) \to H^1_{\tpl}(U)$, obtained
by integrating along cycles; one can slightly perturb a closed path in
$U$ to miss the isolated singularities of $U$, and integrating against
a one-form on $U$ which is closed mod torsion (hence closed when
restricted to the smooth locus of $U$) produces a well-defined answer,
which depends only on the homology class in $U$ of the closed path.

Then, the restriction map $H^1_{DR}(U) =\bH^1_{DR}(U) \to \bH^1_{DR}(U \setminus \{s\}) = H^1_{\tpl}(U \setminus \{s\})$ factors through the map $H^1_{DR}(U) \to H^1_{\tpl}(U)$, which is surjective
 by the main result of \cite{BH-drcas}.

Then, \eqref{e:is-l-top-surj} implies that we have a surjection
\begin{equation}\label{e:is-l-hdr-surj}
\bH^1_{DR}(X \setminus \{s\}) \oplus H^1_{DR}(U) \to
\bH^1_{DR}(U \setminus \{s\}),
\end{equation}
where here we note that $H^1_{DR}(U) = \bH^1_{DR}(U)$ since $U$ is affine.

Since $X = (X \setminus \{s\}) \cup U$, we have an exact Mayer-Vietoris
sequence on hypercohomology of the triple $(X, X \setminus \{s\}), U)$, which in part takes the form
\begin{equation}\label{e:mv-hyper}
\bH^1_{DR}(X \setminus \{s\}) \oplus H^1_{DR}(U) \to
\bH^1_{DR}(U \setminus \{s\}) \to H^2_{DR}(X) \to \bH^2_{DR}(X \setminus \{s\})
\oplus H^2_{DR}(U).
\end{equation}
By \eqref{e:is-l-hdr-surj}, the first map is surjective, and hence the last map is injective.

We also have a Mayer-Vietoris sequence for ordinary $H^\bullet_{DR}$,
associated to the exact sequence of complexes of global sections,
\[
0 \to \Omega^\bullet_X \into \Gamma(\Omega^{\bullet}_{X \setminus
  \{s\}}) \oplus \Omega^\bullet_U \onto \Gamma(\Omega^{\bullet}_{U
  \setminus \{s\}}).
\]
This has the form
\begin{equation}\label{e:mv-ord}
  H^1_{DR}(X \setminus \{s\}) \oplus H^1_{DR}(U) \to
  H^1_{DR}(U \setminus \{s\}) \to H^2_{DR}(X) \to H^2_{DR}(X \setminus \{s\})
  \oplus H^2_{DR}(U).
\end{equation}
Now, the final map in \eqref{e:mv-hyper} factors through the final map
in \eqref{e:mv-ord} (since $X$ is affine).  Therefore the last map in
\eqref{e:mv-ord} must also be injective. We conclude that the first
map of \eqref{e:mv-ord}, which is the same as \eqref{e:is-l-usurj},
is surjective. This completes the proof.
\end{proof}
\subsection{Formal $\caD$-localizability of Poisson vector fields}
It turns out that formal $\caD$-localizability of Poisson vector
fields is a stronger condition, which implies (in the incompressible
case) that $X$ is generically symplectic.
\begin{proposition}
\label{p:p-floc}
  If $X$ is irreducible affine
  Poisson and $P(X)$ flows incompressibly, then if $P(X)$ is formally
  $\caD$-localizable, then $X$ must be generically symplectic.
\end{proposition}
Note that this in particular implies that the condition of
Proposition \ref{p:ploc} is satisfied: $P(U)$ has generic rank equal to
$\dim X$ for all $U$.
\begin{proof}[Proof of Proposition \ref{p:p-floc}]
  Suppose that $X$ is not generically symplectic.  Then, in the
  neighborhood of some sufficiently generic smooth point, $\hat X_x
  \cong V \times V'$ as a formal Poisson scheme, where $V$ is a
  symplectic formal polydisc and $V'$ is a positive-dimensional formal
  polydisc with the zero Poisson bracket.  So $P(\hat X_x) = P(V)
  \otimes \cO_{V'} \oplus \Vect(V')$.  This is evidently not
  incompressible since $V'$ is positive-dimensional.  Thus $M(\hat
  X_x, P(\hat X_x)) = 0$. However, if we assume $P(X)$ flows
  incompressibly, then $M(X, P(X))|_{\hat X_x} \neq 0$ for
  sufficiently generic $x$ (with $P(X)$ here the constant sheaf). Thus
  $P(X)$ is not formally localizable.
\end{proof}
We can also give a positive result parallel to Proposition
\ref{p:ploc-is}:
\begin{proposition}\label{p:fploc-is}
Suppose $X$ is affine Poisson, normal, and symplectic on its smooth locus. Let $S$
be its singular locus.
\begin{enumerate}
\item[(i)] For every $s \in S$, let $E_s \subseteq S$ be the union of all irreducible components of $S$ containing $s$. 
Suppose that, for all $s \in S$, the natural map
\begin{equation}\label{e:is-fl-surj}
H^1_{DR}(X \setminus E_s) \oplus H^1_{DR}(\Spec \hat \cO_{X,s}) \to
H^1_{DR}( \Spec \hat \cO_{X,s}  \setminus E_s)
\end{equation}
is surjective.  Then, $X$ is formally
Poisson localizable. Moreover, for all $s \in S$,
\begin{equation}\label{e:imp-fploc}
P(\Spec \hat \cO_{X,s}) =P(X)|_{\Spec \hat \cO_{X,s}} \cdot \hat \cO_{X,s} + 
LH(\Spec \hat \cO_{X,s}).
\end{equation}
\item[(ii)] Suppose that $S$ is finite and $\bk=\bC$. Then the
  hypothesis of (i) is satisfied if, for sufficiently small neighborhoods
  $U$ of $s$ in the complex topology, 
  $H^1_{\tpl}(X \setminus \{s\}) \to H^1_{\tpl}(U \setminus \{s\})$ is surjective.
 In particular, in this case, $X$ is formally Poisson
localizable, and \eqref{e:imp-fploc} holds.
\end{enumerate}
\end{proposition}
\begin{remark}\label{r:is-l-top-surj}
  The condition of (ii) is equivalent to asking that $H^1_{\tpl}(X
  \setminus \{s\}) \to H^1_{\tpl}(U \setminus \{s\})$ be surjective
  for any fixed contractible neighborhood $U$ of $s$ (whose existence
  was proved in 
  \cite{Gil-eavlc}).
  Thus, the condition of (ii) is the same as that of
  \eqref{e:is-l-top-surj}, except replacing Zariski open subsets by
  analytic neighborhoods, and using holomorphic functions rather than
  algebraic functions.
\end{remark}
\begin{proof}[Proof of Proposition \ref{p:fploc-is}]
The proof of part (i) of the proposition is the same as in Proposition
\ref{p:ploc-is}, except replacing $U$ by $\hat X_x$. We omit the
details. Note that, when $x \notin S$, one has $P(\hat X_x) = H(\hat
X_x)$, since then $\hat X_x$ is symplectic.  

For part (ii), we use holomorphic functions and analytic neighborhoods
and results about them contained in \S \ref{ss:comp}
below. As in Proposition \ref{p:ploc-is}, for every analytic
neighborhood $U$ of $s$, the assumption of (ii) together with
Grothendieck's theorem implies that the map on hypercohomology,
\[
\bH^1_{DR}(X
\setminus \{s\})\to \bH^{1,\an}_{DR}(U \setminus \{s\}),
\]
 is surjective.  Using the
Mayer-Vietoris sequence for the exact sequence of complexes of sheaves (\eqref{e:mv} below for $Y=X$, $Z = \{s\}$, and $V = U$),
we conclude that the map 
\[
H^2_{DR}(X) \to
\bH^2_{DR}(X \setminus \{s\}) \oplus H^{2,\an}(U)
\] 
is injective. This map factors
through the map from ordinary cohomology to hypercohomology, so we
conclude that the 
map 
\[
H^2_{DR}(X) \to
H^2_{DR}(X \setminus \{s\}) \oplus H^{2,\an}(U)
\] 
is also injective.
Using the Mayer-Vietoris sequence for ordinary cohomology (using
the global sections of \eqref{e:mv}, which is an exact sequence
of complexes since $X$ is affine),
we conclude that 
\begin{equation}\label{e:h1x}
H^1_{DR}(X \setminus \{s\}) \oplus H^{1,\an}_{DR}(U) \to H^{1,\an}_{DR}(U
\setminus \{s\})
\end{equation}
is surjective.
Then, by Theorem \ref{t:form-an-comp} below, we conclude that
\begin{equation}
H^1_{DR}(X \setminus \{s\}) \oplus H^1_{DR}(\hat X_s) \to H^1_{DR}(\hat X_s \setminus \{s\})
\end{equation}
is also surjective, as desired. \qedhere
\end{proof}
\begin{remark}
  In fact, we did not need the full strength of Theorem
  \ref{t:form-an-comp} below, but only the fact that the maps
  $H^1_{DR}(U) \to H^1_{DR}(\hat X_s)$ and $H^1_{DR}(U \setminus
  \{s\}) \to H^1_{DR}(\hat X_s \setminus \{s\})$ are surjective.  At
  least the first fact can be proved in an elementary way by lifting
  closed formal differential forms to closed analytic differential
  forms, and does not require resolution of singularities as used in
  the proof of Theorem \ref{t:form-an-comp}.
\end{remark}
\begin{example} 
  Here is an example of a surface with an isolated singularity, which
  is normal and symplectic away from the singularity, which is
  $\caD$-localizable (in fact satisfying \eqref{e:is-l-top-surj}) but
  not formally $\caD$-localizable (so in particular not satisfying
  \eqref{e:is-fl-surj}). This example was pointed out to us by
  J. McKernan.

  Let $E \subseteq \bP^2$ be a smooth cubic curve.  Then, under the
  intersection pairing on $\bP^2$, $E \cdot E = 9$. Now, blow up
  $\bP^2$ at twelve generic points of $E$. Let $Y$ be the resulting
  projective surface, and let $E' \subseteq Y$ be the proper transform
  of $E$. Then $E' \cdot E' = 9 - 12 = -3$, so we can blow down $E'$
  to obtain a new surface, call it $Z$, where the image of $E'$ is a
  singular point, call it $s$, whose formal neighborhood $\hat Z_s$ is
  isomorphic to the cone over an elliptic curve.

  Note that $H^1_{\tpl}(Z \setminus \{s\}) \cong H^1_{\tpl}(Y
  \setminus E') \cong H^1_{\tpl}(\bP^2 \setminus E) = 0$, since $E
  \subseteq \bP^2$ has a nontrivial normal bundle.

  Next, embed $Z$ into projective space $\bP^N$ of some dimension $N >
  2$.  Let $C \subseteq Z$ be the intersection of $Z$ with a generic
  hyperplane, and let $X := Z \setminus C$ be the resulting affine
  surface.  Since $\cO(C)$ is (very) ample, $C$ has a nontrivial
  normal bundle.  Hence, the restriction map induces isomorphisms
  $H^1_{\tpl}(Z) \iso H_{\tpl}^1(X)$ and $H^1_{\tpl}(Z \setminus
  \{s\}) \iso H^1_{\tpl}(X\setminus \{s\})$.  In particular, these are
  zero as well.

  Thus, $\bH^1_{DR}(X \setminus \{s\}) = 0$.  We claim that
  $H^1_{DR}(X \setminus \{s\}) = 0$ as well.  
More generally, this follows from the following statement:
\begin{lemma}\label{l:h-bh-inj}
Let $V$ be a scheme or complex analytic space. Then the map
$H^1_{DR}(V) \to \bH^1_{DR}(V)$ is injective.
\end{lemma}
We remark that, in the case $V$ is a smooth variety (as with $V = X \setminus
\{s\}$ above), by Grothendieck's theorem we can replace
$\bH^1_{DR}(V)$ by the topological first cohomology of $V$, and then
the statement follows because, if an algebraic or analytic one-form is
the differential of a smooth ($C^\infty$) function, then the
function must actually be algebraic (or analytic).
\begin{proof}[Proof of Lemma \ref{l:h-bh-inj}]
Consider the spectral sequence $H^i(R^j \Gamma (\Omega_V)) \Rightarrow \bH_{DR}^{i+j}(V)$.  In total degrees $\leq 2$, the second page has the form
\[
H^0_{DR}(V) \to H^1_{DR}(V) \oplus H^0(R^1 \Gamma(\Omega_V)) \to
H^2_{DR}(V) \oplus H^1(R^1 \Gamma(\Omega_V)) \oplus
H^0(R^2 \Gamma(\Omega_V)).
\]
The first map above is zero,
and the restriction of the second map above to $H^1_{DR}(V)$ is zero.
Therefore the summand of $H^1_{DR}(V)$ maps injectively to a summand of
the third page of the spectral sequence.  The same argument shows that,
at every page, $H^1_{DR}(V)$ maps injectively to the next page, so 
the map $H^1_{DR}(V) \to \bH^1_{DR}(V)$ is injective.
\end{proof}

Now, since $X \setminus \{s\}$ is symplectic, all global Poisson
vector fields are locally Hamiltonian given by a global closed
one-form.  By the above, $H^1_{DR}(X \setminus \{s\}) = 0$, so that
locally Hamiltonian vector fields are Hamiltonian.  Since global
functions on $X \setminus \{s\}$ coincide with those on $X$, all
global Poisson vector fields on $X$ are Hamiltonian.

  On the other hand, not all Poisson vector fields on $\hat X_s$ are
  Hamiltonian, since $\hat X_s$ is isomorphic to the formal
  neighborhood of the vertex in the cone over an elliptic curve, and
  there one has the Euler vector field which is not Hamiltonian.
  Hence, $P(X)$ is not formally $\caD$-localizable.  (In fact, $P(X)$
  is not \'etale-locally $\caD$-localizable either, since the Euler
  vector field exists in an \'etale neighborhood of $x$, or
  equivalently in the Henselization of the local ring at $x$.)

  On the other hand, we claim that $P(X)$ is $\caD$-localizable, and
  in fact that \eqref{e:is-l-top-surj} holds.  Let $U \subseteq X$ be
  any affine open subset containing $s$.  Since $Z$ is rational (as
  $Y$, and hence $Z$, is birational to $\bP^2$), so is $U$.  Now, we
  claim that the map $H^1_{DR}(U) 
\to H^1_{DR}(U
  \setminus \{s\})$ is surjective.  Consider the sequence \eqref{e:mv-formal} for the pair $(U, \{x\})$: this yields the exact sequence
\[
H^1_{DR}(U) \to H^1_{DR}(U \setminus \{s\}) \oplus H^1_{DR}(\hat U_s)
\to H^1_{DR}(\hat U_s \setminus \{s\}).
\]
It suffices to show that the map $H^1_{DR}(U \setminus \{s\}) \to
H^1_{DR}(\hat U_s \setminus \{s\})$ is zero.  By Lemma
\ref{l:h-bh-inj} above, this is equivalent to showing that the map
$H^1_{DR}(U \setminus \{s\}) \to \bH^1_{DR}(\hat U_s \setminus \{s\})$
is zero. This map factors through the hypercohomology of any punctured
neighborhood of $s$ contained in $U \setminus \{s\}$, which by
Grothendieck's theorem is the same as the topological cohomology of
that punctured neighborhood.  Such punctured neighborhoods, for
sufficiently small contractible $U$, are homotopic to nontrivial
$S^1$-bundles over an elliptic curve, and their fundamental group is
isomorphic to that of the elliptic curve. If the map $H^1_{DR}(U
\setminus \{s\}) \to \bH^1_{DR}(\hat U_s \setminus \{s\})$ were
nonzero, then a nontrivial period of the elliptic curve would be
computable by integrals of closed algebraic one-forms on $U \setminus \{s\}$.
However, as remarked, $U$ is rational.  Thus this would imply that a
nontrivial period of the elliptic curve were computable by integrals of 
rational closed one-forms along contours in $\bC^2$.  This is
well-known to be impossible, since these periods are given by
transcendental hypergeometric functions with infinite monodromy.
Thus, $P(X)$ is $\caD$-localizable.  (Note that this paragraph also
gives another proof that $P(X)$ is not formally $\caD$-localizable, and
in fact that $P(U)$ is not formally $\caD$-localizable for every open affine
neighborhood $U$ of $s$:
these periods are computable in a formal neighborhood of $s$, but by
the above, they are not computable using global closed
one-forms. Passing from closed one-forms to Poisson vector fields via
the symplectic form on $U \setminus \{s\}$, this yields that $P(U)$ is
not formally $\caD$-localizable.)
\end{example}
\begin{example}
  We give an example where $X$ is smooth and $\caD$-localizable but
  not formally $\caD$-localizable.  By Propositions \ref{p:ploc} and
  \ref{p:p-floc}, one way this happens is if the rank of $P(U)$
  equals that of $H(U)$ and is constant but less than the dimension of
  $X$ (which in particular is not generically symplectic).

  Let $X = (\bA^\times)^3 = \Spec \bk[x^{\pm 1}, y^{\pm 1}, z^{\pm
    1}]$ with the Poisson bracket $\{x,y\}=xyz$ and
  $\{x,z\}=\{y,z\}=0$. The Poisson bivector, call it $\pi$, is
  $\pi=xyz \partial_x \wedge \partial_y$.  Then $H(U)$ has rank two
  everywhere, for every open affine subset $U \subseteq X$. We claim
  that any rational Poisson vector field on $X$ annihilates $z$.
  Therefore, the rank of every vector field in $P(U)$ is also
  everywhere two, for every open subset $U \subseteq X$, as desired.
 
  To prove that every rational Poisson vector field annihilates $z$,
  it is enough to assume that $\bk=\bC$.  Let $\xi$ be a rational
  Poisson vector field and let $c \in \bC$ be such that it does not
  have a pole at $z=c$.  Then the irregular locus of $\xi$ in
  $\{z=c\}$ is an algebraic curve in $\bA^2$.  One can show that such
  a curve must avoid a real two-torus $T=\{|x|=r, |y|=s\}$, and then
  $\int_{T \times \{z\}} \pi^{-1}$ is a nonzero constant multiple of
  $\frac{1}{z}$, for $z$ in some topological neighborhood of $c$
  (which is independent of $r$ and $s$).  Since $\xi$ preserves $\pi$,
  one concludes that it must be parallel to the level sets of $z$,
  i.e., it annihilates $z$.
\end{example}
\begin{remark}
  We can also give an elementary algebraic proof that, in the above
  example, every rational Poisson vector field annihilates $z$.  Any
  rational Poisson vector field must send $z$ to a rational function
  of $z$, since these are all the rational Casimirs.  Moreover, any
  such vector field is still Poisson after multiplying by an arbitrary
  rational function of $z$.  Hence, if such a vector field exists
  which does not annihilate $z$, then there must be one of the form
  $\partial_z + f \partial_x + g \partial_y$ for some rational
  functions $f,g$ on $X$.

  On the other hand, we can explicitly write one such non-rational
  vector field, $\xi := \partial_z + \frac{x \log x}{z} \partial_x$. This
  vector field is best understood by writing the Poisson bracket in
  coordinates $(u,v,z)=(\log x, \log y, z)$, as $\{u, v\} = z$, and
  the vector field as $\xi=\partial_z + \frac{u}{z} \partial_{u}$.

  Thus, given a rational vector field $\partial_z + f \partial_x +
  g \partial_y$, taking the difference, we would obtain a non-rational
  Poisson vector field of the form $\eta = -\frac{x \log x}{z} \partial_x +
  f \partial_x + g \partial_y$.  But no such vector field can be
  Poisson, since a vector field parallel to the symplectic leaf is
  Poisson if and only if its symplectic divergence, $-L_{\eta}\pi / \pi$, 
vanishes.   But this symplectic divergence is 
\[
x\partial_x(-(\log x)/z) + x \partial_x(f/x) + y\partial_y(g/y)
= -\frac{1}{z} + x \partial_x(f/x) + y\partial_y(g/y).
\]
Dividing by $x$ and taking the residue at $x=0$, we would obtain
\[
-\frac{1}{z} + y \partial_y(g/y)|_{x=0} = 0.
\]
Now, dividing by $y$ and taking the residue at $y=0$, we would obtain
\[
=- \frac{1}{z} = 0,
\]
which is a contradiction.
\end{remark}

\subsection{Analytic-to-formal comparison for de Rham cohomology}
\label{ss:comp}
\subsubsection{Preliminaries on analytic forms and Mayer-Vietoris sequences}
We will need to use holomorphic differential forms, on an algebraic
or (complex) analytic variety $Y$ which need not be affine.
\begin{definition}\label{d:an}
Let $Y$ be an algebraic or analytic variety over $\bk =
  \bC$. Let $\Omega_Y^{\bullet,\an}$ denote  the complex of sheaves
of holomorphic K\"ahler
  differential forms, and $\tilde \Omega_Y^{\bullet,\an}$ its quotient modulo torsion.  Let $\bH_{DR}^{\bullet,\an}(Y)$ denote the hypercohomology of this complex,
  and $H_{DR}^{\bullet,\an}(Y)$ denote the cohomology of the complex of global
sections $\Gamma(\Omega_Y^{\bullet,\an})$.
\end{definition}

Grothendieck's theorem also extends to the holomorphic setting, where
we obtain that $\bH_{DR}^{\bullet,\an}(U) \cong H_{\tpl}^\bullet(U)$
if $U$ is smooth.  

For $Z \subseteq Y$ a subvariety and $V$ an analytic neighborhood of $Z$, we
will make use of the Mayer-Vietoris sequence associated to the exact sequence of complexes,
\begin{equation}
\label{e:mv}
0 \to \tilde \Omega^\bullet_Y \into \tilde \Omega^\bullet_{Y \setminus Z} 
\oplus \tilde \Omega^{\bullet,\an}_V \onto \tilde \Omega^{\bullet,\an}_{V \setminus Z} \to 0.
\end{equation}
Similarly, we will need the corresponding sequence when $V$ is replaced by
a formal neighborhood of $Z$:
\begin{equation}
\label{e:mv-formal}
0 \to \tilde \Omega^\bullet_Y \into \tilde \Omega^\bullet_{Y \setminus Z} 
\oplus \tilde \Omega^{\bullet}_{\hat Y_Z} \onto \tilde \Omega^{\bullet}_{\hat Y_Z
\setminus Z} \to 0.
\end{equation}
Note that there is a natural map by restriction from the sequence
\eqref{e:mv} to \eqref{e:mv-formal}. This forms the commutative
diagram with exact rows,
\begin{equation}\label{e:mvseqs}
  \xymatrix{
    \cdots \ar[r]  & \bH_{DR}^{i-1}(Y) \ar[r] \ar@{=}[d] &
    \bH_{DR}^{i-1}(Y \setminus Z)
    \oplus \bH_{DR}^{i-1,\an}(V) \ar[r] \ar[d] & \bH^{i-1,\an}_{DR}(V \setminus Z) \ar[r] \ar[d]
    & \bH_{DR}^i(Y) \ar[r] \ar@{=}[d] & \cdots \\
    \cdots \ar[r] & \bH_{DR}^{i-1}(Y) \ar[r] & \bH_{DR}^{i-1}(Y \setminus Z)
    \oplus \bH_{DR}^{i-1}(\hat Y_Z) \ar[r] & \bH^{i-1}_{DR}(\hat Y_Z \setminus Z) \ar[r]
    & \bH_{DR}^i(Y) \ar[r] & \cdots \\
  }
\end{equation}
Finally, note that, when $Y$ is affine, we can also consider the same
diagram for ordinary rather than hypercohomology, since the sequences
\eqref{e:mv} and \eqref{e:mv-formal} remain exact on the level of
global sections.
 
\subsubsection{Comparison isomorphisms for smooth varieties}
Now consider the case that $Y$
is smooth.
Then, we will need the result that a
small enough tubular neighborhood $V$ of $Z$ retracts onto $Z$.  By Grothendieck's theorem, this implies
\begin{equation}\label{e:tubnbhd}
\bH_{DR}^{\bullet,\an}(V) \cong H^\bullet_{\tpl}(Z),
\end{equation}
where as before $\bH$ denotes hypercohomology (which is necessary
since we do not require $Y$ to be affine).

Hartshorne's theorem \cite{Har-adrc,Har-drcav} gives an algebraic
analogue of the above statement:
\begin{equation}\label{e:hartshorne}
\bH_{DR}^\bullet(\hat Y_Z) \cong H^\bullet_{\tpl}(Z).
\end{equation}
Moreover, the isomorphism \eqref{e:hartshorne} composed with the
restriction $\bH_{DR}^{\bullet,\an}(V) \to \bH_{DR}^\bullet(\hat Y_Z)$ is
the natural isomorphism \eqref{e:tubnbhd}.  Put together, 
we deduce that the restriction map is an isomorphism,
\begin{equation}\label{e:comp-smooth1}
\bH_{DR}^{\bullet,\an}(V) \iso \bH_{DR}(\hat Y_Z).
\end{equation}
Therefore,  the five-lemma implies that the
vertical arrows in \eqref{e:mvseqs} are all isomorphisms. In
particular, this yields also
\begin{equation}\label{e:comp-smooth2}
\bH_{DR}^{\bullet,\an}(V \setminus Z) \iso \bH_{DR}^{\bullet}(\hat Y_Z \setminus Z).
\end{equation}
Note that, when $Y$ is affine, we can also replace
hypercohomology with ordinary cohomology (in the second isomorphism),
by using \eqref{e:mvseqs} for ordinary cohomology.

\subsubsection{Comparison theorem for isolated singularities}
\begin{theorem}\label{t:form-an-comp}
  Suppose that $X$ is a complex algebraic variety with an isolated
  singularity at $x \in X$. Then, for sufficiently small contractible
  neighborhoods $U$ of $x$, there are canonical isomorphisms
\begin{equation}\label{e:an-to-form}
\bH^{\bullet,\an}_{DR}(U) \iso \bH^{\bullet}_{DR}(\hat X_x),
\quad \bH^{\bullet,\an}_{DR}(U \setminus \{x\}) \iso
\bH^\bullet_{DR}(\hat X_x \setminus \{x\}).
\end{equation}
If in addition $U$ is Stein, then we have canonical isomorphisms on cohomology of global sections,
\begin{equation}\label{e:an-to-form-stein}
 H^{\bullet,\an}_{DR}(U) \iso H^{\bullet}_{DR}(\hat X_x),
 \quad
H^{\bullet,\an}_{DR}(U \setminus \{x\}) \iso H^{\bullet}_{DR}(\hat X_x \setminus \{x\}).
\end{equation}
\end{theorem}
\begin{remark} The theorem also extends to the case where $X$ is an
  analytic variety with an isolated singularity at $x$, with the same
  proof as below, since Hironaka's theorem on resolution of
  singularities also applies to analytic varieties.  (This is a strict
  generalization of the theorem, since every algebraic variety is also
  analytic, and the objects above are the same.)
\end{remark}
\begin{proof}
  Let $Y \to X$ be a resolution of singularities, and let $Z \subseteq
  Y$ be the fiber over $x$.  Let $V$ be a tubular neighborhood of $Z$
  which retracts to $Z$ and $U$ its image under the resolution, which
  therefore retracts to $x$.  
  Then the resolution maps restrict to isomorphisms $Y \setminus Z
  \iso X \setminus \{x\}$, $V \setminus Z \iso U \setminus \{x\}$, and
  $\hat Y_Z \setminus Z \iso \hat X_x \setminus \{x\}$.  By
  \eqref{e:comp-smooth2}, we conclude the second isomorphism in
  \eqref{e:an-to-form}.

  Now, the above was for specific neighborhoods $U$, namely those
  obtainable from tubular neighborhoods $V$ of $Z \subseteq Y$.  For
  any smaller contractible neighborhood $U' \subseteq U$ of $x$, the
  restriction map $H_{\tpl}^\bullet(U \setminus \{x\}) \to
  H_{\tpl}^\bullet(U' \setminus \{x\})$ is an isomorphism by
  Grothendieck's theorem, and hence the second isomorphism of
  \eqref{e:an-to-form} holds for sufficiently small contractible
  neighborhoods of $x$.

  Consider now \eqref{e:mvseqs} for the pair $(X,\{x\})$, with $U$
  such that the second isomorphism of \eqref{e:an-to-form} holds.  The
  five-lemma then implies that the vertical arrows are all
  isomorphisms, which implies the first isomorphism of
  \eqref{e:an-to-form}.

  Next, the first isomorphism of \eqref{e:an-to-form-stein} follows
  immediately, since $U$ is Stein, so hypercohomology
  of $U$ and $\hat X_x$ coincides with the cohomology of global
  sections.  Finally, since $X$ is affine, we can consider
  \eqref{e:mvseqs} for the pair $(X,\{x\})$ using ordinary rather than
  hypercohomology. The five-lemma now implies that the vertical arrows
  are once again isomorphisms, yielding the second isomorphism of
  \eqref{e:an-to-form-stein}.
\end{proof}

\section{Complete intersections with isolated
  singularities}\label{s:cy-cplte-int-is}
In this section, we explicitly compute $(\cO_X)_\mfv$, $M(X,\mfv)$,
and $\pi_* M(X,\mfv)$, in the case that $X \subseteq Y$ is a locally
complete intersection of positive dimension, $Y$ is affine Calabi-Yau,
and $X$ has only isolated singularities, equipping $X$ with a top
polyvector field as in Example \ref{ex:cy-cplte-int}.  For $M(X,\mfv)$
itself, the assumption that $Y$ (and hence $X$) is affine is not
necessary, using \S \ref{s:ex-dloc}.

We set $\mfv = H(X)$ (one could equivalently use $LH(X)$, in view of
Proposition \ref{p:vtop-hlh}.)  Note that, in the case $X$ is
two-dimensional, then $X$ is a Poisson variety and $H(X)$ is the Lie
algebra of Hamiltonian vector fields.  

\subsection{Complete intersections: Greuel's formulas}
Here we recall from \cite{Gre-GMZ} an explicit
formula for the de Rham cohomology of an analytic neighborhood of $x$. This
is also closely related to the results of \cite{HLY-plhdRc, Yau-vniis}.

Embed $\hat X_x \subseteq \bA^n$ cut out by equations $f_1, \ldots,
f_k$ such that $(f_1, \ldots, f_i)$ has only isolated singularities
for all $i$. Then define the ideals
\begin{equation}\label{e:jxi-defn}
  J_{X,x,i} = (f_1, \ldots, f_{i-1},  \frac{\partial(f_1, \ldots, f_i)}
  {\partial(x_{j_1}, \ldots, x_{j_i})}, 1 \leq j_1 \leq \cdots \leq j_i \leq n)
\subseteq \cO_{\hat \bA^n,x}.
\end{equation}
Here $\frac{\partial(f_1, \ldots, f_i)}{\partial(x_{j_1}, \ldots,
  x_{j_i})}$ is the determinant of the matrix of partial derivatives
$\partial_{x_{j_p}}(f_q), 1 \leq p,q \leq i$.  
Then, the Milnor number, $\mu_x$, of the singularity of $X$ at $x$ is given by
\begin{equation}
  \mu_x = \sum_{i=1}^k (-1)^{k-i}\codim_{\hat \cO_{\bA^n,x}} J_{X,x,i}.
\end{equation}
\begin{definition} \label{d:tj-no} Let $X$ and $x$ be as above.
  Define the singularity ring, $\mathcal{C}_{X,x}$, of $X$ at $x$ to
  be
\[
\mathcal{C}_{X,x} := \hat \cO_{\bA^n,x} / (J_{X,x,k},f_k),
\]
and define the Tjurina number, $\tau_x$, to be the dimension of
$\mathcal{C}_{X,x}$.
\end{definition}
Note that the ring $\mathcal{C}_{X,x}$ does not depend on the
embedding $\hat X_x \subseteq \hat \bA^n_x$ and is also definable
intrinsically as the quotient of $\hat \cO_{X,x}$ by the $m$-th
Fitting ideal of $\Omega^1_{\hat X_x}$; cf.~\cite{Har-dcr} and Remark
\ref{r:sch-sing}.
\begin{theorem} \cite[Proposition 5.7.(iii)]{Gre-GMZ} \label{t:greuel}
  If $x$ is an isolated singularity which is locally a complete
  intersection in the analytic topology, then
\begin{equation}\label{e:gre-fla}
  H^\bullet(\tilde \Omega^{\bullet,\an}_{X,x}) 
  \cong \bk^{\mu_x - \tau_x}[-\dim X].
\end{equation}
\end{theorem}
Here, $V[-\dim X]$ is the graded vector space concentrated in degree
$\dim X$ with underlying vector space $V$.

\subsection{General structure}
Since $H(X)$ has finitely many leaves, $M(X,H(X))$ is holonomic.  Let
$i: Z \into X$ be the (finite) singular locus of $X$.

Note that $i_* H^0 i^* M(X,\mfv)$ is the maximal quotient of
$M(X,\mfv)$ supported on $Z$.  Let $N$ be its kernel. Let $X^\circ :=
X \setminus Z$ and let $\IC(X) = j_{!*} \Omega_{X^\circ}$ be the
intersection cohomology $\caD$-module of $X$, i.e., the intermediate
extension of $\Omega_{X^\circ}$.  Since $j^! M(X,\mfv) \cong
\Omega_{X^\circ}$, this is a composition factor of $M(X,\mfv)$, and
all other composition factors are delta function $\caD$-modules of
points in $Z$.  Since $N$ has no quotient supported on $Z$, it must be
an indecomposable extension of the form
\begin{equation}\label{e:ciis-ext}
0 \to K \into N \onto \IC(X) \to 0,
\end{equation}
where $K$ is supported at $Z$.  Then, the structure of $M(X,\mfv)$ reduces
to computing $i_*H^0 i^* M(X,\mfv)$, the extension \eqref{e:ciis-ext}, and
how these two are extended.  The first question has a nice general answer:
\begin{theorem} \label{t:is-str} For every $z \in Z$, with $i_z: \{z\} \to X$
the embedding, there is a canonical exact sequence
\begin{equation}\label{e:is-str}
0 \to H_{DR}^{\dim X}(\hat X_z) \into
H^0 i_z^* M(X,\mfv) \onto \mathcal{C}_{X,z} \to 0.
\end{equation}
\end{theorem}
By Theorem \ref{t:form-an-comp}, there is a canonical isomorphism
$H^{\dim X}(\Omega_{X,z}^{\bullet,\an}) \iso H_{DR}^{\dim X}(\hat
X_z)$. By Theorem \ref{t:greuel}, the former has dimension
$\mu_z-\tau_z$. On the other hand, $\dim \mathcal{C}_{X,z} =
\tau_z$. We conclude 
\begin{corollary}
$i_* H^0 i^* M(X,\mfv) \cong \bigoplus_{z \in Z} \delta_z^{\mu_z}$.
\end{corollary}

The following basic result will be useful in the theorem and later on.
For an arbitrary scheme $X$ and point $x \in X$, let $(\hat
\cO_{X,x})^*$ be the continuous dual of $\hat \cO_{X,x}$ with respect
to the adic topology.
\begin{lemma}\label{l:maxquot-coinv}
Let $(X,\mfv)$ and $x \in X$ be arbitrary. Then $\Hom(M(X,\mfv), \delta_x)
  \cong ((\hat \cO_{X,x})^*)^\mfv$.
\end{lemma}
\begin{proof}
  Note that $\Hom(\caD_X, \delta_x) \cong (\hat \cO_{X,x})^*$, since
  the latter are exactly the delta function distributions at $x$. By
  definition of $M(X,\mfv)$, each $\phi \in \Hom(M(X,\mfv), \delta_x)$
  is uniquely determined by $\phi(1)$, which can be any element of
  $\delta_x$ which is invariant under $\mfv$.
\end{proof}
The theorem can therefore be restated as
\begin{theorem} \label{t:is-coinv} For all $z \in Z$, there is a canonical
exact sequence
\begin{equation}\label{e:is-coinv}
0 \to  H_{DR}^{\dim X}(\hat X_z) \to (\hat
  \cO_{X,z})_\mfv \to \mathcal{C}_{X,z} \to 0.
\end{equation}
In particular, $\dim (\hat \cO_{X,z})_{\mfv} = \mu_z$.
\end{theorem}
In the case that $Y=\bA^3$ and $X$ is a quasihomogeneous hypersurface,
the consequence that $\dim (\cO_{X,z})_\mfv = \mu_z=\tau_z$ was
discovered in \cite{AL} without using the earlier results of
\cite{Gre-GMZ}.
\begin{proof}
  Let $n := \dim Y$, $m := \dim X$, and $k := n-m$.  Let $I_X := (f_1,
  \ldots, f_k)$ be the ideal defining $X$.  Consider the map
\[
\Phi: \tilde \Omega^\bullet_X \to \Omega^{\bullet+k}_{Y} / I_X \cdot \Omega^{\bullet+k}_{Y},
\quad \alpha \mapsto \alpha\wedge df_1 \wedge \cdots \wedge df_k,
\]
which induces also a map taking the completion at $z$, which we also
denote by $\Phi$.   Note that, in this formula, we have to lift $\alpha$ to
a form on $Y$, but the map is independent of the choice of lift. 
Furthermore, $\Phi$ is injective, since $X\setminus Z$ is locally transversely
cut out by $f_1, \ldots, f_k$.
Let
$\widetilde{H(X)} \subseteq H(Y)$ be the Lie algebra of vector fields
obtained from the $(n-2)$-forms $\Phi(\tilde \Omega^{m-2}_X)$.  Then we have
an identification
\begin{equation}\label{e:is1}
(\hat \cO_{X,z})_{\mfv} \iso
 \Omega^n_{\hat Y_z} / (\widetilde{H(X)}(\hat \cO_{Y,z}) + I_X) \cdot \vol_{\hat Y_z},
\end{equation}
obtained by multiplying by $\vol_{\hat Y_z}$.  In turn,
  $\widetilde{H(X)}(\hat \cO_{Y,z}) \cdot \vol_{\hat Y_z}$ identifies with
$d \Phi(\Omega_{\hat Y_z}^{m-1})$. Therefore, 
\begin{equation}\label{e:is2}
(\hat \cO_{X,z})_{\mfv} \iso
\Omega^n_{\hat Y_z} / (d \Phi(\Omega^{m-1}_{\hat X_z}) + I_X \Omega^n_{\hat Y_z}).
\end{equation}
We now compute the RHS. Recall that $\Phi$ is an injection of complexes.
The image of $H^m(\tilde \Omega^\bullet_{\hat X_z})$ is a subspace of \eqref{e:is2}.
Moreover, the quotient of $\Omega^n_{\hat Y_z}$ by this image is
\begin{equation}\label{e:is4}
  \mathcal{C}_{X,z} = \Omega^n_{\hat Y_z} / (I_X \Omega^n_{\hat Y_z} + \Phi(\Omega^m_{\hat X_z})).
\end{equation}
We obtain the desired canonical exact sequence \eqref{e:is-coinv}.
\end{proof}
We can be more specific about the meaning of $K$ in \eqref{e:ciis-ext}
and use this to describe the derived pushforward $\pi_* M(X,\mfv)$,
where $\pi: X \to \pt$ is the projection to a point.  Let $\pi_i := H^i \pi_*$. If we apply
$\pi_*$ to \eqref{e:ciis-ext}, we obtain isomorphisms $\pi_i N  \cong
\pi_i \IC(X)$ for $i > 1$, and an exact sequence
\begin{equation}
0 \to \pi_1 N \into \IH^{\dim X - 1}(X) \to \pi_0 K \to \pi_0 N \onto
\IH^{\dim X}(X) \to 0.
\end{equation}
Here $\IH^*(X)$ denotes the intersection cohomology of $X$, 
$\IH^*(X) := \pi_{\dim X - *} \IC(X)$.  

Similarly, from the exact sequence $0 \to N \into M(X,\mfv) \onto i_*
H^0 i^* M(X,\mfv) \to 0$, we obtain isomorphisms $\pi_i(N) \cong \pi_i
M(X,\mfv)$, $i \geq 1$, and a split exact sequence
\[
0 \to \pi_0 N \into (\cO_X)_{\mfv} \onto H^0 i^* M(X,\mfv) \to 0.
\]
Put together, we obtain
\begin{corollary}
For $i \geq 2$, $\pi_i M(X,\mfv) \cong \IH^{\dim X - i}(X)$.  For some
decomposition $K = K' \oplus K''$, one has a split exact sequence
\begin{equation}
0 \to \pi_1 M(X,\mfv) \into \IH^{\dim X - 1}(X) \onto \pi_0 K' \to 0,
\end{equation}
and an isomorphism
\begin{equation} 
  (\cO_X)_{\mfv} \cong \IH^{\dim X}(X) \oplus \bigoplus_{z \in Z} (\hat \cO_{X,z})_\mfv \oplus \pi_0 K''.
\end{equation}
\end{corollary}
\begin{remark}
We plan to show in \cite{ES-ciiss} that 
$N = H^0j_! \Omega_{X\setminus\{0\}}$, so one obtains an exact sequence
\[
0 \to K \into N \onto IC(X) \to 0.
\]
Moreover, we plan to show that $K = K' = \bigl(\Ext(\IC(X), \delta_0)^* \otimes \delta_0\bigr)$.
Finally, will then conclude that  $\pi_\bullet M(X,\mfv) \cong H^{\dim X - \bullet}_{\tpl}(X) \oplus
\bk^{\mu_z}$.
\end{remark}

\subsection{The quasihomogeneous case}
Now suppose that $X \subseteq \bA^n$ where $\bA^n = \Spec
\bk[x_1,\ldots,x_n]$, each of the $x_i$ is assigned a weight $m_i \geq
1$, and $X$ is cut out by $k:=n-\dim X$ weighted-homogeneous
polynomials in the $x_i$.  In this case, $\HP_0(\cO_X)$ is a
nonnegatively graded vector space by weight.  Moreover, $M(X,H(X))$ is
a weakly $\bG_m$-equivariant $\caD$-module which decomposes into
weight submodules.  Hence, $H^0 i^* M(X,H(X))$ is weight-graded. Then, the
proofs of the preceding results generalize to this context
(considering also \cite{Gre-GMZ} and references therein).  Moreover,
by \cite{Fer-chdcas} (cf.~\cite[Korollar 5.8]{Gre-GMZ}), in this case
$H_{DR}^\bullet(X)=0$ and \eqref{e:gre-fla} implies that $\mu_z=\tau_z$,
which is the dimension of the singularity ring (see Definition \ref{d:tj-no}).
By using the weight-graded versions of the arguments of \cite{Gre-GMZ}
one deduces, for $X_{\text{sing}}$ the scheme-theoretic singular locus
of $X$, defined by the ideal $(J_{X,0,k},f_k)$,
\begin{theorem} \label{t:is-qh}
The graded vector space $H^0 i^* M(X,H(X))$
  has Poincar\'e polynomial
\begin{equation} \label{e:qh-iim} P(H^0 i^* M(X,H(X));t) =
  P(\cO_{X_{\text{sing}}};t) = P(\cO_{\bA^n}/(J_{X,0,k},f_k);t)=
\sum_{i=1}^k (-1)^{k-i} P(\cO_{\bA^n}/J_{X,0,i};t).
\end{equation} 
\end{theorem}
Since $\cO_X$ is nonnegatively graded and $X$ is connected, $H(X)$ is
spanned by homogeneous vector fields, and $(\cO_X)_{H(X)}$ is
finite-dimensional, we conclude that $(\hat \cO_X)_{H(X)} \cong
(\cO_X)_{H(X)}$.  Therefore, Lemma \ref{l:maxquot-coinv} implies
\begin{corollary}\label{c:is-qh}
$P((\cO_X)_{H(X)};t) =  P(\cO_{\bA^n}/(J_{X,0,k},f_k);t)$.
\end{corollary}
In particular, in this case, $\IH^{\dim X}(X) = 0$ and $K''=0$ (i.e., $K=K'$).
\begin{remark}\label{r:is-qh}
  In the case that $k=1$, i.e., $X$ is a quasihomogeneous hypersurface
  $Z(f)$, the ideal of the singular locus of $X$ is also known as the
  Jacobi ideal $J_X = (\partial_i f) = (\partial_i f, f)$.  For the
  last equality,  let $m_i$
  be the weight of $x_i$ for all $i$ as above, and set $m := \sum_i m_i$. Then
   $f = \frac{1}{m} \sum_i m_i x_i \partial_i f$.

In this case, one can prove the theorem in an elementary way.  Namely,
we need to show that
\[
H(X)(\cO_X) = J_X / (f).
\]
Equivalently, we have to show that
\begin{equation}\label{e:qhs-cond}
\Omega^{n-1}_X \wedge df + I_X \cdot \Omega^n_{\bA^n}
= d\Omega^{n-2}_X \wedge df + I_X \cdot \Omega^n_{\bA^n}.
\end{equation}
For this, let $\Eu := \sum_i m_i x_i \partial_i$ be the Euler vector
field on $\bA^n$. Set $\Eu^\vee := i_{\Eu}(\vol_{\bA^n}) \in
\Omega^{n-1}_{\bA^n}$. Then, for all $g \in \cO_{\bA^n}$, we have the
identities
\[
\Eu^\vee \wedge dg = \Eu(g) \cdot \vol_{\bA^n}, \quad d(g\Eu^\vee) = (\Eu(g)+ m \cdot g) \cdot \vol_{\bA^n}.
\]
Therefore, we conclude that, for all quasihomogeneous $\alpha \in
\Omega^{n-1}_X$, letting $|\cdot|$ denote the weighted degree
function,
\[
\bar \alpha := \alpha -(|\alpha|+m)^{-1} (d\alpha / \vol_{\bA^n}) \Eu^\vee \in d \Omega^{n-2}_{\bA^n}.
\]
Moreover,
\[
\alpha \wedge df \equiv \bar \alpha \wedge df
\pmod{I_X \cdot \Omega_{\bA^n}^n}.
\]
We conclude \eqref{e:qhs-cond}, and hence the theorem in this case.
\end{remark}

\section{Finite quotients of Calabi-Yau varieties}\label{s:fqcyvar}
Let $X$ be an affine connected Calabi-Yau variety and $\Xi$ the top
polyvector field inverse to the volume form; for instance, we could
have $X=\bA^n$ with the inverse to the standard volume form.  In this
case, $H(X)=LH(X)=P(X)$. Let $G$ be a finite group acting by
automorphisms on $X$ preserving $\Xi$.
In this section we
will compute the $\caD$-module $M(X,H(X)^G)$.  Everything generalizes
without change to the case where $X$ is not affine, using \S
\ref{s:ex-dloc}.

Note that, if $X$ is one-dimensional, then $G$ must be trivial.
Therefore, we will assume until the end of the section that $\dim X
\geq 2$.

As noticed at the end of \S \ref{ss:vtop},
using the induced top polyvector field on
$X/G$, $H(X)^G=P(X/G)$.  So we also deduce
$M(X/G,P(X/G))=q_*M(X,H(X)^G)^G$ where $q: X \to X/G$ is the
projection, and hence also its underived pushforward to a point,
$(\cO_{X/G})_{P(X/G)}$.  We note that, by Proposition
\ref{p:hg-vtop-lf}, since $\dim X \geq 2$, $H(X)^G$ has finitely many
leaves and hence is holonomic, so $P(X/G)$ is as well; however, in
general, $H(X/G)$ and $LH(X/G)$ are not holonomic (by Corollary
\ref{c:inc-degloc}, they are holonomic if and only if $X/G$ has only
finitely many singular points, i.e., only finitely many points of $X$
have nontrivial stabilizers in $G$).

Recall from \S \ref{ss:vtop} that we call a subgroup $K < G$
\emph{parabolic} if there exists a point $x \in X$ such that
$\Stab_G(x)=K$. Let $\Par(G)$ be the set of parabolic subgroups of
$G$.  For $K \in \Par(G)$, the connected components of $X^K$ are
called \emph{parabolic subvarieties} of $X$. By Proposition
\ref{p:hg-vtop-lf}, these are exactly the closures of the leaves of
$\mfv$, which are the connected components of $(X^K)^\circ = \{x \in X
\mid \Stab_G(x)=K\}$.

\begin{definition}
  Let $\Parpt(X,G)$ be the collection of points which are parabolic
  subvarieties; call them \emph{parabolic points}.
\end{definition}
Equivalently, the parabolic points $x \in X$ are those such that, for
some open neighborhood $U$ containing $x$, $\Stab_G(x)$ is strictly
larger than the stabilizer of any point in $U \setminus \{x\}$.
\begin{theorem} There is a canonical isomorphism
\[
M(X,\mfv) \cong \Omega_X \oplus \bigoplus_{x \in \Parpt(X,G)}
  \delta_x \otimes (\hat \cO_{X,x})_{\mfv},
\]
and each $(\hat \cO_{X,x})_{\mfv}$ is finite-dimensional.
\end{theorem}
\begin{proof}
  Let $X^\circ \subseteq X$ be the inclusion of the open locus where
  $G$ acts freely.  Then, $M(X, H(X)^G)|_{X^\circ} \cong
  \Omega_{X^\circ}$. Since $G$ preserves volume, it follows that $X
  \setminus X^\circ$ has codimension at least two, since this is true
  in the local setting $X = \bA^n$ and $G < \SL_n$.

  We claim that there are no nontrivial extensions between local
  systems supported on $X \setminus X^\circ$ and $\Omega_X$.  Indeed,
  this reduces in a formal neighborhood of an arbitrary point to the
  statement that $\Ext^1(\cO_{\bA^m}, \delta) = 0 = \Ext^1(\delta,
  \cO_{\bA^m})$ when $\delta$ is the delta-function $\caD$-module of a
  proper subspace of $\bA^m$ of codimension $k \geq 2$. Then, by the
  K\"unneth theorem, $\Ext^i(\cO_{\bA^m}, \delta)$ and $\Ext^i(\delta,
  \cO_{\bA^m})$ vanish for $i \neq k$, since the statement is true
  when $m=1$.  Therefore, the extension $H^0 j_! \Omega_{X^\circ}$ of
  $\Omega_X$ must be trivial, i.e., the canonical map $H^0 j_!
  \Omega_{X^\circ} \to \Omega_X$ is an isomorphism.

By adjunction, we have a map $H^0 j_!
\Omega_{X^\circ} =\Omega_X \to M(X, H(V)^G)$, and the cokernel of this
map is supported on a union of proper parabolic subvarieties of
$V$. Suppose that $U \subseteq V^K$ is a maximal such subvariety for
$K \in \Par(G)$.  We claim that $U$ is zero-dimensional, i.e., a
finite union of points.  By formally localizing in the neighborhood of
a generic point of $U$, it suffices to assume that $K=G$.  This
reduces the claim to the statement:
\begin{lemma}\label{l:nopropqt}
  Suppose that $U$ and $W$ are positive-dimensional vector spaces and
  $G < \GL(W)$ is finite.  Then $M(U \times W, H(U \times W)^G)$ admits
  no quotients supported on proper subvarieties of $U \times W$.
\end{lemma}
\begin{proof}
  Let $X := U \times W$ and $\mfv := H(U \times W)^G$.  Suppose there
  were a quotient of $M(X, \mfv)$ supported on $U \times W^K$ for some
  parabolic subgroup $K < G$.  By formally localizing in a
  neighborhood of a generic point of $U \times W^K$, we can reduce to
  the case that $K=G$; let us assume this.  So we have to show that
  there is no quotient supported at $U \times \{0\}$.

  Since $\mfv$ includes constant vector fields in the $U$ direction,
  the defining quotient $\caD_{X} \onto M(X, \mfv)$ factors through
  $\caD_{X} \onto \Omega_U \boxtimes \caD_W$.  Moreover, given a
  vector field $\xi \in \mfv$, write $\xi=\xi_1 + \xi_2$ where $\xi_1
  \in \Vect(U) \otimes \cO_W$ and $\xi_2 \in \cO_U \otimes
  \Vect(W)^G$. Let $D: \Vect(X) \to \cO_X$ be the standard divergence
  function, i.e., $D(\xi)=L_{\xi} \omega / \omega$, where $\omega$ is
  the standard volume form on $X$.  Then, since $\mfv$ includes
  constant vector fields in the $U$ direction, $\xi_1 + D(\xi_1) \in
  \mfv \cdot \caD_X$.  Thus, $\xi_2 - D(\xi_1) = \xi_2 + D(\xi_2) \in
  \mfv \cdot \caD_X$ as well. Conversely, the constant vector fields
  in the $U$ direction together with elements $\xi_2 + D(\xi_2)$ span
  $\mfv \cdot \caD_X$.  We conclude that $M(X,\mfv) = \mfv
  \cdot \caD_X \setminus \caD_X$ is of the form
\[
M(X,\mfv) = \Omega_U  \boxtimes N, \quad N = \langle \xi +
D(\xi) \mid \xi \in \Vect(W)^G \rangle \cdot \caD_W \setminus \caD_W.
\]
Therefore, the lemma reduces to showing that $N$ admits no quotient
supported at $0 \in W$.  First of all, let $\Eu_W \in \Vect(W)$ be the
Euler vector field on $U$.  Then $\Eu_W + D(\Eu_W) = (\Eu_W + \dim(W))
\in \mfv \cdot \caD_X$.  On the other hand, since $\dim(W) > 0$,
$\Eu_W + \dim(W)$ acts by an automorphism on every quotient supported
at zero (note that sections of the delta function $\caD$-module are in
nonpositive polynomial degree, and homogeneous sections in degree $m
\leq 0$ are annihilated by $\Eu + m$ (since we are using right
$\caD$-modules)).  Thus, $N$ admits no such quotient.
\end{proof}

We conclude that the cokernel of the inclusion $\Omega_X \into M(X,
\mfv)$ is supported at finitely many points, i.e., it is a direct sum
of delta-function $\caD$-modules at these points.  Since we assumed
that $\dim X \geq 2$, $\Ext(\Omega_X, \delta)=0$ when $\delta$ is such
a delta-function $\caD$-module (this follows because it is true in the
case $X=\bA^n$ and the point is the origin).  Therefore, $M(X,\mfv)$
is semisimple.

  It remains only to compute the
  multiplicity of $\delta_x$. Note that this must be
  finite-dimensional since $M(X,\mfv)$ is holonomic. The result thus
  follows from Lemma \ref{l:maxquot-coinv}.
\end{proof}
\begin{remark}
  The results of this section can be generalized to the case where $G$
  acts on $X$ preserving $\Xi$ only up to scaling.  Then, $X/G$ is no
  longer equipped with a top polyvector field, but it is equipped with
  a divergence function from $X$. Indeed, since $G$ preserves the flat
  connection which annihilates the volume form on $X$, and this equips
  $X/G$ with a flat connection on its (possibly nontrivial) canonical
  bundle.  So in this case one still has $q_* M(X,H(X)^G)^G \cong
  M(X/G, P(X/G))$, where $q: X \to X/G$ is the quotient map, and
  $P(X/G)$ is interpreted as in \S \ref{s:div}.

  In this case, the above results go through without change if $X$ has
  no parabolic subvarieties of codimension one.  However, when there
  are parabolic subvarieties of codimension one, then $M(X, H(X)^G)$
  is no longer semisimple: although Lemma \ref{l:nopropqt} still
  implies that it has no quotients supported on proper subvarieties of
  $X$, we can have nontrivial extensions on the bottom by submodules
  supported on proper subvarieties.  Let $j^\circ: X^\circ \to X$ be
  the inclusion of the locus where $G$ acts freely, and $j': X' \to X$
  the inclusion of the possibly larger locus which is the complement
  of codimension-one parabolic subspaces (an affine subvariety).
  Then, $H^0 j^\circ_!  \Omega_{X^\circ} = j'_! \Omega_{X'}$, which is
  not equal to $\Omega_X$ when $X' \neq X$.  We then obtain by the
  argument of the proof an exact sequence
\[
j'_! \Omega_{X'} \to M(X, H(X)^G) \to \bigoplus_{x \in \Parpt(X,G)}
  \delta_x \otimes (\hat \cO_{X,x})_{\mfv} \to 0.
\]
Moreover, the computation of Lemma \ref{l:nopropqt} shows that the
first map in the sequence above is injective in codimension one, i.e.,
restricted to the formal neighborhood of a generic point of any
component of $X \setminus X'$, it is injective.  So, for $X \neq X'$, $M(X, H(X)^G)$ is not semisimple.
\end{remark}

\section{Symmetric powers of varieties}\label{s:sym}
Given $(X, \mfv)$, note that $\mfv$ also acts naturally on the
symmetric powers $S^n X := X^n / S_n$. Then, the diagonal embedding of
$X$ into $S^n X$ is invariant under the flow of $\mfv$, and more
generally, arbitrary diagonal embeddings are invariant.

In this section, we compute the coinvariants $(\cO_{S^n X})_{\mfv}$ as
well as the $\caD$-module $M(S^n X, \mfv)$ for all $n \geq 1$ in the
transitive (affine) cases of \S \ref{s:Csla} (the ``global'' versions
of the simple Lie algebras of vector fields).  In the symplectic case
this specializes to the main result of \cite{hp0weyl}.  Our main
result says that, in the Calabi-Yau and symplectic cases, this is a direct sum of the pushforwards under $X^n \onto S^n X$ of the canonical $\caD$-modules $\Omega_\Delta$ as $\Delta$ ranges over
 the diagonal subvarieties $\Delta \subseteq X^n$ up to the action of $S_n$.
In other words, these are the
intersection cohomology
$\caD$-modules of the diagonal subvarieties of $S^n X$.  In the locally
conformally symplectic case, and in a more general transitive
setting that includes all of these cases, we prove the same result, except
replacing $\Omega_\Delta$ by the diagonal embedding of $M(X,\mfv)$.  Moreover,
when $X$ is a contact variety and $\mfv=H(X)$, or
$X$ is smooth and  $\mfv=\Vect(X)$, we show that $M(S^n X, \mfv) = 0$, and
extend these cases to a more general transitive setting where $\mfv$ does not flow incompressibly.

More generally, we will prove general structure theorems on $M(S^n X,
\mfv)$ in the case that $\mfv$ is transitive and satisfies a certain
condition we call \emph{quasi-locality}, which essentially says that
its restriction to the $m$-th infinitesimal neighborhood of every
finite set is equal to the sum of its restrictions to the $m$-th
infinitesimal neighborhood of each point in the set.  For convenience,
we will also generally assume that $X$ is connected; it is easy to
remove this assumption.

\subsection{Relation to Lie algebras for $S^n X$}
The study of $S^n X$ under $\mfv$ is closely related to the study of
$S^n X$ under its own associated Lie algebras of vector fields. Note that
 $\cO_{S^n X} = \Sym^n \cO_X$ is spanned  by
elements $f^{\otimes n}$ for $f \in \cO_X$.  Let $\symm: \cO_X^{\otimes n} \to \Sym^n \cO_X$ be the symmetrization map, 
\[
\symm(f_1 \otimes \cdots \otimes f_n) = \frac{1}{n!} \sum_{\sigma \in
  S_n} f_{\sigma(1)} \otimes \cdots \otimes f_{\sigma(n)}.
\] 
Note that, if $X$ is Poisson with bivector field $\pi$, then so is
$S^n X$, using the unique Poisson bracket on $\Sym^n \cO_X$ obtained
from the Leibniz rule; in other words, one can consider the bivector
field $\sum_{i=1}^n \pi^{i}$ on $X^n = \Spec \cO_X^{\otimes n}$, where
$\pi^i = \Id^{\otimes (i-1)} \otimes \pi \otimes \Id^{\otimes (n-i)}
\in (\wedge_{\cO_X} T_X)^{\otimes n}$ denotes $\pi$ acting on the
$i$-th component.  This then restricts to symmetric functions
$\cO_{S^n X} = \Sym^n \cO_X$.

If $X$ is even-dimensional and equipped with a top polyvector field
$\Xi$, then $S^n X$ is equipped with the top polyvector field
$\wedge^n \Xi$.

As discussed in Remark \ref{r:jac-prod}, when $X$ is Jacobi, there is no
natural Jacobi structure induced on $X^n$ and hence neither on $S^n
X$.

We then have the following elementary proposition (the first part was
essentially used in \cite{hp0weyl}):
\begin{proposition}\label{p:lie-snx}
\begin{enumerate}
\item[(i)] If $X$ is Poisson, 
then $M(S^n X, H(X)) \cong M(S^n X, H(S^n X))$;
\item[(ii)] For $X$ even-dimensional and equipped with a top polyvector
field, $P(X) \subseteq P(S^n X)$;
\item[(iii)] For $X$ 
 equipped with a divergence function $D$ on a coherent
subsheaf $N \subseteq T_X$, one has $P(X,D) \subseteq P(S^n X,D)$, where
$S^n$ is equipped with a divergence function on $\cO_{S^n X} \cdot N$, using
the natural embedding of vector spaces $N \subseteq T_X 
\into T_{S^n X}$ (via extending derivations from $\cO_X$ to
$\cO_{S^n X} = \Sym^n_\bk \cO_X$);
\item[(iv)] For general $X$, $\Vect(X) \subseteq \Vect(S^n X)$.  
\end{enumerate}
\end{proposition}
\begin{proof}
  (i) Given $f \in \cO_X$, it is evident that (up to normalization)
  $n\cdot \xi_{\symm(f \otimes 1^{\otimes (n-1)})}$ identifies with $\xi_f
  \in H(X)$. Hence $H(X) \subseteq H(S^n X)$ (this is also a special
  case of part (ii)). Next, $H(S^n X)$ is spanned by the vector fields
  $\xi_{f^{\otimes n}} = n \cdot \symm(\xi_f \otimes f^{\otimes (n-1)})$ for
  $f \in \cO_X$.  Note the identities $\xi_f(f)=0$ and $\xi_{f^i} = i
  f^{i-1} \xi_f$.  Thus, for all $i \geq 1$,
\begin{multline}
\symm(\xi_{f^i} \otimes 1^{\otimes (n-1)}) \cdot \symm(f^{\otimes
  (n-i-1)} \otimes 1^{\otimes (i+1)}) 
\\ = \frac{i}{n}\symm(\xi_{f^i} \otimes f^{\otimes (n-i-1)} \otimes 1^{\otimes i}) \\+
   \frac{n-i}{n}\symm(\frac{i}{i+1} \cdot \xi_{f^{i+1}} \otimes
f^{\otimes (n-i-2)} \otimes 1^{\otimes (i+1)}).
\end{multline}
The LHS is in $H(X) \cdot \caD_X$, and the RHS terms, taken over all
$i \geq 1$, generate $\symm(\xi_f \otimes f^{\otimes (n-1)})$, as desired.

(ii) It is evident that, if a vector field preserves a top polyvector field $\Xi$ on $X$, then it also preserves $\wedge^n \Xi$ on $S^n X$.

(iii) Similarly, if a vector field $\xi$ preserves a divergence
function $D$, i.e., $D(\xi) = 0$,   then also
it preserves the induced divergence function on $S^n X$, i.e., the induced
divergence function on $S^n X$ by definition also kills $\xi$, viewed
as a vector field on $S^n X$.

(iv) Similarly, given a vector field $\xi \in \Vect(X)$, we
  can take the sum $\sum_{i=1}^n \xi^i \in \Vect(X^n)$ which descends
  to $\Vect(S^n X)$.
\end{proof}
\begin{remark}
  Note that the isomorphism of (i) does not extend, in general, to the
  cases of top polyvector fields. For instance, when $X$ is
  symplectic, then by part (i), viewed as a Poisson variety, $H(S^n
  X)$ and $H(X)$ determine the same $\caD$-module, which is holonomic
  since $S^n X$ has finitely many symplectic leaves (the images of the
  diagonal embeddings).  However, since the singular locus of $S^n X$
  is infinite for $n \geq 2$, by Corollary \ref{c:vtop-leaves}, $H(S^n
  X,\wedge^n \vol_X^{-1})$ does not have finitely many leaves, and by
  Corollary \ref{c:vtop-hol} the associated $\caD$-module is not
  holonomic.
\end{remark}

\subsection{Diagonal embeddings}
Let $\Delta_i: X \to X^i$ be the standard diagonal embeddings
for all $i \geq 1$. Let $\pr_n: X^n \to S^n X$ be the
projection. Recall that a partition $\lambda$ of $n$, which we denote
by $\lambda \vdash n$, is a tuple $(\lambda_1, \ldots, \lambda_k)$
with $\lambda_1 \geq \lambda_2 \geq \cdots \geq \lambda_k \geq 1$ and
$\lambda_1 + \cdots + \lambda_k = n$.  In this case the length,
$|\lambda|$, of $\lambda$ is defined by $|\lambda|:=k$.  Given a
partition $\lambda \vdash n$, define the product of diagonal
embeddings
\[
\Delta_\lambda:= \Delta_{\lambda_1} \times \cdots \times
\Delta_{\lambda_{|\lambda|}}: X^{|\lambda|} \to X^n.
\]
Now, composing with $\pr_n$, we obtain a map $X^{|\lambda|} \to S^n
X$.  On the complement of diagonals in $X^{|\lambda|}$, this is a
covering onto its image whose covering group is the subgroup $S_\lambda <
S_{|\lambda|}$ preserving the partition $\lambda$.  Explicitly,
$S_{\lambda} = S_{r_1} \times \cdots \times S_{r_k}$ where, for all
$j$,
\[
\lambda_{r_1+\cdots+r_j} > \lambda_{r_1+\cdots+r_j+1} = \lambda_{r_1 +
  \cdots + r_j + 2} = \cdots = \lambda_{r_1 + \cdots + r_j + r_{j+1}}.
\]
\subsection{A morphism of graded algebras}
Consider the canonical morphism of graded algebras
\begin{equation}\label{e:sym-tr-invts}
\Phi:  \Sym (t \cdot ((\cO_X)^*)^{\mfv}[t]) \to \bigoplus_{n \geq 0} 
  ((\cO_{S^n X})^*)^{\mfv}, 
\end{equation}
given by the formula
\[
\Phi(t^{r_1} \phi_1 \otimes \cdots \otimes t^{r_k} \phi_k)(f_1 \otimes
\cdots \otimes f_{r_1+\cdots+r_k}) = \prod_{i=1}^k
\phi_i(f_{r_1+\cdots+r_{i-1}+1} \cdots f_{r_1+\cdots+r_i}).
\]
Let us explain the graded algebra structures in
\eqref{e:sym-tr-invts}.  First, the grading is by degree in $t$ on the
left-hand side and by degree in $n$ on the right-hand side. The
algebra structure on the left-hand side is as in a symmetric algebra.
The algebra structure on the right-hand side is obtained from the
natural inclusions
\[
\cO_{S^{n+m}X} \into \cO_{S^n X} \otimes \cO_{S^m X}.
\]
In other words, the above maps are the
symmetrization maps, 
\[
(f_1 \otimes \cdots \otimes f_{m+n}) \mapsto \frac{m!n!}{(m+n)!}
\sum_{\underset{|I|=n}{I \subseteq \{1,\ldots, m+n\}}} f_I \otimes f_{I^c},
\]
where $f_I := \bigotimes_{i \in I} f_i$, and $I^c$ is the complement of
$I$.  

This induces a coproduct on $\bigoplus_{n \geq 0} \cO_{S^n X}$ and
hence an algebra structure on $\bigoplus_{n \geq 0} \cO_{S^n
  X}^*$. The $\mfv$-invariants form a subalgebra.  

Moreover, replacing $(\cO_{S^n X})_{\mfv}$ by the derived pushforward
$\pi_\bullet M(S^n X, \mfv)$ for $\pi: S^n X \to \pt$ the projection to a
point, we obtain a bigraded algebra $\bigoplus_{n \geq 0} \pi_\bullet M(S^n
X, \mfv)^*$, in de Rham and homological degrees. Then \eqref{e:sym-tr-invts} becomes
\begin{equation}\label{e:sym-can-invts}
\Phi: \Sym (t \cdot \pi_\bullet M(X,\mfv)^*[t]) \to \bigoplus_{n \geq 0} 
\pi_\bullet M(S^n X, \mfv)^*.
\end{equation}
Here, $\bullet$ is the homological degree, and the symmetric algebra
is supersymmetric where the parity is given by the homological degree (note that this \emph{differs} from the de Rham parity in the case that $\dim X$ is odd).
By Proposition \ref{p:tr} and Example \ref{ex:cy2}, in the case that
$X$ is symplectic or Calabi-Yau, \eqref{e:sym-can-invts} can be
restated as
\begin{equation}\label{e:sym-can-invts2}
\Sym (t \cdot H^{\dim X - \bullet}(X)^*[t]) \to \bigoplus_{n \geq 0} 
\pi_\bullet M(S^n X, \mfv)^*.
\end{equation}
\subsection{Quotients of $M(S^nX)$ supported on diagonals}
For arbitrary $(X, \mfv)$, since each $\Delta_\lambda$ is a closed
embedding, one has a natural epimorphism
\[
M(X^n, \mfv) \onto (\Delta_\lambda)_* M(X^{|\lambda|}, \mfv),
\]
Next, note that $M(X^n, \mfv)$ is an $S_n$-equivariant
$\caD$-module, and one has $(\pr_n)_* M(X^n, \mfv)^{S_n} \cong
M(S^n X, \mfv)$.  The morphism above descends to a natural map
\[
M(S^n X, \mfv) \onto (\pr_n)_* (\Delta_\lambda)_* M(X^{|\lambda|},
\mfv)^{S_\lambda}.
\]
Summing over $\lambda$, we obtain a natural map
\begin{equation}\label{e:gen-sn}
  M(S^n X, \mfv) \to \bigoplus_{\lambda \vdash n} (\pr_n)_* 
  \bigl( (\Delta_\lambda)_* M(X^{|\lambda|},\mfv) \bigr)^{S_\lambda}
\end{equation} 
In the case that $X$ is symplectic or Calabi-Yau, by Proposition
\ref{p:tr} and Example \ref{ex:cy2}, \eqref{e:gen-sn} can be restated as
\begin{equation}\label{e:gen-sn2}
  M(S^n X, \mfv) \to \bigoplus_{\lambda \vdash n} (\pr_n)_* 
  \bigl( (\Delta_\lambda)_* \Omega_X^{\boxtimes |\lambda|} \bigr)^{S_\lambda}.
\end{equation} 

\subsection{Main result} 
\begin{theorem}\label{t:sym}
\begin{enumerate}
\item[(i)] If $X$ has pure dimension at least two and
is locally conformally symplectic or Calabi-Yau, 
then with $\mfv = H(X)$,
 \eqref{e:gen-sn} and \eqref{e:sym-can-invts} are isomorphisms.
\item[(ii)] If $(X,\mfv)$ is an (odd-dimensional) contact variety with
  $\mfv=H(X)$, or $(X,\mfv)$ is connected, smooth, and
  positive-dimensional with $\mfv=\Vect(X)$, then $M(S^n X, \mfv) =
  0$.
\end{enumerate}
\end{theorem}
For the case where $X$ is a Calabi-Yau curve, $\mfv$ is one-dimensional,
and $M(S^n X, \mfv)$ is not holonomic for $n > 1$.

\begin{remark} In the symplectic and Calabi-Yau cases, one can
  alternatively consider $H(S^n X)$, $LH(S^n X)$, and $P(S^n X)$,
  where now $S^n X$ is viewed as either a Poisson variety (when $X$ is
  symplectic) or as a variety equipped with a top polyvector field
  (when $X$ is even-dimensional Calabi-Yau) or more generally one can
  consider $H(S^n X,D)$ and $P(S^n X,D)$ when $X$ is odd-dimensional
  and equipped with a divergence function on $T_{S^n X} =
  T_{X^n}^{S_n}$ obtained from the Calabi-Yau divergence function on
  $X^n$.  It is easy to see that the image of the map in
  \eqref{e:sym-can-invts2} is invariant under all of these, since on
  each leaf, i.e., the complement in a diagonal $\pr_n \circ
  \Delta_\lambda(X^{|\lambda|})$ of smaller diagonals, the image of
  the corresponding functionals on the left-hand side are supported on
  this diagonal and invariant under all vector fields that preserve
  the given structure.  Moreover, in the symplectic case, if we
  instead use $H(S^n X)$ (or $LH(S^n X)$), we will obtain the same
  result in view of Proposition \ref{p:lie-snx}.(i) (as already
  noticed in \cite{hp0weyl}).  This recovers the main result of
  \cite{hp0weyl} (where this observation was also used in the proof).

  In the Calabi-Yau case, one can replace $\mfv$ on the RHS of
  \eqref{e:sym-can-invts2} by $P(S^n X)$, since here one also has $P(X)
  \subseteq P(S^n X)$, so the isomorphism factors through the same
  expression with $P(S^n X)$-invariants.

  However, in the Calabi-Yau case, one cannot replace the RHS with
  $H(S^n X)$ or $LH(S^n X)$-invariants, since $H(X)$ is not contained
  in these in general.  In fact, for $n \geq 2$, these invariants are
  infinite-dimensional: already when $X = \bA^2$ equipped with the
  standard volume form, $S^2 \bA^2 \cong (\bA^2 / (\bZ/2)) \times
  \bA^2$, so the coinvariants $(\cO_{S^2 \bA^2})_{H_{\wedge^2 \Xi}(S^2
    \bA^2)} = (\cO_{S^2 \bA^2})_{LH_{\wedge^2 \Xi}(S^2 \bA^2)}$ are
  infinite-dimensional by Remark \ref{r:inc-prod}.
\end{remark}

\begin{remark}
  Theorem \ref{t:sym} may generalize in some form to the case where $X$
  is not necessarily transitive, but has a finite degenerate locus.
  As a first step, in \cite{ESsym}, the authors prove that, when $X
  \subseteq \bA^3$ is a quasihomogeneous isolated surface singularity
  and $\mfv=H(X)$, then abstractly one still has an isomorphism
\begin{equation}\label{e:abst-iso-symqh}
\Sym (t \cdot ((\cO_X)^*)^{\mfv}[t]) \cong \bigoplus_{n \geq 0} 
  ((\cO_{S^n X})^*)^{\mfv},
\end{equation}
but only as algebras graded by symmetric power degree, not by the
weight degree in $\cO_X$.  (To correct this, one can assign $t$ weight
degree $-d$, where the hypersurface cutting out $X$ has weight $d$
(note that here $\cO_X$ has nonnegative weight and $(\cO_X)^*$ has
nonpositive weight).  Then one does obtain an isomorphism of graded
algebras.)
\begin{question}
Does the abstract algebra isomorphism \eqref{e:abst-iso-symqh}, graded
only by symmetric power degree,
extend to the case
 where $X \subseteq \bA^n$ is an arbitrary quasihomogeneous
complete intersection with an isolated singularity, equipped with its
top polyvector field from Example \ref{ex:cy-cplte-int}?  Can it
be corrected to an abstract bigraded isomorphism by assigning $t$ the
appropriate weight?
\end{question}
\begin{question}
Does the abstract algebra isomorphism \eqref{e:abst-iso-symqh} extend
to the case of arbitrary (not necessarily quasihomogeneous) complete
intersections with isolated singularities?  What about if the complete
intersection condition is dropped?
\end{question}
Finally, we remark that, even as nonequivariant $\caD$-modules, the
two sides of \eqref{e:gen-sn} are \emph{not} in general isomorphic,
because $M(S^n X, \mfv)$ is not in general semisimple.  For the case
where $X$ is a quasihomogeneous surface with an isolated singularity,
the authors hope to compute $M(S^n X, \mfv)$ in \cite{ES-ciiss}.
There, only in the case where $X$ has genus zero, i.e., the
hypersurface is a du Val singularity, does it hold that $M(X,\mfv)$ is
semisimple; in all other cases, the two sides of \eqref{e:gen-sn} are
not isomorphic as (nonequivariant) $\caD$-modules for $n \geq 2$.

In the case of the du Val singularities, the two sides of
\eqref{e:gen-sn} are only abstractly isomorphic as nonequivariant
$\caD$-modules, by \cite[\S 1.3]{hp0weyl}.  One can introduce a
correction analogous to the above one to the RHS which makes the two
sides isomorphic as $\bG_m$-equivariant $\caD$-modules, but we do not
know of any natural isomorphism between the two.
\end{remark}

\subsection{Smooth and contact varieties}\label{ss:sym-sc}
By Theorem \ref{t:sym}, in the case that $(X,\mfv)$ is either $(X,
\Vect(X))$ for smooth $X$, or $(X, H(X))$ for $X$ an odd-dimensional
contact variety, then $M(S^n X, \mfv) = 0$ for all $n \geq 0$.
However, it turns out that $M(X^n, \mfv)$ itself is \emph{nonzero}
when $n > \dim X$.  Moreover, this can be explicitly computed as an
$S_n$-equivariant $\caD$-module.  

We first construct some canonical quotients $M(X^n, \Vect(X)) \onto
(\Delta_n)_* \Omega_X$.  Let $d := \dim X$.  We can identify global
sections of $(\Delta_n)_* \Omega_X$ with $\cO_{\Delta_n(X)}$-linear
polydifferential operators $\hat \cO_{X^n, \Delta_n(X)} \to
\Omega_{\Delta_n(X)}$.   Then, we
consider the operator
\[
\phi_{n,d}: (f_1 \otimes \cdots \otimes f_n) \mapsto 
f_{d+2} \cdots f_n  \sum_{\sigma \in S_{d+1}}  
\frac{1}{(d+1)!} \sign(\sigma)
f_{\sigma(1)} df_{\sigma(2)} \wedge \cdots \wedge df_{\sigma(d+1)}.
\]
We can see that $\phi_{n,d}$ is $\cO_{\Delta_n(X)}$-linear (to ensure this, we
had to skew-symmetrize over $S_{d+1}$ rather than $S_d$).  Moreover,
$\bk[S_n] \cdot \phi_{n,d}$ is actually preserved by $\bk[S_{n+1}]$, and as a
representation of $S_{n+1}$, it is
\[
\Ind_{S_d \times S_{n-d-1}} (\sign \boxtimes \, \bk).
\]
Thus, $\bk[S_n] \cdot \phi_{n,d}$ has dimension ${n-1 \choose d}$. Let
$L_n$ be the $S_n$-equivariant local system supported on $\Delta_n(X)$
of rank ${n-1 \choose d}$ corresponding to this quotient (as a
nonequivariant local system, it is $((\Delta_n)_* \Omega_X)^{\oplus
  {n-1 \choose d}}$).  More generally, given a decomposition
$\{1,\ldots,n\} = P_1 \sqcup \cdots \sqcup P_m$ into cells, let
$L_{P_1} \boxtimes \cdots \boxtimes L_{P_m}$ denote the corresponding
tensor product of local systems $L_{|P_i|}$ in the components $P_i$
(i.e., these are all obtained by permutation of components from the
local system $L_{|P_1|} \boxtimes \cdots \boxtimes L_{|P_m|}$). This is equivariant 
with respect to the subgroup of $S_n$ preserving the decomposition, which
is isomorphic to $S_{|P_1|} \times \cdots \times S_{|P_m|}$. Note
that it is nonzero if and only if $|P_i| > d$ for all $i$.
\begin{theorem}
Suppose that $(X,\mfv)$ is either $(X,
\Vect(X))$ for smooth $X$, or $(X, H(X))$ for $X$ an odd-dimensional
contact variety.  Then, we have an isomorphism as $S_n$-equivariant local systems,
\begin{equation}
M(X^n, \mfv)
= \bigoplus_{m \geq 1, P_1 \sqcup \cdots \sqcup P_m = \{1,\ldots,n\}}
  L_{P_1} \boxtimes \cdots \boxtimes L_{P_m}.
\end{equation}
\end{theorem}
Note in the theorem that, even though the individual summands on the
RHS are not $S_n$-equivariant, the direct sum is canonically
$S_n$-equivariant.

\subsection{Quasi-locality and a generalization of Theorem \ref{t:sym}} \label{ss:ql}
\begin{definition}\label{d:ql}
  Say that $(X, \mfv)$ is \emph{quasi-local} if, for every $n$-tuple
  of distinct points $x_1, \ldots, x_n \in X$, and every choice of
  positive integers $m_1, \ldots, m_{n-1} \geq 1$, the subspace of
  $\mfv$ of vector fields vanishing to orders $m_i$ at $x_i$ for all
  $1 \leq i \leq n-1$ topologically span $\mfv|_{\hat X_{x_n}}$.
\end{definition}
Equivalently, as stated in the beginning of the section, the
evaluation of $\mfv$ at every subscheme supported at a finite subset
$S \subseteq X$ is the direct sum of its evaluations at each
connected component of $S$ (i.e., at each subscheme of $S$ supported
on a point of $S_\red$).
\begin{proposition}
  If $(X, \mfv)$ is quasi-local, then the leaves of $(S^n X, \mfv)$
  are the images of the products of leaves of $X$ under $\pr_n$.  In
  particular, if $(X,\mfv)$ has finitely many leaves, so does $(S^n X,
  \mfv)$, and the latter is holonomic.
\end{proposition}
\begin{proof}
At each point $\pr_n \circ \Delta_\lambda(x_1, \ldots, x_{|\lambda|})$,  
the pushforward
$(\pr_n)_*$ induces an isomorphism of vector spaces,
\begin{equation}\label{e:ql-mtr}
(T_{\Delta_\lambda(x_1, \ldots, x_{|\lambda|})} X^n)^{S_\lambda}
\iso  \mfv|_{\pr_n \circ \Delta_\lambda(x_1, \ldots, x_{|\lambda|})}.
\end{equation}
Therefore, along each diagonal, the flow of $\mfv$ is transitive along the images of the products of leaves of $X$.
\end{proof}
\begin{proposition}
If $X$ is Jacobi or equipped with a top polyvector field, 
then $(X,H(X))$ is quasi-local. Similarly, $(X, \Vect(X))$ is quasi-local.
\end{proposition}
\begin{proof} We first consider the Jacobi case.  Given points $x_1,
  \ldots, x_n \in X$, and any orders $m_1, \ldots, m_{n-1} \geq 1$, we
  can consider functions which vanish up to order $m_i$ at $x_i$ for
  $1 \leq i \leq n-1$.  Since the $x_i$ are distinct, these functions
  topologically span $\hat \cO_{X,x_n}$.  Therefore, the Hamiltonian
  vector fields of such functions topologically span all Hamiltonian
  vector fields in the formal neighborhood $\hat X_{x_n}$.

  Next consider the Calabi-Yau case.  This is similar: we replace
  functions which vanish up to order $m_i$ at $x_i$ for $1 \leq i \leq
  n-1$ by $(\dim X - 2)$-forms with this vanishing property.  Again,
  these topologically span $\Omega_{\hat X_{x_n}}$, and we conclude the
  result.

  For the case of all vector fields, this is immediate.
\end{proof}
\begin{theorem}\label{t:sym-ql} 
Suppose that $(X,\mfv)$ is 
  transitive and quasi-local
  and that $X$ has pure dimension at least $2$.
\begin{itemize} 
\item[(i)] 
  If $\mfv$ flows incompressibly, then
  \eqref{e:gen-sn} is an isomorphism if and only if:
\begin{itemize}
\item[(*)] For all $n$, and any (or every) $x \in X$, the space of
  $\mfv$-invariant polydifferential operators $\Sym^n \hat \cO_{X,x}
  \to \hat \cO_{X,x}$ is spanned by the multiplication
  operator.
\end{itemize}
\item[(ii)] If $\mfv$ does not flow incompressibly, then $M(S^n X, \mfv) = 0$ for all $n \geq 1$ if and only if,
for 
all $n \geq 1$, there are no $\mfv$-invariant polydifferential
operators $\Sym^{n} \hat \cO_{X,x} \to \Omega_{\hat X_x}$.
\end{itemize}
\end{theorem}
We will prove this theorem as a consequence of more general results in
the not-necessarily quasi-local case below.  First we explain why this
theorem implies Theorem \ref{t:sym}:
\begin{proposition}
\begin{itemize}
\item[(i)]  Let $(X,H(X))$ be locally conformally symplectic 
or Calabi-Yau of pure dimension
  at least two. Then (*) of Theorem \ref{t:sym-ql} is satisfied.

\item[(ii)]  In the case where $(X,\mfv)$ is either an odd-dimensional
  contact variety with $\mfv = H(X)$, or smooth with $\mfv=\Vect(X)$,
  then for all $x \in X$, all $\mfv$-invariant polydifferential
  operators $\hat \cO_{X,x}^{\otimes n} \to \Omega_{\hat X,x}$ are
  spanned over $\bk[S_n]$ by the operator
\[
(f_1 \otimes \cdots \otimes f_n) \mapsto f_1 \cdots f_{n-\dim X}
df_{n-\dim X+1} \wedge \cdots \wedge df_n.
\]
In particular there are no symmetric such operators.
\end{itemize}
\end{proposition}
\begin{proof}
(i)  This relies on the Darboux theorem, following
  \cite[Lemma 2.1.8]{hp0weyl}.  In a formal neighborhood $\hat X_x$,
  we can reduce to the case of the standard symplectic or Calabi-Yau
  structure, since in the locally conformally symplectic case, $H(\hat
  X_x)$ equals $H(\hat X_x, \omega_0)$, where $\omega_0$ is a standard
  symplectic structure, as explained in Example \ref{ex:st-lcs}.

  Now, given a polydifferential operator $\phi: \Sym^n \hat \cO_{X,x}
  \to \hat \cO_{X,x}$, view it as a polynomial function $\bar \phi:
  \hat \cO_{X,x} \to \hat \cO_{X,x}$ on the pro-vector space $\hat
  \cO_{X,x}$. Then $\bar \phi$ is uniquely determined by its
  restriction to functions with nonvanishing first derivative, since
  the complement has codimension at least two.  Let $f \in \hat
  \cO_{X,x}$ be such a function.  Let $G_{X,x}$ be the formal group
  obtained by integrating $H(X)$, which acts on $\hat \cO_{X,x}$. By
  the Darboux theorem, there is a coordinate change by $G_{X,x}$ that
  takes $f$ to a coordinate function $x_1$ of $X$.  Now, if a
  polydifferential operator is invariant under $H(X)$, it must take
  $x_1$ to a function invariant under the formal subgroup of $G_{X,x}$
  preserving $x_1$, i.e., to a polynomial in $x_1$.  Now, to be
  invariant under automorphisms in $G_{X,x}$ sending $x_1$ to $\lambda
  x_1$, $\phi$ must have the form $x_1 \mapsto c \cdot x_1^n$ for some
  $c \in \bk$.  It remains to note that, if $f,g \in \hat \cO_{X,x}$
  are two functions with nonvanishing first derivative, again by the
  Darboux theorem there is an element of $G_{X,x}$ sending $f$ to
  $g$, so the constant $c$ must be independent of the choice of $f$.
  Therefore, $\bar \phi(g)=c g^n$ for all $g$. We can easily see that
  this is $\mfv$-invariant.

  (ii) Restricting to $\hat X_x$, suppose first
  that $\mfv$ is arbitrary such that, in some coordinate system, it
  contains the constant vector fields and an Euler vector $\Eu =
  \sum_i m_i \partial_i$ for $m_i > 0$.  Let $m := \sum_i m_i$.  Let
  $\vol$ be the standard volume form in this coordinate system. The
  polydifferential operators $\hat \cO_{X,x}^{\otimes n} \to
  \Omega_{\hat X_x}$ invariant under the aforementioned vector fields
  are spanned by
\[
(F_1 \otimes \cdots \otimes F_n) \cdot \vol, \quad |F_1|+\cdots+|F_n| = -m,
\]
where each $F_i$ is a constant-coefficient monomial in the $\partial_i$, and
here $|\cdot|$ denotes the weighted degree with respect to $\Eu$.  This is a finite-dimensional vector space. 

Now, in the case where $\mfv = \Vect(X)$, in order to be invariant
under all possible Euler vector fields, the operator must be a linear
combination of terms such that $F_1 \cdots F_n$ is linear in each
coordinate.  Moreover, to be invariant under volume-preserving linear
changes of basis, i.e., under $\SL(T_x X)$, we conclude that the
operator is spanned by images under $S_n$ of $(\vol^{-1}
\otimes 1^{\otimes (n-\dim X)}) \cdot \vol$, as desired.

  In the case $\mfv = H(X)$ and $X$ is odd-dimensional contact variety,
  then we can take $\Eu$ as in Example \ref{ex:st-cont}, so that $F_1 \cdots F_n$ must have total degree $-(\dim X + 1)$ (since $|x_i|=|y_i|=1$ and $|t|=2$,
 and the partial derivatives have negative this degree).    Also, the
polydifferential operator must be preserved by
all linear changes of basis preserving $\partial_t$.  In particular, since
it is preserved by
 $\GL(\langle \partial_{x_i}, \partial_{y_i} \rangle)$, the operator must
be in the $\bk[S_n]$-span of
\[
\bigl(
\wedge^{\dim X - 1} \langle \partial_{x_i}, \partial_{y_i} \rangle \otimes
\partial_t \otimes 1^{\otimes (n-\dim X)} \bigr) \cdot \vol.
\]
Since it is preserved by transformations $x_i \mapsto x_i + \lambda t$ for
$\lambda \in \bk$, we conclude in fact that it is in the $\bk[S_n]$-span of
$(\vol^{-1} \otimes 1^{\otimes (n-\dim X)}) \cdot \vol$, as desired.
\end{proof}
\subsection{General decomposition statement}
First, we generalize Theorem \ref{t:sym-ql}  by replacing (*) by
a general decomposition statement about $M(S^n X, \mfv)$. Then (*) becomes
a multiplicity-one condition.
\begin{definition}
  Given a smooth affine variety $X$ and an integer $m \geq 1$, let
  $\PDiff(\cO_X,\Omega_X, m+1)$ be the space of polydifferential
  operators $\cO_X^{\otimes m} \to \Omega_X$ of degree $m$, i.e.,
  linear maps which are differential operators in each component.
\end{definition}
Note that there there is a natural action of $S_{m+1}$ on
$\PDiff(\cO_X, \Omega_X, m+1)$ given by viewing these operators as
distributions along on the diagonal in $X^{m+1}$, i.e., as sections of the
$\caD_{X^{m+1}}$-module $(\Delta_{m+1})_* \Omega_X$, which has its
natural $S_{m+1}$-action. The $S_m$ action is just by permutation of
components, and the extension to $S_{m+1}$ is explicitly given by the
integration by parts rule.  For example, when $X = \bA^1$ with the
standard volume, this action restricted to the span of partial
derivatives $\partial_1, \ldots, \partial_m$ is the reflection
representation of $S_{m+1}$ (viewed as a type $A_m$ Weyl group);
explicitly this can be viewed as the usual permutation action on
$\partial_1, \ldots, \partial_{m+1}$ where we set $\partial_{m+1} =
-\sum_{i=1}^m \partial_i$.

For all $m \geq 1$, let $L_m$ be the maximal quotient of $M(S^m X,
\mfv)$ supported on the diagonal, i.e., $L_m = (\pr_m \circ \Delta_m)_*
H^0 (\pr_m \circ \Delta_m)^* M(S^m X, \mfv)$ (which at least makes sense when $M(S^m
X, \mfv)$ is holonomic, as in the quasi-local transitive case).
\begin{theorem}\label{t:sym-ql2}
Suppose that $(X,\mfv)$ is quasi-local and transitive and has pure
dimension at least two.
Then, there is a canonical isomorphism
\begin{equation}\label{e:msn-dec1}
M(S^n X, \mfv) \iso \bigoplus_{\lambda \vdash n} (\pr_n)_*  (\Delta_\lambda)_*
L_\lambda^{S_\lambda}, \quad L_\lambda := L_{\lambda_1} \boxtimes \cdots \boxtimes
L_{\lambda_{|\lambda|}}.
\end{equation}
Moreover, the rank of $L_m$ is equal to the dimension of $(\PDiff(\hat
\cO_{X,x}, \Omega_{\hat X_x},m)^{\mfv})^{S_m}$.
\end{theorem}
The canonical isomorphism is given by the direct sum of
the morphisms 
\begin{equation}\label{e:msn-dec1-fact}
M(S^n X, \mfv)
\onto (\pr_n)_*  (\Delta_\lambda)_*
L_\lambda^{S_\lambda},
\end{equation} 
obtained by adjunction from the canonical quotients $H^0 (\pr_n \circ
\Delta_\lambda)^* M(S^n X, \mfv) \onto L_\lambda$.

The theorem implies that composition factors from distinct leaves do
not appear in nontrivial extensions:
\begin{corollary}\label{c:d-ie}
  In the situation of the theorem, $M(S^n X, \mfv)$ is a direct sum of
  intermediate extensions of local systems on the leaves (locally
  closed diagonals).
\end{corollary}
\begin{proof}
  For each diagonal $X_\lambda := \pr_n \circ
  \Delta_\lambda(X^{|\lambda|})$, let $j_\lambda: X_\lambda^\circ
  \into X_\lambda$ be the open embedding of the complement of smaller
  diagonals, i.e., such that $X_\lambda^\circ$ is a leaf of $S^n X$.
  Let $\tilde j_\lambda: \tilde X_\lambda^\circ \into
  \Delta_\lambda(X) \subseteq X^n$ be the preimage of
  $X_\lambda^\circ$.  Then, for each factor in \eqref{e:msn-dec1},
\[
j_\lambda^* (\pr_n)_*  (\Delta_\lambda)_*
L_\lambda^{S_\lambda} \cong (\pr_n)_* \tilde j_\lambda^* (\Delta_\lambda)_*
L_\lambda^{S_\lambda}.
\]
Since $\pr_n$ is a covering of $\tilde X_\lambda^\circ$ onto its image (with
covering group $S_\lambda$), the above is a local system on
$X_\lambda^\circ$.  It now suffices to prove that
\begin{equation}
 (\pr_n)_*  (\Delta_\lambda)_*
L_\lambda^{S_\lambda} \cong j_{!*}j_\lambda^* (\pr_n)_*  (\Delta_\lambda)_*
L_\lambda^{S_\lambda}.
\end{equation}
This follows because, since $\pr_n$ is finite, the singular support of
$(\pr_n)_* (\Delta_\lambda)_* L_\lambda^{S_\lambda}$ is the closure of
the conormal bundle of the leaf $\pr_n \circ
\Delta_\lambda((X^{|\lambda|})^\circ)$, where $(X^{|\lambda|})^\circ$
is the complement in $X^{|\lambda|}$ of the images of all diagonal
embeddings of $X^r$ for all $r < |\lambda|$.
\end{proof}
We can make a similar statement about $M(X^n, \mfv)$ itself: Let
$\tilde L_m = (\Delta_m)_* H^0 \Delta_m^* M(X^m, \mfv)$ be the maximal
quotient of $M(X^m, \mfv)$ supported on the diagonal. This is
$S_m$-equivariant, and $(\tilde L_m)^{S_m} = L_m$.
\begin{theorem}\label{t:sym-ql3} Let $(X,\mfv)$ be as in Theorem \ref{t:sym-ql2}. 
  Then, there is a canonical isomorphism
\begin{equation}\label{e:msn-dec2}
M(X^n, \mfv) \iso \bigoplus_{\lambda \vdash n} S_n
(\widetilde{L_\lambda}), \quad
\widetilde{L_\lambda} := \widetilde{L_{\lambda_1}} \boxtimes \cdots
\boxtimes \widetilde{L_{\lambda_{|\lambda|}}}.
\end{equation}
Here, $S_n (\widetilde{L_\lambda})$ is the $S_n$-equivariant local
system on the $S_n$-orbit of $\Delta_\lambda(X)$ whose restriction to
$\Delta_\lambda(X)$ is the $N_{S_n}(S_\lambda)/S_\lambda$-equivariant
local system $\tilde L_m$.

Moreover, the rank of $\tilde L_m$ is the dimension of $\PDiff(\hat
\cO_{X,x},\Omega_{\hat X_x},m)^{\mfv}$.
\end{theorem}
As in Corollary \ref{c:d-ie}, it follows from this that the entire
pushforward $(\pr_n)_* M(X^n, \mfv)$ on $S^n X$ is a direct sum of
intermediate extensions of $S_\lambda$-equivariant local systems on
the diagonals corresponding to partitions $\lambda \vdash n$.

\subsection{Proof of Theorems \ref{t:sym-ql2} and \ref{t:sym-ql3}}
We will work with $M(X^n, \mfv)$. Since this is $S_n$-equivariant and
$M(S^n X, \mfv) = (\pr_n)_* M(X^n, \mfv)^{S_n}$, this will also
compute the latter.

By transitivity and quasi-locality, the closures of the leaves of
$M(X^n, \mfv)$ are the diagonals $\Delta_\lambda(X^{|\lambda|})$
together with the diagonals obtained from these by the action of
$S_n$.  Hence, $M(X^n, \mfv)$ is holonomic and its composition factors
are intermediate extensions of local systems on these leaves.  Similarly
to \eqref{e:msn-dec1-fact}, one has canonical surjections
\begin{equation}\label{e:msn-dec2-fact}
M(X^n, \mfv)
\onto (\Delta_\lambda)_*
L_\lambda,
\end{equation}
and similarly for the orbits of these under $S_n$ (there is one of
these for each coset in $S_n / (N_{S_n}(S_\lambda))$, and each is a
local system on the image of $\Delta_\lambda(X^{|\lambda|})$ under the
element of $S_n$ which is equivariant under the corresponding
conjugate of the subgroup $N_{S_n}(S_\lambda) < S_n$).  It suffices to
prove the following:
\begin{enumerate}
\item[(i)] The quotient \eqref{e:msn-dec2-fact} is the maximal
  quotient supported on $\Delta_\lambda(X^{|\lambda|})$, i.e., it is
  $(\Delta_\lambda)_* H^0 (\Delta_\lambda)^* M(X^n, \mfv)$;
\item[(ii)] For distinct $\lambda$ or distinct orbits for a fixed
  $\lambda$, that the above factors have no nontrivial extensions
  (i.e., the $\Ext$ group of the two is zero). 
\end{enumerate}
For (i), by restricting to a formal neighborhood of a generic point 
$y = \Delta_\lambda(x)$ of $\Delta_\lambda(X^{|\lambda|})$, it suffices to find
an isomorphism
\[
\Hom_{\hat \caD_{X^n,y}}(M(X^n,\mfv)|_{\widehat{X^n}_y},
(\Delta_\lambda)_* \Omega_{\widehat{X^{|\lambda|}}_x}) \cong
\bigotimes_{i=1}^{|\lambda|} \PDiff(\hat \cO_{X,x},\Omega_{\hat X_x}, \lambda_i)^{\mfv}.
\]
By quasi-locality, it suffices to restrict to the case $|\lambda|=1$ (for all $n$).  For this, note that there is a canonical isomorphism
\[(\Delta_n)_* \Omega_{\widehat{X,x}} \cong\PDiff(\hat \cO_{X,x},n).
\]
Moreover, for any $\caD_{X^n}$-module $N$, we have
a canonical isomorphism $\Hom(M(X^n, \mfv), N) \cong N^\mfv$, by considering the image of the canonical generator of $M(X^n, \mfv)$.   Putting these together, we deduce part (i).

For (ii), note that the factors $(\Delta_\lambda)_* L_\lambda$, as
well as their images under the action of $S_n$, are local systems on
smooth closed subvarieties of $X^n$.  Moreover, the intersection of
two of these subvarieties has codimension a multiple of $\dim X$ in
each, which in particular is codimension at least $2$.  Thus, the
result follows from the following basic lemma:
\begin{lemma} \label{l:extsubvars} \cite[Lemma 2.1.1]{hp0weyl} 
Suppose that $Z$ is a smooth
  variety, and $Z_1, Z_2 \subseteq Z$ as well as $Z_1 \cap Z_2$ are
  smooth closed subvarieties, all of pure dimension.  Let
  $\mathcal{L}_1, \mathcal{L}_2$ be local systems on $Z_1$ and $Z_2$,
  respectively, and let $i_1: Z_1 \to Z$ and $i_2: Z_2 \to Z$ be the
  inclusions.  Then,
\begin{equation}\label{e:extsubvars}
  \Ext^j((i_1)_* \mathcal{L}_1, (i_2)_* \mathcal{L}_2) = 0, \text{ for } j < 
  (\dim Z_1 - \dim Z_1 \cap Z_2) + (\dim Z_2 - \dim Z_1 \cap Z_2).
\end{equation}
\end{lemma}

\subsection{Proof of Theorem \ref{t:sym-ql}}
(i) If $\mfv$ flows incompressibly, then we have an
isomorphism of modules over the Lie algebra $\mfv$, $\hat \cO_{X,x}
\iso \Omega_{\hat X_x}$, obtained from the formal volume at $x$
preserved by $\mfv$.  Therefore, in the theorem, we can replace the
polydifferential operators described by $(\PDiff(\hat \cO_{X,x},
\Omega_{\hat X_x},n+1)^\mfv)^{S_n}$.  Then, the result is almost
immediate from Theorem \ref{t:sym-ql2}, except Theorem \ref{t:sym-ql}
deals with $S_n$-invariant polydifferential operators, whereas the
multiplicity spaces of Theorem \ref{t:sym-ql2} are more symmetric:
they are $(\PDiff(\hat \cO_{X,x}, \Omega_{\hat X_x},n+1)^\mfv)^{S_{n+1}}$.

Thus, it suffices to show that, if such $\mfv$-invariant
$S_{n+1}$-invariant polydifferential operators of degree $n$ are
spanned by the multiplication operator  \emph{for all n},
then the same is true requiring only $S_n$-invariance.

  For this, note that, given a $\mfv$-invariant polydifferential
  operator $\phi$ on $\hat \cO_{X,x}$ of degree $n$, then the space of
  $\mfv$-invariant polydifferential operators of degree $n+1$ includes
  the space $\Ind_{S_n \times S_1}^{S_{n+1}} \langle \phi \boxtimes
  \bk \rangle$ spanned over $S_{n+1}$ by the operator $f \boxtimes g
  \mapsto \phi(f) \cdot g$ for all $f \in \hat \cO_{X,x}^{\otimes n}$
  and $g \in \hat \cO_{X,x}$. 

To proceed, we will need the following technical combinatorial result, which we prove below:
\begin{lemma}\label{l:tech-sn} Suppose that $\phi$ generates a
  $S_{n+1}$-representation $V$, and that $\phi$ is $S_n$-invariant but not
  $S_{n+1}$-invariant.  Then the $S_{n+1}$-representation $\Ind_{S_n
    \times S_1}^{S_{n+1}} (V|_{S_n} \boxtimes \bk)$ extends to a
  unique $S_{n+2}$-representation, up to isomorphism, and this has a
  nonzero $S_{n+2}$-invariant vector.
\end{lemma}
Let us use the lemma to finish the proof of the first statement.  
We conclude from the lemma
that there exists a $S_{n+2}$-invariant, $\mfv$-invariant
polydifferential operator $\phi$ on $\hat \cO_{X,x}$ of degree
$n+1$. We claim that this is not the multiplication operator (up to
scaling).  Indeed, we could have assumed that $\phi$ were homogeneous
of positive order (since $\mfv$ preserves the grading by order of
differential operators), so the latter $S_{n+2}$-operator can be
assumed to have positive order.  This contradicts our hypothesis.  Hence,
(*) of Theorem \ref{t:sym-ql} is indeed satisfied.

(ii) If $\mfv$ does not flow incompressibly, $M(X,
\mfv) = 0$, by Proposition \ref{p:tr}. Next, suppose that there
existed an $S_n$-invariant, $\mfv$-invariant polydifferential operator $\phi: \hat
\cO_{X,x}^{\otimes n} \to \Omega_{\hat X_x}$ but not a
$S_{n+1}$-invariant one.  Again, we can form the polydifferential
operator $(\phi \boxtimes 1)$, sending $f_1 \otimes \cdots \otimes
f_{n+1}$ to $\phi(f_1 \otimes \cdots \otimes f_n) f_{n+1}$.  So as
before, we would obtain that, as $S_{n+1}$-representations,
$\PDiff(\hat \cO_{X,x}, \Omega_{\hat X_x}, n+2)^{\mfv} \supseteq
\Ind_{S_n \times S_1}^{S_{n+1}} (\bk^n \boxtimes \bk)$.  By the same
argument as above, this would contain an $S_{n+2}$-invariant operator.  Thus,
$M(S^{n+2} X, \mfv) \neq 0$.  So, if $M(S^n X, \mfv) = 0$ for all $n \geq 1$,
then there are no $S_n$-invariant, $\mfv$-invariant polydifferential operators $\phi: \hat
\cO_{X,x}^{\otimes n} \to \Omega_{\hat X_x}$, for all $n$.  The converse is clear
from Theorem \ref{t:sym-ql2}.
\begin{proof}[Proof of Lemma \ref{l:tech-sn}]
  Under the assumption, $V$ must include a summand isomorphic to the
  reflection representation $\bk^n$ ($V$ is either this or the direct
  sum of $\bk^n$
  with a trivial representation). As a representation of $S_n$,
  $\bk^n$ is the standard representation.

Thus, we can assume that $V = \bk^n$.  As an $S_n$-representation,
  $V \cong \bk^{n-1} \oplus \bk$. Then, for $n \geq 3$, one computes the
  decomposition into irreducible $S_{n+1}$-representations:
\[
\Ind_{S_n \times S_1}^{S_{n+1}} V|_{S_n} \boxtimes \bk 
\cong \rho_{(1,1)[n+1]} \oplus \rho_{(2)[n+1]} \oplus \bk^n \oplus
\bk^n \oplus \bk,
\]
where, given a partition $\lambda \vdash (n+1)$, the representation
$\rho_\lambda$ is the irreducible representation with Young diagram
$\lambda$. Moreover, given $\lambda' \vdash m$, we let $\lambda'[n+1]$
denotes the diagram obtained from $\lambda'$ by adding a new row on
top with $n+1-m$ boxes (which makes sense if $n+1-m \geq \lambda'_1$.)
Now, if the $S_{n+1}$-structure above extends to a
$S_{n+2}$-structure, then the decomposition into irreducible
$S_{n+2}$-representations (up to isomorphism) must be
\[
\rho_{(1,1)[n+2]} \oplus \rho_{(2)[n+2]} \oplus \bk.
\]
We conclude that $\Ind_{S_n \times S_1}^{S_{n+1}} \langle \phi
\boxtimes \bk\rangle$ must contain a $S_{n+2}$-fixed vector.

In the case that $n=2$, the second decomposition (as
$S_{n+2}=S_4$-representations) above is still valid, so we still
obtain the $S_{n+2}$-fixed vector.
\end{proof}




\bibliographystyle{amsalpha}
\bibliography{master}
\end{document}